\documentclass{article}

\usepackage{microtype}
\usepackage{graphicx}
\usepackage{subfigure}
\usepackage{booktabs}

\usepackage[accepted]{icml2023}
\usepackage[bookmarks=true,bookmarksopen=true,pdfencoding=auto,psdextra,unicode,colorlinks]{hyperref}
\definecolor{mydarkblue}{rgb}{0,0.08,0.45}
\hypersetup{ %
    pdftitle={},
    pdfsubject={Proceedings of the International Conference on Machine Learning 2023},
    pdfkeywords={},
    pdfborder=0 0 0,
    pdfpagemode=UseOutlines,
    bookmarksnumbered=true,     
    bookmarksopen=true,
    bookmarksopenlevel=1,
    colorlinks=true,
    linkcolor=mydarkblue,
    citecolor=mydarkblue,
    filecolor=mydarkblue,
    urlcolor=mydarkblue,
}
\usepackage{bookmark}

\usepackage{amsmath}
\usepackage{amssymb}
\usepackage{mathtools}
\usepackage{amsthm}

\usepackage{tikz}
\tikzset{
  block/.style = {draw, fill=white, rectangle, minimum height=3em, minimum width=5em},
  arrow/.style = {thick,->,>=stealth}
}
\usetikzlibrary{positioning, shapes, arrows.meta}

\usepackage[capitalize,noabbrev]{cleveref}
\usepackage{crossreftools}
\usepackage{blindtext}

\usepackage{etoc}

\theoremstyle{plain}
\newtheorem{theorem}{Theorem}[section]
\newtheorem{proposition}[theorem]{Proposition}
\newtheorem{lemma}[theorem]{Lemma}
\newtheorem{corollary}[theorem]{Corollary}
\theoremstyle{definition}
\newtheorem{definition}[theorem]{Definition}
\newtheorem{assumption}[theorem]{Assumption}
\theoremstyle{remark}

\usepackage[textsize=tiny]{todonotes}

\usepackage{tabularx}
\newcolumntype{Y}{>{\centering\arraybackslash}X}
\newcolumntype{P}{>{\raggedleft\arraybackslash}X}
\usepackage{multirow}

\crefname{assumption}{assumption}{assumptions}
\allowdisplaybreaks
\usepackage[shortlabels]{enumitem}
\usepackage{dsfont}
\usepackage{mathrsfs}
\RequirePackage{bm}

\def\sR{\mathscr{R}}
\def\sG{\mathscr{G}}
\def\sM{\mathscr{M}}
\def\sN{\mathscr{N}}
\def\sD{\mathscr{D}}
\def\sT{\mathscr{T}}
\def\sE{\mathscr{E}}

\newcommand{\rd}{{\,\mathrm{d}}}

\newcommand{\asto}{\xrightarrow{\textup{a.s.}}}
\newcommand{\as}{\textup{a.s.}}
\newcommand{\pto}{\xrightarrow{\textup{p}}}
\newcommand{\dto}{\xrightarrow{\textup{d}}}

\newcommand{\norm}[1]{\left\lVert#1\right\rVert}
\let\hat\widehat
\let\tilde\widetilde

\def\given{\,\middle|\,}

\newcommand{\ba}{\bm{a}}
\newcommand{\bb}{\bm{b}}
\newcommand{\bc}{\bm{c}}

\newcommand{\bw}{\bm{w}}
\newcommand{\bx}{\bm{x}}
\newcommand{\by}{\bm{y}}
\newcommand{\bz}{\bm{z}}

\newcommand{\bA}{\bm{A}}
\newcommand{\bB}{\bm{B}}
\newcommand{\bC}{\bm{C}}
\newcommand{\bD}{\bm{D}}
\newcommand{\bE}{\bm{E}}

\newcommand{\bI}{\bm{I}}

\newcommand{\bL}{\bm{L}}
\newcommand{\bM}{\bm{M}}
\newcommand{\bN}{\bm{N}}

\newcommand{\bQ}{\bm{Q}}
\newcommand{\bR}{\bm{R}}
\newcommand{\bS}{\bm{S}}

\newcommand{\bU}{\bm{U}}
\newcommand{\bV}{\bm{V}}
\newcommand{\bW}{\bm{W}}
\newcommand{\bX}{\bm{X}}
\newcommand{\bY}{\bm{Y}}
\newcommand{\bZ}{\bm{Z}}

\newcommand{\cC}{\mathcal{C}}
\newcommand{\cD}{\mathcal{D}}

\newcommand{\cF}{\mathcal{F}}

\newcommand{\cI}{\mathcal{I}}

\newcommand{\cK}{\mathcal{K}}

\newcommand{\cN}{\mathcal{N}}

\newcommand{\CC}{\mathbb{C}}
\newcommand{\EE}{\mathbb{E}}

\newcommand{\NN}{\mathbb{N}}
\newcommand{\PP}{\mathbb{P}}

\newcommand{\RR}{\mathbb{R}}

\newcommand{\bbeta}{\bm{\beta}}

\newcommand{\bepsilon}{\bm{\epsilon}}

\newcommand{\bSigma}{\bm{\Sigma}}

\newcommand{\argmin}{\mathop{\mathrm{argmin}}}

\newcommand{\tr}{\mathop{\mathrm{tr}}}

\DeclareMathOperator{\Var}{{\rm Var}}

\DeclareMathOperator{\ind}{\mathds{1}}  %

\newcommand*{\zero}{{\bm 0}}

\newcommand{\oper}{\mathop{\mathrm{op}}}
\usepackage{xspace}

\newcommand{\SNR}{\texttt{SNR}\xspace}

\newcommand{\betaridge}{\hat{\bbeta}^{\lambda}\xspace}

\newcommand{\hSigma}{\hat{\bSigma}\xspace}
\newcommand{\hf}{{\hat{f}}}
\newcommand{\tf}{{\widetilde{f}}}
\newcommand{\tbeta}{{\widetilde{\bbeta}}}
\newcommand{\tfWR}[2]{{\tf_{{#1},{#2}}}}

\def \hbeta {\widehat{\bbeta}}

\newcommand{\tv}{\widetilde{v}}
\newcommand{\tc}{\widetilde{c}}

\newcommand{\RzeroMe}[3]{{\sR^0_{{#1}}({#2},{#3})}}
\newcommand{\RlamMe}[3]{{\sR_{{#1}}^{\lambda}({#2},{#3})}}
\newcommand{\RzeroM}[2]{{\sR_{{#1}}^{0}({#2},\phi_s)}}

\usepackage{xparse}
\NewDocumentCommand{\RlamM}{ O{\lambda} O{M} O{\phi} O{\phi_s}}{{\sR_{{#2}}^{{#1}}({#3},{#4})}}
\NewDocumentCommand{\RlamMtr}{ O{\lambda} O{M} O{\phi} O{\phi_s}}{{\sT^{{#1}}_{{#2}}({#3},{#4})}}

\newcommand{\BlamM}[2]{{\mathscr{B}_{{#1}}^{\lambda}({#2},\phi_s)}}
\newcommand{\VlamM}[2]{{\mathscr{V}_{{#1}}^{\lambda}({#2},\phi_s)}}

\newcommand{\rhoar}{\rho_{\mathrm{AR1}}}

\newcommand{\asympequi}{\simeq}

\usepackage{xcolor}
\usepackage{colortbl}

\newcommand{\SRS}{\textup{\texttt{SRS}}}

\newcommand{\Errtrain}{\mathrm{Err}_{\mathrm{train}}}
\newcommand{\Errtest}{\mathrm{Err}_{\mathrm{test}}}

\newcommand{\gcv}{\textup{gcv}}

\NewDocumentCommand{\Rdet}{ O{\lambda} O{M} O{\phi} O{\phi_s}}{{\sR_{{#2}}^{{#1}}({#3},{#4})}}
\NewDocumentCommand{\gcvdet}{ O{\lambda} O{\phi} O{\phi_s} O{\infty}}{{\sG_{#4}^{#1}({#2},{#3})}}
\NewDocumentCommand{\Mdet}{ O{\lambda} O{M} O{\phi} O{\phi_s}}{{\sM^{{#1}}_{{#2}}({#3},{#4})}}
\NewDocumentCommand{\Ndet}{ O{\lambda} O{M} O{\phi} O{\phi_s}}{{\sN^{{#1}}_{{#2}}({#3},{#4})}}
\NewDocumentCommand{\Ddet}{ O{\infty} O{\lambda} O{\phi} O{\phi_s}}{{\sD_{{#1}}^{{#2}}({#3},{#4})}}

\newcommand\ddfrac[2]{\tfrac{\displaystyle #1}{\displaystyle #2}}

\newcommand{\titletext}{Subsample Ridge Ensembles: Equivalences and Generalized Cross-Validation}

\newcommand{\titletextno}{Subsample Ridge Ensembles: Equivalences and Generalized Cross-Validation}

\icmltitlerunning{\titletextno}

\begin{document}

\twocolumn[
\icmltitle{\titletext}

\icmlsetsymbol{equal}{*}

\begin{icmlauthorlist}
\icmlauthor{Jin-Hong Du}{equal,cmustats}
\icmlauthor{Pratik Patil}{equal,berkeleystats}
\icmlauthor{Arun Kumar Kuchibhotla}{cmustats}
\end{icmlauthorlist}

\icmlaffiliation{cmustats}{Department of Statistics and Data Science, Carnegie Mellon University, Pittsburgh, PA 15213, USA.}
\icmlaffiliation{berkeleystats}{Department of Statistics, University of California, Berkeley, CA 94720, USA}

\icmlcorrespondingauthor{Jin-Hong}{jinhongd@andrew.cmu.edu}
\icmlcorrespondingauthor{Pratik}{pratikpatil@berkeley.edu}

\icmlkeywords{Machine Learning, ICML}

\vskip 0.3in
]

\printAffiliationsAndNotice{\icmlEqualContribution}  %

\begin{abstract}
We study subsampling-based ridge ensembles in the proportional asymptotics regime, where the feature size grows proportionally with the sample size
such that their ratio converges to a constant.
By analyzing the squared prediction risk of ridge ensembles as a function of the explicit penalty $\lambda$ and the limiting subsample aspect ratio $\phi_s$ (the ratio of the feature size to the subsample size), we characterize contours in the $(\lambda, \phi_s)$-plane at any achievable risk.
As a consequence, we prove that the risk of the optimal full ridgeless ensemble (fitted on all possible subsamples) matches that of the optimal ridge predictor.
In addition, we prove strong uniform consistency of generalized cross-validation (GCV) over the subsample sizes for estimating the prediction risk of ridge ensembles. 
This allows for GCV-based tuning of full ridgeless ensembles without sample splitting and yields a predictor whose risk matches optimal ridge risk.
\end{abstract}

\section{Introduction}\label{sec:introduction}

    Ensemble methods \citep{breiman_1996} are widely used in various real-world applications in statistics and machine learning.
    They combine a collection of weak predictors to produce more stable and accurate predictions.
    One notable example of an ensemble method is bagging (\textbf{b}ootstrap \textbf{agg}regat\textbf{ing}) \citep{breiman_1996,buhlmann2002analyzing}. Bagging involves averaging base predictors that are fitted on different subsampled datasets and has been shown to stabilize the prediction and reduce the predictive variance \citep{buhlmann2002analyzing}.
    In this paper, we study such a class of ensemble methods that fit each base predictor independently using a different subsampled dataset of the full training data.
    As a prototypical base predictor, we focus on \emph{ridge regression} \citep{hoerl_kennard_1970_1,hoerl_kennard_1970_2}, 
    one of the most popular statistical methods.
    We refer readers to the ``ridgefest'' by \citet{hastie2020ridge} for the history and review of ridge regression.
    
    \begin{figure}[!t]
        \centering
        \includegraphics[width=0.45\textwidth]{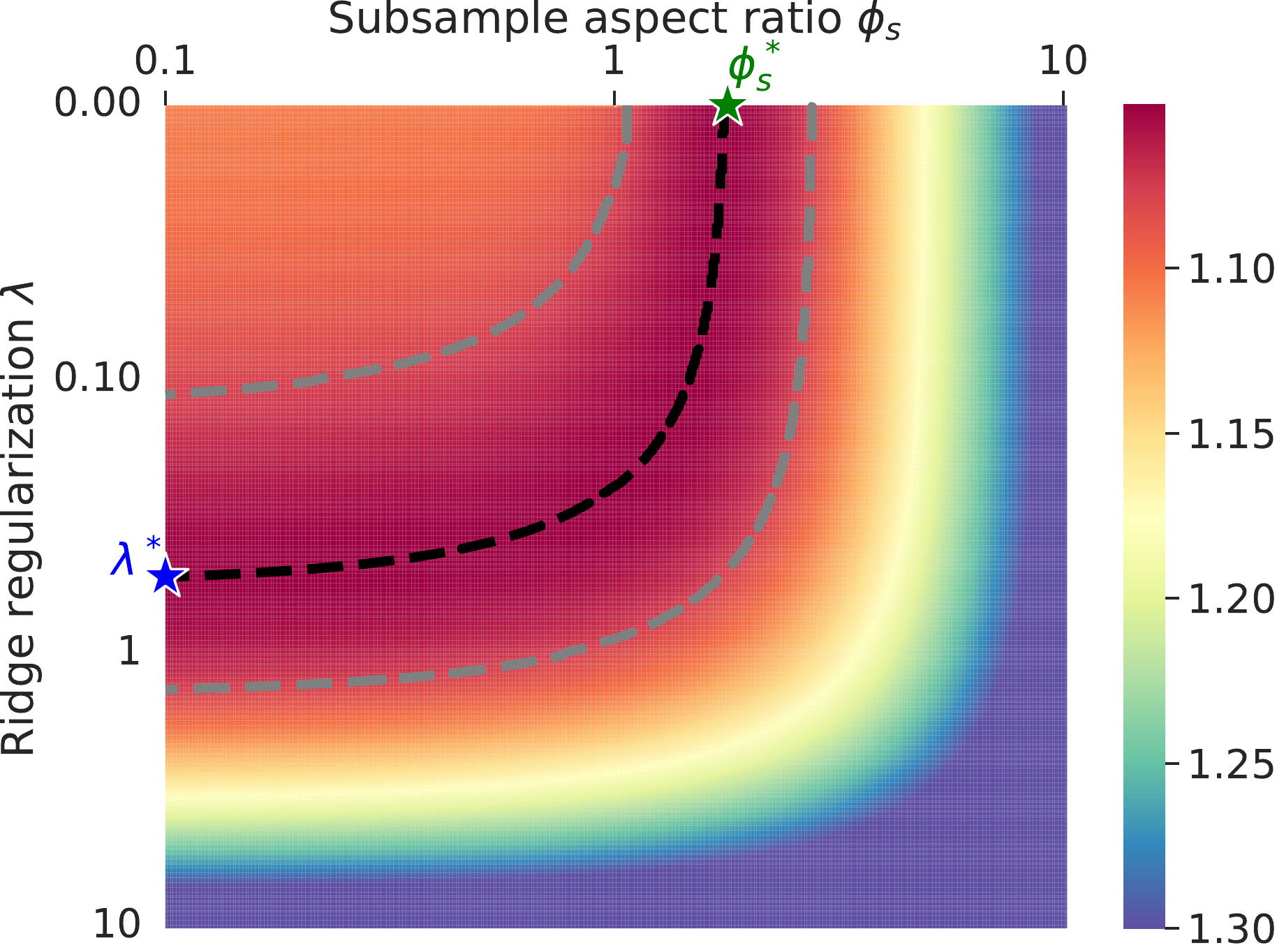}
        \caption{Heat map of the asymptotic prediction risk landscape of full ridge ensembles as the number of observation $n$, the subsample size $k$, and the feature dimension $p$ tend to infinity, for varying regularization parameters $\lambda$ and limiting subsample aspect ratio $\phi_s=\lim p/k$.
        The data $(\bx,y)\in \RR^{p}\times\RR$ is generated from a non-isotropic linear model $y = \bx^{\top}\bbeta_0 + \epsilon$ with $\phi=\lim p/n=0.1$, where the features, the coefficients, and the residuals are distributed as $\bx\sim\cN(0,\bSigma_{\mathrm{AR1}})$, $\bbeta_0=\frac{1}{5}\sum_{j=1}^5 \bw_{(j)}$, and $\epsilon\sim \cN(0,1)$, respectively.
        Here, the covariance matrix $(\bSigma_{\mathrm{AR1}})_{ij} = 0.5^{|i-j|}$, $\bw_{(j)}$ is the top $j$th eigenvector of $\bSigma_{\mathrm{AR1}}$.
        The green and blue stars denote the risk of the optimal full-ensemble ridgeless predictor and the optimal ridge predictor without subsampling, respectively.
        The black dashed line denotes the set of $(\lambda,\phi_s)$ pairs that yield the same risk as $(\lambda^*,\phi)$ and $(0,\phi_s^*)$, while the gray dashed lines indicate the set of pairs that all result in the same sub-optimal risk.}
        \label{fig:overview}
    \end{figure}

    Ridge regression has recently attracted great interest, particularly the limiting case of zero regularization (where the regularization parameter tends to zero), termed ``ridgeless'' regression.
    In the underparameterized regime,
    the ridgeless predictor is ordinary least squares.
    However, in the overparameterized regime, 
    it interpolates the training data and exhibits a peculiar risk behavior \citep{belkin_hsu_xu_2020,bartlett_long_lugosi_tsigler_2020,hastie2022surprises,muthukumar_vodrahalli_subramanian_sahai_2020}.
    \citet{lejeune2020implicit,patil2022bagging} have recently analyzed the statistical properties of the ensemble ridge and ridgeless predictors under proportional asymptotics.
    Under a linear model with the isotropic Gaussian covariate distribution, \citet{lejeune2020implicit} prove that the \emph{full ensemble} (ensemble fitted on all possible subsampled datasets) of least squares predictors with optimal subsample size has the same risk as that of ridge predictor with optimal regularization. 
    Under a more general but still isotropic covariate distribution, \citet{patil2022bagging} prove similar risk equivalence of the optimized full ridgeless ensemble and the optimized ridge predictor.
    
    These findings inspire two natural avenues to investigate.
    
    \emph{(1) Understanding the extent of risk equivalences.}
    As a curious experiment, one can empirically observe that a similar phenomenon to the one just mentioned appears to hold under quite general non-isotopic data models, as illustrated in \Cref{fig:overview}.
    We observe that the optimal ridgeless in the full ensemble (the green star) has the same prediction risk as the optimal ridge on the full data (the blue star).
    Furthermore, any pair of $(\lambda, \phi_s)$ on the black line achieves the same optimal risk.
    Such a relationship also extends to any other attainable risk value.
    For example, see the grey lines for $(\lambda, \phi_s)$ pairs that all achieve the same sub-optimal risk.
    This inspires our first investigation to establish risk equivalences between subsampling and ridge regression under general settings.
    
    \emph{(2) Overcoming limitations of split cross-validation.}
    Apart from its theoretical interest, the risk equivalences also suggest an alternative practical way to tune the ridge regularization parameter by tuning the subsample size.
    In terms of practical tuning of the ridge and ridgeless ensembles, \citet{patil2022bagging} provide a split cross-validation method to estimate the prediction risk of ensembles with a fixed (finite) number of ensemble sizes and further prove that the split cross-validation consistently selects the best subsample size. The split cross-validation procedure has two disadvantages: (a) sample splitting introduces additional external randomness in the predictor; and (b) the reduced sample size, although asymptotically negligible, has significant finite sample effects, especially near the interpolation thresholds.
    This inspires our second investigation to address these limitations by considering generalized cross-validation (GCV) that does not require any sample splitting.
    The consideration of GCV as a viable estimator of the prediction risk for ridge ensembles stems from the observation that the ridge ensembles are also in fact linear smoothers.

    \subsection{Summary of Contributions} 
    
    Below we provide a brief overview of our main results.
    
    \begin{itemize}[leftmargin=*]
        \item
        \textbf{General risk equivalences.}
        We establish general equivalences between the subsample-optimized ridgeless ensemble, the optimal ridge predictor, and the optimal subsample ridge ensemble (see \Cref{thm:comparison_optimal_ridge}).
        In addition, for any $\tau\geq 0$, we provide an exact characterization of the sets $\cC_{\tau}$ of pairs $(\lambda,\phi_s)$ (the regularization parameter and the limiting subsample aspect ratio) such that the risk of the full ridge ensemble with ridge regularization $\lambda$ and subsample aspect ratio $\phi_s$ is equal to the risk of the ridge predictor with ridge regularization $\tau$.
        In essence, this amounts to showing that the implicit regularization of subsampling is the same as additional explicit ridge regularization.

        \item
        \textbf{Uniform consistency of GCV.}
        We establish the uniform consistency of GCV across all possible subsample sizes for full ridge ensembles with fixed regularization parameters (see \Cref{thm:uniform-consistency-k}).
        Notably, this result is also applicable to zero explicit regularization and covers the case of ridgeless regression.
        This finding enables tuning over the subsample size in a data-dependent manner, and in conjunction with \Cref{thm:comparison_optimal_ridge}, 
        it implies that GCV tuning leads to a predictor with the same risk as the optimal ridge predictor (see \Cref{cor:gcv-opt-ridge}).

        \item 
        \textbf{Finite-ensemble surprises.}
        Even though GCV is consistent for the non-ensemble ridge and full-ensemble ridge predictors, interestingly, this is the first paper that proves GCV \emph{can} be inconsistent even for ridge ensembles when the ensemble size is two (see \Cref{prop:inconsistency}).
        This finding is in contrast to other known results of GCV for ridge (see \Cref{subsec:related-work} for more details).
        Nevertheless, experiments on synthetic data and real-world single-cell multiomic datasets demonstrate the applicability of GCV for tuning subsample sizes, even with moderate ensemble sizes (roughly of order 10).
    \end{itemize}
    
    \subsection{Related Work}\label{subsec:related-work}

    \emph{Ensembles and risk analysis.}
    Ensemble methods are effective in combining weak predictors 
    to build strong predictors in both regression and classification settings \citep{friedman_hastie_tibshirani_2009}.
    Early work on ensemble methods includes classical papers by \citet{breiman_1996,buhlmann2002analyzing}.
    There has been further work on the ensembles of smooth weak predictors \citep{buja2006observations,friedman_hall_2007}, non-parametric estimators \citep{buhlmann2002analyzing,loureiro2022fluctuations}, and classifiers \citep{hall_samworth_2005,samworth2012optimal}.
    Under proportional asymptotics, \citet{d2020double,adlam2020understanding,loureiro2022fluctuations} study ensemble learning under random feature models.
    For ridge ensembles, \citet{sollich1995learning,krogh1997statistical} derive risk asymptotics under Gaussian features.
    \citet{lejeune2020implicit} consider least squares ensembles obtained by subsampling such that the final subsampled dataset has more observations than the number of features.
    The asymptotic risk characterization for general data models 
    has been derived by \citet{patil2022bagging}.
    Both of these works show the equivalence between the subsample optimized full ridgeless ensemble and the optimal ridge under isotropic models.
    Our work significantly extends the scope of these results by characterizing risk equivalences for both optimal and suboptimal risks and for arbitrary feature covariance and signal structures.
    See the remarks after \Cref{thm:comparison_optimal_ridge}
    for a detailed comparison,

    \emph{Cross-validation and consistency.} 
    Cross-validation (CV) is arguably the most popular class of methods for model assessment and selection.
    Classical work on CV include: \citet{allen_1974,stone_1974,stone_1977,geisser_1975}, among others.
    We refer the reader to \citet{arlot_celisse_2010,zhang_yang_2015} for comprehensive  surveys of different CV variants.
    In practice, $k$-fold CV is widely used with typical $k$ being $5$ or $10$ \citep{friedman_hastie_tibshirani_2009,gyorfi_kohler_krzyzak_walk_2006}, but such small values of $k$ suffer from bias in high dimensions \citep{rad_maleki_2020}.
    The extreme case of leave-one-out cross-validation (LOOCV) (when $k = n$) alleviates the bias issues in risk estimation, and various statistical consistency properties of LOOCV have been analyzed in recent years; see, e.g., \citet{kale_kumar_vassilvitskii_2011,kumar_lokshtanov_vassilviskii_vattani_2013,obuchi_kabashima_2016,rad_zhou_maleki_2020}.
    Except for special cases, LOOCV is computationally expensive, and consequently, various approximations and their theoretical properties have been studied; see, e.g, \citet{wang_zhou_lu_maleki_mirrokni_2018,rad_maleki_2020,rad_zhou_maleki_2020,xu_maleki_rad_2019}.
    Generalized cross-validation (GCV) is a sort of approximation for the ``shortcut'' leave-one-out formula \citep{friedman_hastie_tibshirani_2009}, originally studied for the fixed-X design setting for linear smoothers by \citet{golub_heath_wabha_1979, craven_wahba_1979}.
    The consistency of GCV in such a setting has been investigated in \citet{li_1985, li_1986, li_1987}.
    More recently, in the random-$X$ setting, GCV has received considerable attention.
    In particular, consistency of GCV for ridge regression has been established in \citet{adlam_pennington_2020neural,hastie2020ridge,patil2021uniform,patil2022estimating,wei_hu_steinhardt} under various data settings.
    Our work contributes to this body of work by analyzing GCV 
    for subsampled ensemble ridge regression.

\section{Subsample and Ridge Equivalences}\label{sec:equiv}

    We consider the standard supervised regression setting.
    Let $\mathcal{D}_n = \{(\bx_j, y_j) : j \in [n] \}$ denote a dataset containing i.i.d.\ random vectors in $\RR^{p} \times \RR$, $\bX\in\RR^{n\times p}$ denote the feature matrix whose $j$-th row contains $\bx_j^\top$, and $\by\in\RR^n$ denote the response vector whose $j$-th entry contains $y_j$.
    For an index set $I\subseteq [n]$ of size $k$, let $\mathcal{D}_{I} = \{(\bx_j, y_j):\, j \in I\}$ be a subsampled dataset and let $\bL_I \in \RR^{n \times n}$ denote a diagonal matrix such that its $j$th diagonal entry is $1$ if $j \in I$ and $0$ otherwise. 
    Noting that the feature matrix and response vector associated with $\cD_I$ are $\bL_I \bX$ and $\bL_I \by$, respectively, the \emph{ridge} estimator $\hbeta_{k}^\lambda(\cD_I)$ fitted on $\cD_I$ (containing $k$ samples) with regularization parameter $\lambda>0$ can be expressed as:
    \begin{align}
        \hbeta^{\lambda}_{k}(\cD_I) 
        &= \argmin\limits_{\bbeta\in\RR^p}
        \sum_{j \in I} (y_j  - \bx_j^\top \bbeta)^2 / k
        + \lambda \| \bbeta \|_2^2 \notag \\
        &= (\bX^{\top} \bL_I \bX / k  + \lambda\bI_p)^{-1}{\bX^{\top} \bL_I \by}/{k} .\label{eq:ingredient-estimator}
    \end{align}
    Letting $\lambda\rightarrow0^+$, $\hbeta^{0}_{k}(\cD_I):=(\bX^{\top} \bL_I \bX/k)^{+}\bX^{\top} \bL_I \by/k$ becomes the so-called \emph{ridgeless} estimator, 
    where $\bA^+$ denotes the Moore-Penrose inverse of matrix $\bA$.
    
    \textbf{Ensemble estimator.}
    To introduce the ensemble estimator, it helps to define the set of all $k$ distinct elements from $[n]$ to be $\mathcal{I}_k:= \{\{i_1, i_2, \ldots, i_k\}:\, 1\le i_1 < i_2 < \ldots < i_k \le n\}$.
    Note that the cardinality of $\cI_k$ is $\smash{\binom{n}{k}}$.
    For $\lambda\geq0$, the ensemble estimator is then defined as:
    \begin{align}
        \label{eq:def-M-ensemble}
        \tbeta^{\lambda}_{k,M}(\cD_n;\{I_{\ell}\}_{\ell=1}^M) 
        &:= \frac{1}{M} \sum_{\ell\in[M]}\hbeta^{\lambda}_{k}(\cD_{I_\ell}),
    \end{align}
    where $I_{1},\ldots,I_M$ are simple random samples from $\cI_k$.
    The \emph{full-ensemble} ridge estimator is the average of predictors fitted on all possible subsampled datasets:
    \begin{align}
        \label{eq:def-full-ensemble}
        \tbeta^{\lambda}_{k,\infty}(\cD_n)
        &:= 
        \frac{1}{|\cI_k|} \sum_{I \in \cI_k} \hbeta^{\lambda}_{k}(\cD_{I})
        =
        \EE[\hbeta^{\lambda}_{k}(\cD_{I}) \,|\, \cD_n],
    \end{align}
    where the conditional expectation is taken with respect to the randomness of sampling from $\cI_k$.
    \Cref{lem:full-ensemble} shows that $\tbeta^{\lambda}_{k,\infty}(\cD_n)$ is also almost surely equivalent to letting the ensemble size $M$ tend to infinity in \eqref{eq:def-M-ensemble} conditioning on the full dataset $\cD_n$, thus justifying the notation in \eqref{eq:def-full-ensemble}.
    For simplicity, we drop the dependency on $\cD_n$, $\{I_{\ell}\}_{\ell=1}^M$ and only write $\tbeta^{\lambda}_{k,M}$, $\tbeta^{\lambda}_{k,\infty}$, when it is clear from the context.

    \textbf{Prediction risk.}
    We assess the performance of an $M$-ensemble predictor via conditional squared prediction risk:
    \begin{align}
        R_{k,M}^{\lambda}&:=\EE_{(\bx,y)}[(y-\bx^{\top} \tbeta_{k,M}^{\lambda})^2\mid\cD_n, \{I_{\ell}\}_{\ell = 1}^M] , \label{eq:R_M}
    \end{align}
    where $(\bx, y)$ is an independent test point sampled from the distribution as $\cD_n$.
    Note that the conditional risk $R_{k,M}^{\lambda}$ is a random variable that depends on both the dataset $\cD_n$ and the random samples $I_{\ell}$, $\ell = 1, \dots, M$.
    For the full ensemble estimator $\tbeta_{k,\infty}^{\lambda}$, the conditional prediction risk is defined analogously, except the risk now only depends on $\cD_n$:
    \begin{align}
        R_{k,\infty}^{\lambda} := \EE_{(\bx,y)}[(y-\bx^{\top} \tbeta_{k,\infty}^{\lambda})^2\mid\cD_n].
        \label{eq:R_inf}
    \end{align}

    \subsection{Data Assumptions} 

    For our theoretical results, we work under a proportional asymptotics regime, in which the original \emph{data aspect ratio} ($p / n$) converges to $\phi \in (0, \infty)$ as $n, p \to \infty$, and the \emph{subsample aspect ratio} ($p/k$) converges to $\phi_s$ as $k, p \to \infty$.
    Note that because $k \le n$, $\phi_s$ always lie in $[\phi, \infty]$.
    In addition, we impose two structural assumptions on the feature matrix and response vector as summarized in \Crefrange{asm:rmt-feat}{asm:lin-mod}, respectively.

    \begin{assumption}[Feature model]\label{asm:rmt-feat}
        The feature matrix decomposes as $\bX = \bZ\bSigma^{1/2}$, where 
        $\bZ\in\RR^{n\times p}$ contains i.i.d.\ entries with mean $0$, variance $1$, bounded moments of order $4+\delta$ for some $\delta > 0$,
        and $\bSigma \in \RR^{p \times p}$ is deterministic and symmetric with eigenvalues uniformly bounded between $r_{\min}>0$ and $r_{\max}<\infty$.        
        Let $\bSigma = \sum_{j=1}^p r_j\bw_j\bw_j^{\top}$ denote the eigenvalue decomposition, where $(r_j, \bw_j), j \in [p]$, are pairs of associated eigenvalue and normalized eigenvector.
        We assume there exists a deterministic distribution $H$ such that the empirical spectral distribution of $\bSigma$, $H_{p}(r) := p^{-1}\sum_{i=1}^{p} \ind_{\{r_i \le r\}}$, weakly converges to $H$, almost surely (with respect to $\bX$).
    \end{assumption}
    
    \begin{assumption}[Response model]\label{asm:lin-mod}
        The response vector decomposes as $\by=\bX\bbeta_0+\bepsilon$, where $\bbeta_0 \in \RR^{p}$ is an unknown signal vector with $\ell_2$-norm uniformly bounded and $\lim_{p\rightarrow\infty}\|\bbeta_0\|_2^2=\rho^2,$ and $\bepsilon$ is an unobserved error vector independent of $\bX$ with mean $0$, variance $\sigma^2$, and bounded moment of order $4 + \delta$ for some $\delta > 0$.
        We assume there exists a deterministic distribution $G$ such that the empirical distribution of $\bbeta_0$'s (squared) projection onto $\bSigma$'s eigenspace, $G_{p}(r) := \| \bbeta_0 \|_2^{-2} \sum_{i = 1}^{p} (\bbeta_0^\top \bw_i)^2 \, \ind_{\{ r_i \le r \}}$, weakly converges to $G$, almost surely (with respect to $\bX$).
    \end{assumption}

    \Cref{asm:rmt-feat,asm:lin-mod} are standard in the study of the ridge and ridgeless regression under proportional asymptotics; see, e.g., \citet{hastie2022surprises,patil2022mitigating,patil2022bagging}.
    It is possible to further relax both of these assumptions.
    Specifically, one can incorporate other feature models, e.g., random features \cite{mei_montanari_2022}, and can allow for certain non-linearities in the regression function \cite{bartlett_montanari_rakhlin_2021} for the response model.
    We leave these for future work.
    
    \begin{table*}[!t]
        \centering
        \begin{tabularx}{0.82\textwidth}{c|cccc|cccc}
            \toprule
            \multicolumn{1}{c}{\multirow{2}{*}{\textbf{Variable}}} & \multicolumn{4}{c}{\textbf{$M$-ensemble}} & \multicolumn{4}{c}{\textbf{Full ensemble}} \\\cmidrule(lr){2-5}\cmidrule(lr){6-9}
            \multicolumn{1}{c}{} & \multicolumn{2}{c}{\textbf{Finite-sample}} & \multicolumn{2}{c}{\textbf{Asymptotic}} & \multicolumn{2}{c}{\textbf{Finite-sample}} & \multicolumn{2}{c}{\textbf{Asymptotic}}\\
            \midrule
             Prediction risk & $R_{k,M}^{\lambda}$ & \eqref{eq:R_M} & $\RlamM$  & \eqref{eq:risk-det-with-replacement} & $R_{k,\infty}^{\lambda}$ & \eqref{eq:R_inf}& $\RlamM[\lambda][\infty]$ & \eqref{eq:risk-det-with-replacement} \\[0.2em]
             Test error & $\overline{R}_{k,M}^{\lambda}$ & \eqref{eq:Rb_M} & $\RlamM$  & \eqref{eq:risk-det-with-replacement} & \cellcolor{lightgray!25} & \cellcolor{lightgray!25} & \cellcolor{lightgray!25} & \cellcolor{lightgray!25}\\[0.2em]\cmidrule(lr){1-9}
             Training error & $T_{k,M}^{\lambda}$ & \eqref{eq:T_M} & $\RlamMtr$ & \eqref{eq:train-err} & $T_{k,\infty}^{\lambda}$ & \eqref{eq:T_inf} & $\RlamMtr[\lambda][\infty]$ & \eqref{eq:Ndet}\\[0.2em]             
             GCV denominator & $D_{k,M}^{\lambda}$ & \eqref{eq:D-kM} & \cellcolor{lightgray!25} & \cellcolor{lightgray!25} & $D_{k,\infty}^{\lambda}$ & \eqref{eq:D-inf}  & $\Ddet$ & \eqref{eq:Ddet}\\[0.2em]
             GCV estimator & $\gcv_{k,M}^{\lambda}$ & \eqref{eq:gcv-kM-inf} & \cellcolor{lightgray!25} & \cellcolor{lightgray!25} & $\gcv_{k,\infty}^{\lambda}$ & \eqref{eq:gcv-kM-inf} & $\gcvdet$ & \eqref{eq:gcv-det}\\[0.2em]
            \bottomrule
        \end{tabularx}
      \caption{Summary of notations and pointers to definitions of important empirical quantities used in this paper and their asymptotic limits.}\label{tab:notations}
    \end{table*}
    
    \subsection{Risk Equivalences}\label{subsec:connect-subsample-ridge}
    Under the above assumptions, \Cref{thm:ver-with-replacement} from \citet{patil2022bagging} implies that for every $M \ge 1$, the prediction risk $R_{k,M}^{\lambda}$ of the ridge and ridgeless predictors in the full ensemble converges to some deterministic limit $\RlamM$ as $k,n,p\rightarrow\infty$, $p/n\rightarrow \phi$ and $p/k\rightarrow\phi_s$.
    When $\phi_s=\phi$ (e.g., $k=n$), the asymptotic risk $\RlamM[\lambda][M][\phi][\phi]$ is equal to $\RlamM[\lambda][1][\phi][\phi]$ of the ridge predictor on the full dataset $\cD_n$ for all $M\ge1$, and we denote this risk simply by $\RlamM[\lambda][\infty][\phi][\phi]$.
    To facilitate our discussion and for simplicity, \Cref{tab:notations} provides pointers to definitions of all important quantities used in the paper.

    From a practical point of view, it is important to understand the least attainable risk that could be attained in the full ensemble.
    For the full ridge ensembles, we found that the explicit ridge regularization is unnecessary when considering optimal bagging and that the implicit regularization of ridgeless and subsampling suffices.
    The result below formalizes this empirical observation.
    
    \begin{theorem}[Optimal ridgeless ensemble vs optimal ridge]
    \label{thm:comparison_optimal_ridge}
        Under \Cref{asm:rmt-feat,asm:lin-mod}, for all $\phi\in(0,\infty)$, we have
        \begin{align*}
            \underbrace{\vphantom{\min_{\substack{\phi_s\geq\phi,\\\lambda\geq 0}}}\min_{\phi_s\geq \phi}
            \RzeroM{\infty}{\phi}}_{\substack{\mbox{opt. ensemble}\\\mbox{and no ridge}}} 
            \stackrel{(a)}{=} \underbrace{\vphantom{\min_{\substack{\phi_s\geq\phi,\\\lambda\geq 0}}}
            \min_{\lambda\geq 0}\RlamMe{\infty}{\phi}{\phi}}_{\substack{\mbox{no ensemble}\\\mbox{and opt. ridge}}} 
            \stackrel{(b)}{=} \underbrace{\min_{\substack{\phi_s\geq\phi,\\\lambda\geq 0}}\RlamMe{\infty}{\phi}{\phi_s}}_{\substack{\mbox{opt. ensemble}\\\mbox{and opt. ridge}}}. %
        \end{align*}
        Further, if $\phi_s^*$ is the optimal subsample aspect ratio for ridgeless, and $\lambda^*$ is the optimal ridge regularization with no subsampling, then for any $\theta \in [0,\lambda^*]$, full ridge ensemble with penalty parameter $\lambda = \lambda^* - \theta$ and subsample aspect ratio of $\phi_s = \phi + \theta (\phi_s^* - \phi)/\lambda^*$ also attains the optimal prediction risk. 
    \end{theorem}
    In words, Theorem~\ref{thm:comparison_optimal_ridge} says that optimizing subsample size (i.e. $k$) with the full ridgeless ensemble attains the same prediction risk as just optimizing the explicit regularization parameter (i.e., $\lambda$) of the ridge predictor. 
    Further, both of them are the same as optimizing both $k$ and $\lambda$. 
    If one uses a lesser ridge penalty than needed for optimal prediction (i.e., uses $\lambda < \lambda^*$), then a full ensemble at a specific subsample aspect ratio $\phi_s = \phi+(1 - \lambda/\lambda^*)(\phi_s^* - \phi) > \phi$ can recover the remaining ridge regularization. In this sense, the implicit regularization provided by the ensemble amounts to adding more explicit ridge regularization. Similarly, one can supplement a sub-optimal implicit regularization of subsampling by adding explicit ridge regularization.

    A special case of equivalence of $(a)$ in \Cref{thm:comparison_optimal_ridge}
    was previously formalized in~\citet{lejeune2020implicit,patil2022bagging} for isotropic covariates.
    Working with isotropic design helps their proof significantly, as the spectral distributions are the same for all $p, n$, and the closed-form expression of the asymptotic prediction risk can be derived analytically.
    However, in the general non-isotropic design, the asymptotic risk does not admit a closed-form expression, and one needs to account for this carefully.

    \textbf{General risk equivalences.} \Cref{thm:comparison_optimal_ridge} proves the risk equivalence of the ridge and full ensemble ridgeless when they attain minimum risk. \Cref{app:subsec:equiv} shows a further risk equivalence in the full range, i.e., 
    for any $\bar{\phi}_s\in[\phi,+\infty]$, there exists a $\bar{\lambda} \geq 0$ such that $\RzeroMe{\infty}{\phi}{\bar{\phi}_s} = \RlamM[\bar{\lambda}][\infty][\phi][\phi]$. Further, $\RlamM[\lambda][\infty]$ remains constant as $(\lambda, \phi_s)$ varies on the line segment $(1 - \theta)\cdot(\bar{\lambda}, \phi) + \theta\cdot(0, \bar{\phi}_s)$, for all $\theta\in[0, 1]$. 

    A remarkable implication of \Cref{thm:comparison_optimal_ridge} is that for a fixed dataset $\cD_n$, one does not need to tune both the subsample size (i.e., $k$) and the ridge regularization parameter (i.e., $\lambda$), but it suffices to fix for example $\lambda = 0$ and only tune $\phi_s$. 
    Alternatively, one can also fix $k = n$ and just tune $\lambda \ge 0$, which was considered in~\citet{patil2021uniform}.
    Performing tuning over $\lambda \ge 0$ requires one to discretize an infinite interval, while tuning the subsample size for a fixed $\lambda$ only requires searching over a finite grid varying from $k = 1$ to $k = n$.
    For this reason, we fix $\lambda$ and focus on tuning over $k$ in this paper.    
    In the next section, we investigate the problem of tuning the subsample size in the full ensemble to achieve the minimum oracle risk via generalized cross-validation.

\section{Generalized Cross-Validation}\label{sec:gcv-consistency}

        Suppose $\hf(\cdot; \cD_n) : \RR^{p} \to \RR$ is a predictor trained on $\cD_n$.
        We call $\hf(\cdot; \cD_n)$ a linear smoother if $\hf(\bx; \cD_n) = \ba_{\bx}^{\top}\by$ for some vector $\ba_{\bx}$ that only depends on the design $\bX$ (and $\bx$).
        Define the smoothing matrix $\bS \in \RR^{n \times n}$ with rows $\ba_{\bx_1}^{\top}, \ldots, \ba_{\bx_n}^{\top}$, which in turn is only a function of $\bX$.
        For any linear smoother, the generalized cross-validation (GCV) estimator of the prediction risk is defined to be $n^{-1}\| \by - \bS\by \|_2^2/(1 - n^{-1} \tr(\bS))^2$; see, e.g., \citet[Section 5.3]{wasserman2006all}.
        The numerator of GCV is the training error, which typically is biased downwards, and the denominator attempts to account for such optimism of the predictor.
        
        \textbf{Ensemble GCV.}
        Before we analyze GCV for the ridge ensemble, we first introduce some notations.
        Let $I_{1:M}:=\cup_{\ell=1}^MI_{\ell}$ and $I_{1:M}^c:=[n]\setminus I_{1:M}$.
        We define the \emph{in-sample} training error and the \emph{out-of-sample} test error 
        of $\tbeta^\lambda_{k,M}$ as:
        \begin{align}
            T_{k,M}^{\lambda}&:=\frac{1}{|I_{1:M}|}\sum_{i\in I_{1:M}} ( y_i-\bx_i^{\top}\tbeta^{\lambda}_{k,M})^2  , \label{eq:T_M}\\
            \overline{R}_{k,M}^{\lambda}
            &:=\frac{1}{|I_{1:M}^c|}
            \sum_{i\in I_{1:M}^c} ( y_i-\bx_i^{\top}\tbeta^{\lambda}_{k,M})^2 .\label{eq:Rb_M}
        \end{align}
        Since the full ensemble estimator $\tbeta^\lambda_{k,\infty}$ uses all the data $\cD_n$, its training error, denoted by $T_{k,\infty}^\lambda$, is simply:
        \begin{align}
            T_{k,\infty}^{\lambda} &:= \frac{1}{n}\sum_{i\in [n]} ( y_i-\bx_i^{\top}\tbeta^{\lambda}_{k,\infty})^2.   \label{eq:T_inf}
        \end{align}
        Since $I_{1:M} \asto [n]$ for any $n \in \NN$ as $M \to \infty$, the notation $T_{k,\infty}^\lambda$ in \eqref{eq:T_inf} is justified as a limiting case of \eqref{eq:T_M} (see \Cref{app:subsec:full-ensemble} for more details).
        Now, observe that a ridge ensemble is a linear smoother because        
        $\bX_{I_{1:M}}\tbeta^{\lambda}_{k,M}= \bS^{\lambda}_{k,M}\by_{I_{1:M}}$, where the smoothing matrix $\bS^{\lambda}_{k,M}$ is given by:
        {
        \small
        \begin{align}
            \bS^{\lambda}_{k,M} &= \frac{1}{M}\sum_{\ell=1}^M\bX_{I_{\ell}}({\bX_{I_{\ell}}^{\top} \bX_{I_{\ell}}} / {k}+\lambda \bI_p)^{+}{\bX_{I_{\ell}}^{\top} }/{k}.\label{eq:smooth-matrix-M}
        \end{align}
        }
        Analogously, the smoothing matrix for $\tbeta^\lambda_{k,\infty}$ is given by:
        {\small
        \begin{align}
            \bS^{\lambda}_{k,\infty} &= \frac{1}{|\cI_k|}\sum_{I\in \cI_k}\bX({\bX^{\top} \bL_I \bX}/{k}+\lambda \bI_p)^{+}{\bX^{\top} \bL_I}/{k}.
        \end{align}}
        Thus, the GCV estimates for ridge predictors in the finite and full ensemble case 
        are respectively given by:
        \begin{align}
            \gcv_{k,M}^{\lambda}
            &= \frac{T_{k,M}^{\lambda}}{D_{k,M}^{\lambda}},\qquad 
            \gcv_{k,\infty}^{\lambda}= \frac{T_{k,\infty}^{\lambda}}{D_{k,\infty}^{\lambda}},\label{eq:gcv-kM-inf}
        \end{align}    
        where the denominators $D^{\lambda}_{k,M}$
        and $D^\lambda_{k,\infty}$ are as follows:
        \begin{align}
            D_{k,M}^{\lambda} &:= (1 - |I_{1:M}|^{-1} \tr(\bS^{\lambda}_{k,M}))^2, \label{eq:D-kM}\\
            D_{k,\infty}^{\lambda} &:= (1 - n^{-1} \tr(\bS^{\lambda}_{k,\infty}))^2 .\label{eq:D-inf}
        \end{align}

    \subsection{Full-Ensemble Uniform Consistency}\label{subsec:uniform-consistency}

    Let $\cK_n\subset\{0, 1, \ldots, n\}$ be a grid of subsample sizes that covers the full range of $[0, n]$ asymptotically in the sense that $\{k/n:\,k\in\mathcal{K}_n\}$ ``converges'' to the set $[0, 1]$ as $n\to\infty$. 
    One simple choice is to set 
    \begin{align*}
        \cK_n=\{0, k_0,2k_0,\ldots. \lfloor n / k_0\rfloor k_0\}, \label{eq:grid}
    \end{align*}
    where the increment is $k_0=\lfloor n^{\nu}\rfloor$ for some $\nu\in(0,1)$.
    Here, we adopt the convention that when $k = 0$, the predictor reduces to a null predictor that always returns zero.
    Based on the definition above, we now present the uniform consistency results of the GCV estimator \eqref{eq:gcv-kM-inf} for full ensembles when the ridge regularization parameter $\lambda$ is fixed.
    \begin{theorem}[Uniform consistency of GCV]\label{thm:uniform-consistency-k}
        Suppose \Cref{asm:rmt-feat,asm:lin-mod} hold.
        Then, for all $\lambda\geq 0$, we have
        \begin{align*}
            \max_{k\in\cK_n} |\gcv_{k,\infty}^{\lambda} - R_{k,\infty}^{\lambda}| \asto 0,
        \end{align*}
        as $n, p \to \infty$ such that $p/n \to \phi \in (0, \infty)$.
    \end{theorem}

    \Cref{thm:uniform-consistency-k}
    shows the uniform consistency of GCV in the full ensemble for fixed subsample size $k$ and ridge regularization parameter $\lambda$.
    The almost sure qualification in \Cref{thm:uniform-consistency-k}
    is with respect the entire training data $(\bX, \by)$.
    An implication of \Cref{thm:uniform-consistency-k} is that one can select the optimal subsample size in a data-dependent manner, i.e.,
    selecting $\hat{k}^{\lambda}\in\argmin_{k\in\cK_n}\gcv_{k,\infty}^{\lambda}$ guarantees to track the minimum prediction risk $\min_{k\in[n]}R_{k,\infty}^{\lambda}$ asymptotically.
    
    We first provide numerical illustrations for \Cref{thm:uniform-consistency-k} under the non-isotropic AR(1) data model, which is the same as the one used for \Cref{fig:overview}; see \Cref{app:numerical-details} for model details.
    \Cref{fig:gcv_est} shows both the GCV estimate and the asymptotic risk for the full ridge ensemble.
    We observe a close match of the theoretical curves and the GCV estimates.

    \begin{figure*}[!t]
        \centering
        \includegraphics[width=0.95\textwidth]{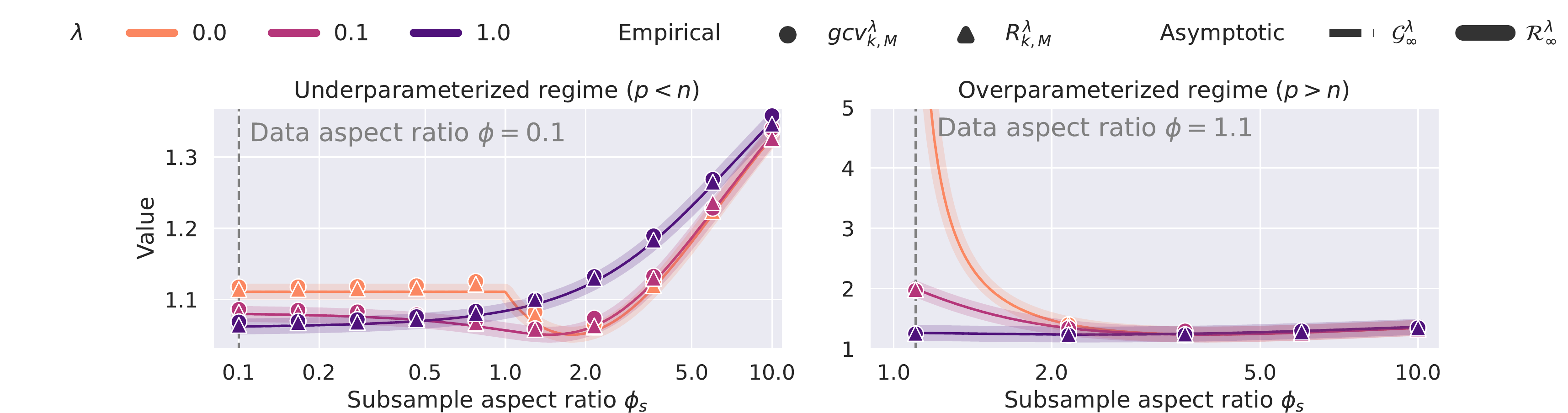}
        \caption{Asymptotic risk and GCV curves for full ridge ensembles, under model \eqref{eq:model-ar1} when $\rhoar=0.5$ and $\sigma^2=1$ with varying regularization parameters $\lambda \in \{ 0, 0.1, 1\}$ and subsample sizes $k=\lfloor p/\phi_s\rfloor$.
        The points denote finite-sample risks averaged over 50 dataset repetitions with an ensemble size of $M=500$, with $n=\lfloor p/\phi\rfloor$ and $p=500$.
        The left and the right panels illustrate the underparameterized and overparameterized cases with the limiting data aspect ratio $\phi=0.1$ and $\phi=1.1$, respectively.}
        \label{fig:gcv_est}
    \end{figure*}
    \begin{figure}[!t]
        \centering
        \includegraphics[width=0.45\textwidth]{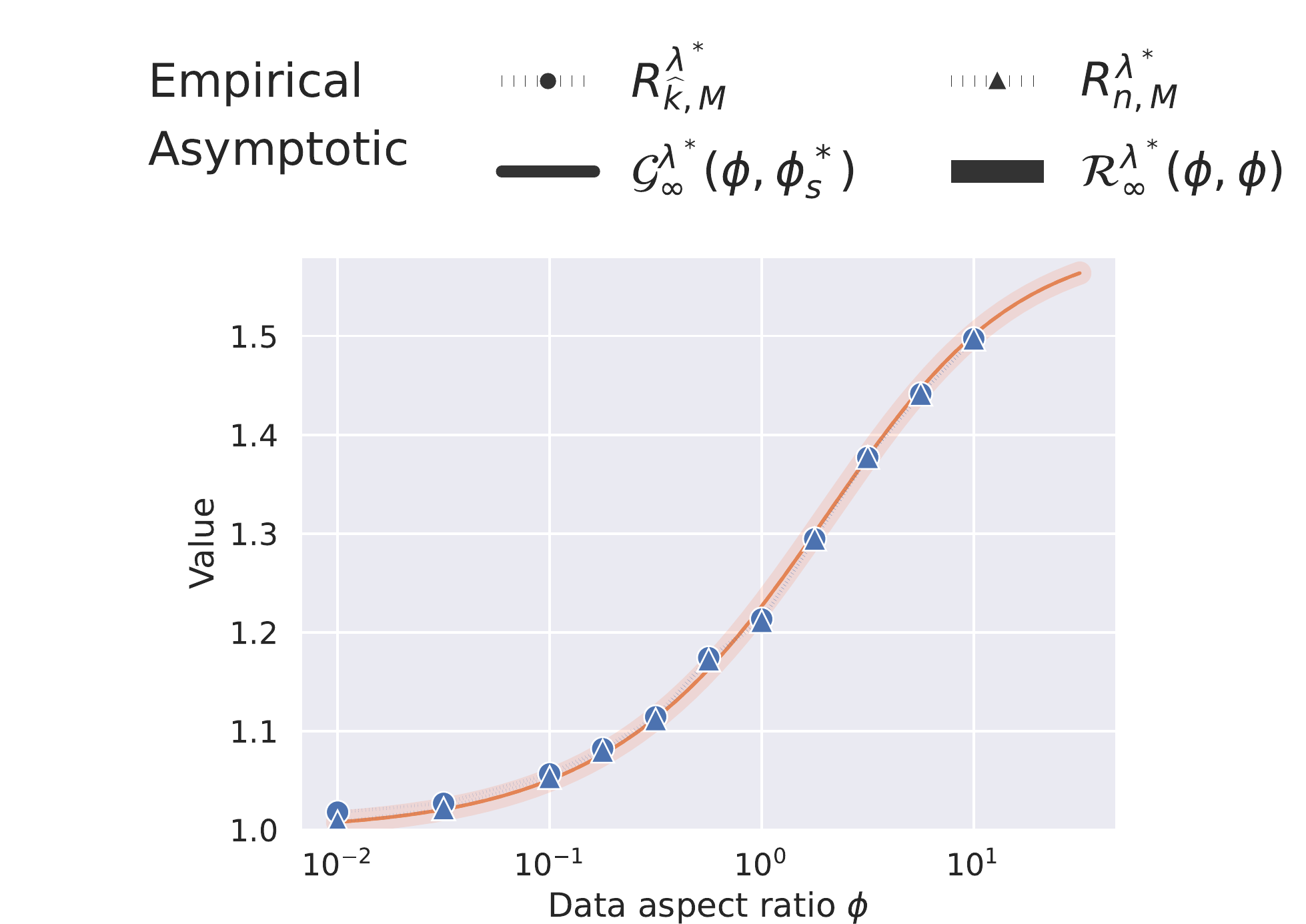}
        \caption{Asymptotic prediction risk curves with optimal tuned parameters $\lambda^*$ and $\phi_s^*$, under model \eqref{eq:model-ar1} when $\rhoar=0.5$, $\sigma^2=1$, for varying data aspect ratio $\phi$.
        The curves represent the theoretical asymptotic GCV estimate in the full ensemble and the asymptotic risk of the optimal ridge predictors.
        The points represent the finite-sample risks of the best 500-ensemble ridgeless and the best ridge predictor averaged over 50 dataset repetitions, with $n = \lfloor p/\phi\rfloor$ and $p = 500$.}
        \label{fig:gcv-opt}
    \end{figure}

    Combining \Cref{thm:comparison_optimal_ridge} and \Cref{thm:uniform-consistency-k}, we can obtain the following corollary regarding GCV subsample tuning.
    \begin{corollary}[Ridge tuning by GCV subsample tuning]\label{cor:gcv-opt-ridge}
        Suppose \Cref{asm:rmt-feat,asm:lin-mod} hold.
        Then, we have
        \[
            \gcv_{\hat{k}^0,\infty}^0 ~\asto~  \min_{\phi_s\geq\phi, \lambda\geq 0}\RlamMe{\infty}{\phi}{\phi_s},
        \]
        as $k, n, p \to \infty$ such that $p/n \to \phi \in (0, \infty)$.
    \end{corollary}

    \Cref{cor:gcv-opt-ridge} certifies the validity of GCV tuning for achieving the optimal risk over all possible regularization parameters and subsample sizes.
    In practice, tuning for the ridge parameter $\lambda$ requires one to determine a grid of $\lambda$'s for cross-validation.
    However, the maximum value for the grid is generally chosen by some ad hoc criteria.
    For example, there is no default maximum value for ridge tuning in the widely-used package \texttt{glmnet} \citep{friedman2010regularization}.
    From \Cref{thm:comparison_optimal_ridge}, when the signal-noise ratio $\rho^2/\sigma^2$ is small, the subsample size should be small enough (so that $\phi_s$ is large), and the range of $\lambda$'s grid should be large enough to cover its optimal value.
    On the contrary, the GCV-based method does not need such an upper bound for the grid $\cK_n$ of subsample sizes because the sample size provides a natural grid in finite samples,
    informed by the dataset.

    In \Cref{cor:gcv-opt-ridge}, we fix the ridge regularization parameter $\lambda$ to be zero.
    But, one can also use other value of $\lambda<\lambda^*$ and the similar statement still holds with $\gcv_{\hat{k}^0,\infty}^0$ replaced by $\gcv_{\hat{k}^{\lambda},\infty}^{\lambda}$ based on \Cref{thm:comparison_optimal_ridge}.
    Furthermore, one can construct the estimator of $\lambda^*$ as $\hat{\lambda}= \lambda (n-\hat{k}^0) / (\hat{k}^{\lambda}-\hat{k}^0)$ by extrapolating the line segment between $(0,\hat{k}^0)$ and $(\lambda,\hat{k}^{\lambda})$.

    In \Cref{fig:gcv-opt}, we numerically compare the optimal subsampled ridgeless ensemble with the optimal ridge predictor to verify \Cref{cor:gcv-opt-ridge}.
    As we can see, their theoretical curves exactly match, and the empirical estimates in finite samples are also close to their asymptotic limits.

    \subsection{A Finite-Ensemble Inconsistency Result}\label{subsec:inconsistency}
    While deriving GCV asymptotics in the proof of \Cref{thm:uniform-consistency-k} for the full ensemble, we also obtain as a byproduct the asymptotic limit of the GCV estimate for finite ensembles.
    From related work \citep{patil2021uniform} and \Cref{thm:uniform-consistency-k}, we already know that $\gcv_{k,1}^{\lambda}$ and $\gcv_{k,\infty}^{\lambda}$ are consistent estimators of the non-ensemble risk ($R_{k,1}^{\lambda}$) and the full ensemble risk ($R_{k,\infty}^{\lambda}$), respectively.
    However, $\gcv_{k,M}^{\lambda}$ for $1<M<\infty$ may not be consistent, which is somewhat surprising.
    As an example, GCV for $M=2$ is not a consistent estimator for the prediction risk $R_{k,2}^{\lambda}$, as shown in the following proposition.
    
    \begin{proposition}[GCV inconsistency for ridgeless, $M=2$]\label{prop:inconsistency}
        Suppose \Cref{asm:rmt-feat,asm:lin-mod} hold with $\rho^2,\sigma^2 \in (0, \infty)$.
        Then, for any $\phi \in (0, \infty)$, 
        we have
        \[
            | \gcv_{k,2}^{0} - R^{0}_{k,2} | \not\pto 0,
        \]
        as $k, n, p \to \infty$, $p/n \to \phi$, $p/k \to \phi_s \in (1,\infty) \cap (\phi, \infty)$.
    \end{proposition}

    Intuitively, the inconsistency for a finite in large part happens because, for a finite $M$, the residuals computed using the bagged predictor contain non-negligible fractions of out-of-sample and in-sample, and all of them are treated equally.
    As a result, the GCV estimate for finite ensembles indirectly relates to the original data through the aspect ratios $(\phi,\phi_s)$, even though the GCV estimate is computed only using the training observations.
    See \Cref{sec:discussion} about possible approaches for the corrected GCV estimate for arbitrary ensemble sizes.
    Though in practice, the correction may not be crucial for a moderate $M$, because the GCV estimate is close to the underlying target as shown in \Cref{fig:gcv-M}.

    \begin{figure}[!t]
        \centering
        \includegraphics[width=0.48\textwidth]{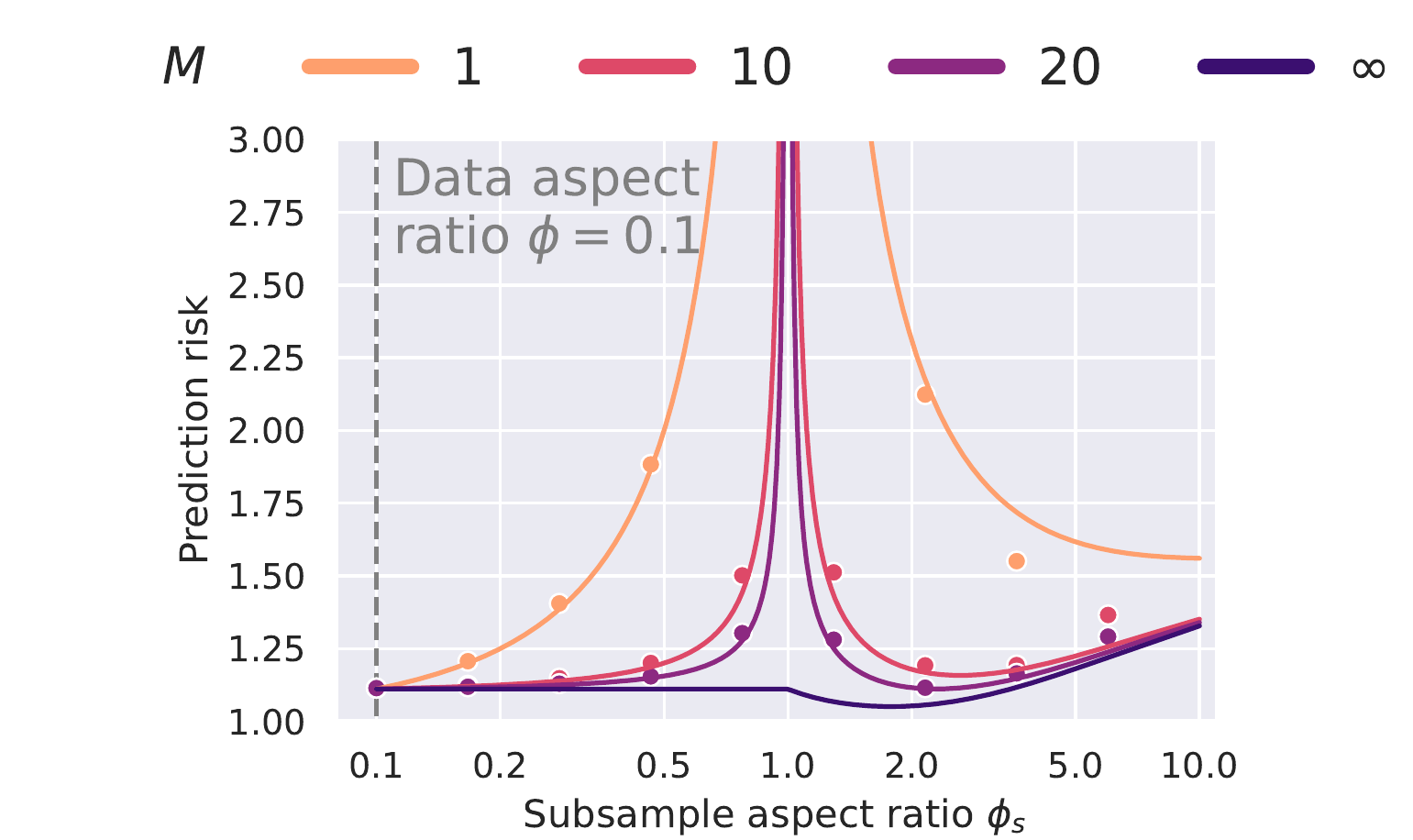}
        \caption{Asymptotic prediction risk curves of ridgeless ensembles, under model \eqref{eq:model-ar1} when $\rhoar=0.5$, $\sigma^2=1$, and $\phi=0.1$.
        The points denote the finite-sample GCV estimates of ridgeless ensembles for varying ensemble sizes $M \in \{ 1, 10, 20 \}$ averaged over 50 dataset repetitions, with $n=\lfloor p/\phi \rfloor$ and $p=500$.
        }
        \label{fig:gcv-M}
    \end{figure}

    \subsection{Proof Outline of \Cref{thm:uniform-consistency-k}}\label{subsec:proof-outline}
        There are three key steps are involved to prove \Cref{thm:uniform-consistency-k}.
        (1) Deriving the asymptotic limit of the prediction risk $R_{k,M}^{\lambda}$. 
        (2) Deriving the asymptotic limit of GCV estimate $\gcv_{k,\infty}^{\lambda}$.
        (3) Showing pointwise consistency in $k$ by matching the two limits and then lifting to uniform convergence in $k$.
        We briefly explain key ideas for showing the three steps below.
        
    \textbf{(1) Asymptotic limit of risk.}\label{subsec:asymp-risk}
        We build upon prior results on the risk analysis of ridge ensembles.
        Under \Cref{asm:rmt-feat,asm:lin-mod}, \Cref{thm:ver-with-replacement} adapted from \citet{patil2022bagging} implies that the conditional prediction risks under proportional asymptotics converge to certain deterministic limits:
        \begin{align}
            R_{k,M}^{\lambda} \asto \RlamM[\lambda][M],\qquad R_{k,\infty}^{\lambda} \asto \RlamM[\lambda][\infty], \label{eq:Rrisk}
        \end{align}
        where $\RlamM[\lambda][M]$, $\RlamM[\lambda][\infty]$ are as defined in \eqref{eq:risk-det-with-replacement}.

    \textbf{(2) Asymptotic limit of GCV.}\label{subsec:asymp-gcv}
    To analyze the asymptotic behavior of the GCV estimates, we obtain the asymptotics of the denominator and the numerator of GCV separately.
    We first show the regular cases when $\phi_s<\infty$ and $\lambda>0$, 
    and then incorporate boundary cases of $\phi_s = \infty$ and $\lambda = 0$.
    Our analysis begins with the following lemma that provides asymptotics for the denominator $D_{k,\infty}^{\lambda}$ (as in \eqref{eq:D-inf}) of GCV:
    \begin{lemma}[Asymptotics of the GCV denominator] \label{lem:gcv-den}
        Suppose \Cref{asm:rmt-feat} holds.
        Then, for all $\lambda> 0$,
        \begin{align*}
            D_{k,\infty}^{\lambda}\asto \Ddet,
        \end{align*}
        as $k, n, p \to \infty$, $p/n \to \phi \in (0, \infty)$, $p/k \to \phi_s \in [\phi, \infty)$.
    \end{lemma}

    It is worth noting that \Cref{lem:gcv-den} does not require \Cref{asm:lin-mod} because the smoothing matrix only concerns the design matrix $\bX$ and does not depend on the response $\by$.
    
    Towards obtaining asymptotics for the numerator $T_{k,\infty}^{\lambda}$ (as in \eqref{eq:T_inf}) of GCV, we first decompose $T_{k,\infty}^{\lambda}$ into simpler components via \Cref{lem:decomp-train-err}. Specifically, the full mean squared training error admits the following decomposition:
    \[
        T_{k,\infty}^{\lambda}
        - \sum_{m=1}^2
            ( c_m T_{k,m}^{\lambda}
            + (1-c_m) \overline{R}_{k,m}^{\lambda}) \asto 0,
    \]
    where $c_1=\phi/\phi_s$ and $c_2= 2\phi(2\phi_s-\phi)/\phi_s^2$.
    Here, $T_{k,m}^{\lambda}$ and $\overline{R}_{k,m}^{\lambda}$ are the in-sample training and out-of-sample test errors of the $m$-ensemble for $m=1$ and $2$, as defined in \eqref{eq:T_M} and \eqref{eq:Rb_M}.
    This decomposition implies that the full training error is asymptotically simply a linear combination of training and test errors.
    Therefore, it suffices to obtain the asymptotics of each of these components.
    As analyzed in \Cref{lem:conv-test-err}, it is easy to show that the test errors converge to $\sR_{m}^{\lambda}$ for $m=1,2$.
    On the other hand, it is more challenging to derive the asymptotic limits for the training errors $T_{k,m}^{\lambda}$.
    We first split the $T_{k,m}^{\lambda}$ into finer components via a bias-variance decomposition of $T_{k,m}^{\lambda}$. 
    By developing novel asymptotic equivalents of resolvents arising from the decomposition of in-sample errors in \Cref{lem:resolv-insample}, we are able to show the convergence of the bias and variance components (see \Cref{lem:conv-train-err}).
    Combining \Crefrange{lem:decomp-train-err}{lem:conv-train-err} yields the convergence of $T_{k,m}^{\lambda}$ to a deterministic limit $\sT_{m}^{\lambda}$
    as summarized in the following lemma:
    
    \begin{lemma}[Asymptotics of the GCV numerator]\label{lem:gcv-num}
        Suppose \Cref{asm:rmt-feat,asm:lin-mod} hold.
        Then, for all $\lambda> 0$,
        \begin{align}
            T_{k,\infty}^{\lambda}\asto\sT_{\infty}^{\lambda}
            &= \sum_{m=1}^2
            ( c_m \sT_{m}^{\lambda}
            + (1-c_m) \sR_{m}^{\lambda}),\label{eq:Ndet}
        \end{align}
        as $k, n, p \to \infty$, $p/n \to \phi \in (0, \infty)$,
        $p/k \to \phi_s \in [\phi, \infty)$,
        where $c_1=\phi/\phi_s$ and $c_2= 2\phi(2\phi_s-\phi)/\phi_s^2$.
    \end{lemma}

    Finally, the boundary cases when $\phi_s=+\infty$ and $\lambda=0$ are taken care of in succession by \Cref{prop:Rdet-ridge-infinity} and \Cref{prop:Rdet-lam-0}, respectively.
    Combining the above results provides the asymptotics for the GCV estimate in the full ensemble:

    \begin{proposition}[Asymptotics of GCV for full ensemble]\label{prop:gcv}
        Suppose \Cref{asm:rmt-feat,asm:lin-mod} hold.
        Then, for all $\lambda\geq 0$,
        \begin{align}
            \gcv_{k,\infty}^{\lambda} \asto  \gcvdet &:= \frac{\sT_{\infty}^{\lambda}(\phi,\phi_s)}{\Ddet},
            \label{eq:gcv-det}
        \end{align}
        as $k, n, p \to \infty$, $p/n \to \phi \in (0, \infty)$,
        $p/k \to \phi_s \in [\phi, \infty]$.
    \end{proposition}

    \textbf{(3) Asymptotics matching and uniform convergence.}\label{subsec:asymp-equiv}
        The asymptotic limits obtained in the first steps can be shown to match with each other by algebraic manipulations.
        This shows the pointwise consistency in \Cref{thm:uniform-consistency-k}.
        The uniform convergence then follows by applying a certain Ce\`saro-type mean convergence lemma (see \Cref{lem:conv_cond_expectation}).

\section{Real Data Example: Single-Cell Multiomics}\label{sec:sc}
    We compare tuning subsample size in the full ridgeless ensemble with tuning the ridge parameter on the full data in a real-world data example from multiomics.
    This single-cell CITE-seq dataset from \citet{hao2021integrated} consists of 50,781 human peripheral blood mononuclear cells (PBMCs) originating from eight volunteers post-vaccination (day 3) of an HIV vaccine, which simultaneously measures 20,729 genes and 228 proteins in individual cells.
    
    We follow the standard preprocessing procedure in single-cell data analysis \citep{hao2021integrated,du2022robust} to select the top 5,000 highly variable genes and the top 50 highly abundant surface proteins, which exhibit high cell-to-cell variations in the dataset.
    \begin{figure}[!t]
        \centering
        \includegraphics[width=0.49\textwidth]{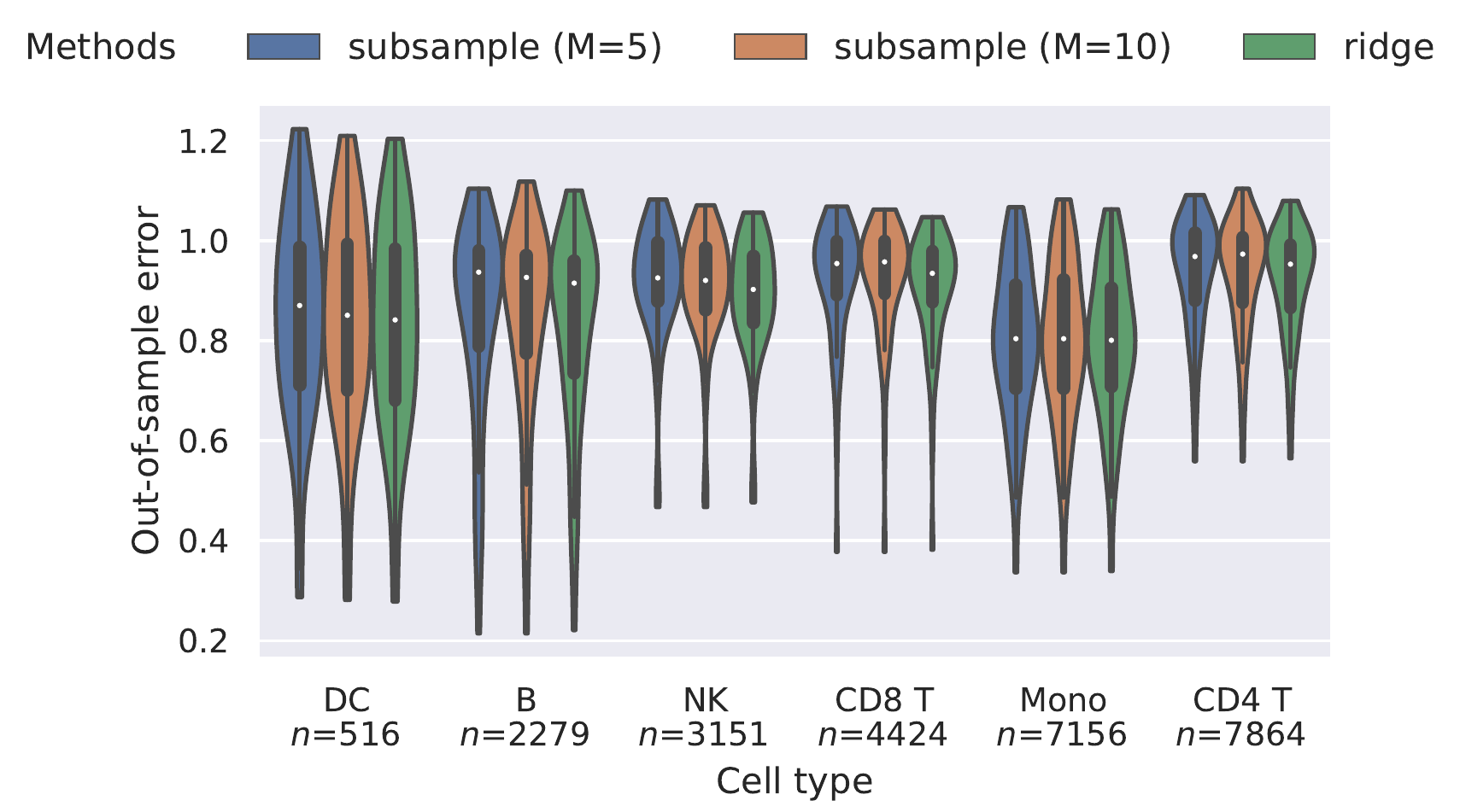}
        \caption{Violin plots of mean squared errors on the randomly held out test sets of different tuning methods for predicting the abundances of 50 proteins in the single-cell CITE-seq dataset. The sizes of the test sets are the same as the sizes of the training sets.}
        \label{fig:sc}
    \end{figure}
    The gene expression and protein abundance counts for each cell are then divided by the total counts for that cell and multiplied by $10^4$ and log-normalized.
    We randomly hold out half of the cells in each cell type as a test set.
    The top 500 principal components of the standardized gene expressions are used as features to predict protein abundances.
    The results of using ensembles with subsample size ($k$) tuning and ridge tuning ($\lambda$) without subsampling based on the GCV estimates are compared in \Cref{fig:sc}.
    For the former, we search over the grid of 25 $k$'s from $n^{\nu}$ to $n$ spaced evenly on the log scale, with $\nu=0.5$ and sample size $n$ ranges from 516 to 7864 for different cell types.
    For the latter, we search over the grid of 100 $\lambda$'s from $10^{-2}$ to $10^2$ spaced evenly on log scale.
    Since different cell types have different sample sizes, this results in different data aspect ratios, presented in increasing order in \Cref{fig:sc}.

    From \Cref{fig:sc}, we see that using a moderate ensemble size ($M=5$ or $10$) and tuning the subsample size have a very similar performance to only tuning the ridge regularization parameter in the full dataset.
    This suggests that the results of \Cref{cor:gcv-opt-ridge} also hold even on real data for different data aspect ratios.
    As discussed after \Cref{cor:gcv-opt-ridge}, subsample tuning is easier to implement because the dataset provides a natural lower and upper bound of the subsample size.
    On the other hand, ridge tuning requires one to heuristically pick the upper regularization threshold for the search grid.

\section{Discussion and Future Directions}\label{sec:discussion}

In this work,
we provide the risk characterization for the full ridge ensemble and establish the oracle risk equivalences between the full ridgeless ensemble and ridge regression. 
At a high level, these equivalences show that implicit regularization induced by subsampling matches explicit ridge regularization, i.e., a subsampled ridge predictor with penalty $\lambda_1$ has the same risk as another ridge predictor with penalty $\lambda_2\geq \lambda_1$.
Additionally, we prove the uniform consistency of generalized cross-validation for full ridge ensembles, which implies the validity of GCV tuning (that does not require sample splitting) for optimal predictive performance.
We describe next some avenues for future work moving forward.

\textbf{Bias correction for finite ensembles.}
In \Cref{prop:inconsistency}, we show that the GCV estimate can be inconsistent in the finite ridge and ridgeless ensembles.
The inconsistency for $M = 2$ occurs because sampling from the whole dataset induces extra randomness beyond those of the training observations used to compute the GCV estimates.
Our analysis of GCV for the full ensemble suggests a way to correct the bias of the GCV estimate.
In \Cref{app:correct}, we outline a possible correction strategy for finite ensembles based on out-of-bag estimates.
An intriguing next research direction is to investigate the implementation 
and uniform consistency of the corrected GCV for finite ensembles in detail.

\textbf{Extensions to other error metrics.}
In this paper, we focus on the in-distribution squared prediction risk. 
It is of interest to extend the equivalences for other error metrics, such as squared estimation risk, general prediction risks, and other functionals of the out-of-sample error distribution, like the quantiles of the error distribution.
Additionally, for the purposes of tuning, it is also of interest to extend the GCV analysis to estimate such functionals of the out-of-sample error distribution.
Such functional estimation could be valuable in constructing prediction intervals for the unknown response.
The technical tools introduced in \citet{patil2022estimating} involving leave-one-out perturbation techniques could prove useful for such an extension.
Furthermore, this extension would also allow for extending the results presented in this paper to hold under a general non-linear response model.

\textbf{Extensions to other base predictors.}
Finally, the focus of this paper is the base ridge predictor.
A natural extension of the current work is to consider kernel ridge regression.
Going further, it is of much interest to consider other regularized predictors, such as lasso.
Whether optimal subsampled lassoless regression still matches with the optimal lasso is an interesting question.
There is already empirical evidence along the lines of \Cref{fig:overview} for such a connection.
The results proved in the current paper make us believe that there is a general story quantifying the effect of implicit regularization by subsampling and that provided by explicit regularization.
Whether the general story unfolds as neatly as presented here for ridge regression remains an exciting next question!

\section*{Acknowledgements}

We are grateful to
Ryan Tibshirani,
Alessandro Rinaldo,
Yuting Wei,
Matey Neykov,
Daniel LeJeune,
Shamindra Shrotriya
for many helpful conversations 
surrounding ridge regression, subsampling, and generalized cross-validation.
Many thanks are also due to the anonymous reviewers
for their encouraging remarks and insightful questions
that have informed several new directions for follow-up future work.
In particular, special thanks to the reviewer
``2t75'' for a wonderful review
and highlighting other related works
on subsampling 
that have made their way into the manuscript.

\bibliography{main.bib}

\begin{thebibliography}{67}
\providecommand{\natexlab}[1]{#1}
\providecommand{\url}[1]{\texttt{#1}}
\expandafter\ifx\csname urlstyle\endcsname\relax
  \providecommand{\doi}[1]{doi: #1}\else
  \providecommand{\doi}{doi: \begingroup \urlstyle{rm}\Url}\fi

\bibitem[Adlam \& Pennington(2020{\natexlab{a}})Adlam and
  Pennington]{adlam2020understanding}
Adlam, B. and Pennington, J.
\newblock Understanding double descent requires a fine-grained bias-variance
  decomposition.
\newblock \emph{Advances in neural information processing systems},
  33:\penalty0 11022--11032, 2020{\natexlab{a}}.

\bibitem[Adlam \& Pennington(2020{\natexlab{b}})Adlam and
  Pennington]{adlam_pennington_2020neural}
Adlam, B. and Pennington, J.
\newblock The neural tangent kernel in high dimensions: Triple descent and a
  multi-scale theory of generalization.
\newblock In \emph{International Conference on Machine Learning}. PMLR,
  2020{\natexlab{b}}.

\bibitem[Allen(1974)]{allen_1974}
Allen, D.~M.
\newblock The relationship between variable selection and data augmentation and
  a method for prediction.
\newblock \emph{Technometrics}, 16\penalty0 (1):\penalty0 125--127, 1974.

\bibitem[Arlot \& Celisse(2010)Arlot and Celisse]{arlot_celisse_2010}
Arlot, S. and Celisse, A.
\newblock A survey of cross-validation procedures for model selection.
\newblock \emph{Statistics surveys}, 4:\penalty0 40--79, 2010.

\bibitem[Bai \& Silverstein(2010)Bai and Silverstein]{bai2010spectral}
Bai, Z. and Silverstein, J.~W.
\newblock \emph{Spectral Analysis of Large Dimensional Random Matrices}.
\newblock Springer, 2010.
\newblock Second edition.

\bibitem[Bartlett et~al.(2020)Bartlett, Long, Lugosi, and
  Tsigler]{bartlett_long_lugosi_tsigler_2020}
Bartlett, P.~L., Long, P.~M., Lugosi, G., and Tsigler, A.
\newblock Benign overfitting in linear regression.
\newblock \emph{Proceedings of the National Academy of Sciences}, 117\penalty0
  (48):\penalty0 30063--30070, 2020.

\bibitem[Bartlett et~al.(2021)Bartlett, Montanari, and
  Rakhlin]{bartlett_montanari_rakhlin_2021}
Bartlett, P.~L., Montanari, A., and Rakhlin, A.
\newblock Deep learning: a statistical viewpoint.
\newblock \emph{Acta numerica}, 30:\penalty0 87--201, 2021.

\bibitem[Belkin et~al.(2020)Belkin, Hsu, and Xu]{belkin_hsu_xu_2020}
Belkin, M., Hsu, D., and Xu, J.
\newblock Two models of double descent for weak features.
\newblock \emph{SIAM Journal on Mathematics of Data Science}, 2\penalty0
  (4):\penalty0 1167--1180, 2020.

\bibitem[Bloemendal et~al.(2016)Bloemendal, Knowles, Yau, and
  Yin]{bloemendal2016principal}
Bloemendal, A., Knowles, A., Yau, H.-T., and Yin, J.
\newblock On the principal components of sample covariance matrices.
\newblock \emph{Probability theory and Related Fields}, 164\penalty0
  (1):\penalty0 459--552, 2016.

\bibitem[Breiman(1996)]{breiman_1996}
Breiman, L.
\newblock Bagging predictors.
\newblock \emph{Machine Learning}, 24\penalty0 (2):\penalty0 123--140, 1996.

\bibitem[B{\"u}hlmann \& Yu(2002)B{\"u}hlmann and Yu]{buhlmann2002analyzing}
B{\"u}hlmann, P. and Yu, B.
\newblock Analyzing bagging.
\newblock \emph{The Annals of Statistics}, 30\penalty0 (4):\penalty0 927--961,
  2002.

\bibitem[Buja \& Stuetzle(2006)Buja and Stuetzle]{buja2006observations}
Buja, A. and Stuetzle, W.
\newblock Observations on bagging.
\newblock \emph{Statistica Sinica}, pp.\  323--351, 2006.

\bibitem[Craven \& Wahba(1979)Craven and Wahba]{craven_wahba_1979}
Craven, P. and Wahba, G.
\newblock Estimating the correct degree of smoothing by the method of
  generalized cross-validation.
\newblock \emph{Numerische Mathematik}, 31:\penalty0 377--403, 1979.

\bibitem[Dobriban \& Sheng(2021)Dobriban and Sheng]{dobriban_sheng_2021}
Dobriban, E. and Sheng, Y.
\newblock Distributed linear regression by averaging.
\newblock \emph{The Annals of Statistics}, 49\penalty0 (2):\penalty0 918--943,
  2021.

\bibitem[Dobriban \& Wager(2018)Dobriban and Wager]{dobriban_wager_2018}
Dobriban, E. and Wager, S.
\newblock High-dimensional asymptotics of prediction: Ridge regression and
  classification.
\newblock \emph{The Annals of Statistics}, 46\penalty0 (1):\penalty0 247--279,
  2018.

\bibitem[Du et~al.(2022)Du, Cai, and Roeder]{du2022robust}
Du, J.-H., Cai, Z., and Roeder, K.
\newblock Robust probabilistic modeling for single-cell multimodal mosaic
  integration and imputation via scvaeit.
\newblock \emph{Proceedings of the National Academy of Sciences}, 119\penalty0
  (49):\penalty0 e2214414119, 2022.

\bibitem[d’Ascoli et~al.(2020)d’Ascoli, Refinetti, Biroli, and
  Krzakala]{d2020double}
d’Ascoli, S., Refinetti, M., Biroli, G., and Krzakala, F.
\newblock Double trouble in double descent: Bias and variance (s) in the lazy
  regime.
\newblock In \emph{International Conference on Machine Learning}, pp.\
  2280--2290. PMLR, 2020.

\bibitem[El~Karoui(2013)]{karoui_2013}
El~Karoui, N.
\newblock Asymptotic behavior of unregularized and ridge-regularized
  high-dimensional robust regression estimators: rigorous results.
\newblock \emph{arXiv preprint arXiv:1311.2445}, 2013.

\bibitem[El~Karoui(2018)]{elkaroui_2018}
El~Karoui, N.
\newblock On the impact of predictor geometry on the performance on
  high-dimensional ridge-regularized generalized robust regression estimators.
\newblock \emph{Probability Theory and Related Fields}, 170\penalty0
  (1):\penalty0 95--175, 2018.

\bibitem[Erd{\H{o}}s \& Yau(2017)Erd{\H{o}}s and Yau]{erdos_yau_2017}
Erd{\H{o}}s, L. and Yau, H.-T.
\newblock \emph{A Dynamical Approach to Random Matrix Theory}.
\newblock American Mathematical Society, 2017.

\bibitem[Friedman et~al.(2010)Friedman, Hastie, and
  Tibshirani]{friedman2010regularization}
Friedman, J., Hastie, T., and Tibshirani, R.
\newblock Regularization paths for generalized linear models via coordinate
  descent.
\newblock \emph{Journal of statistical software}, 33\penalty0 (1):\penalty0 1,
  2010.

\bibitem[Friedman \& Hall(2007)Friedman and Hall]{friedman_hall_2007}
Friedman, J.~H. and Hall, P.
\newblock On bagging and nonlinear estimation.
\newblock \emph{Journal of Statistical Planning and Inference}, 137\penalty0
  (3):\penalty0 669--683, 2007.

\bibitem[Geisser(1975)]{geisser_1975}
Geisser, S.
\newblock The predictive sample reuse method with applications.
\newblock \emph{Journal of the American statistical Association}, 70\penalty0
  (350):\penalty0 320--328, 1975.

\bibitem[Golub et~al.(1979)Golub, Heath, and Wahba]{golub_heath_wabha_1979}
Golub, G.~H., Heath, M., and Wahba, G.
\newblock Generalized cross-validation as a method for choosing a good ridge
  parameter.
\newblock \emph{Technometrics}, 21\penalty0 (2):\penalty0 215--223, 1979.

\bibitem[Greene \& Wellner(2017)Greene and Wellner]{greene2017exponential}
Greene, E. and Wellner, J.~A.
\newblock Exponential bounds for the hypergeometric distribution.
\newblock \emph{Bernoulli}, 23\penalty0 (3):\penalty0 1911, 2017.

\bibitem[Grenander \& Szeg{\"o}(1958)Grenander and
  Szeg{\"o}]{grenander1958toeplitz}
Grenander, U. and Szeg{\"o}, G.
\newblock \emph{Toeplitz Forms and Their Applications}.
\newblock University of California Press, 1958.
\newblock First edition.

\bibitem[Gut(2005)]{gut_2005}
Gut, A.
\newblock \emph{Probability: A Graduate Course}.
\newblock Springer, New York, 2005.

\bibitem[Gy{\"o}rfi et~al.(2006)Gy{\"o}rfi, Kohler, Krzyzak, and
  Walk]{gyorfi_kohler_krzyzak_walk_2006}
Gy{\"o}rfi, L., Kohler, M., Krzyzak, A., and Walk, H.
\newblock \emph{A Distribution-free Theory of Nonparametric Regression}.
\newblock Springer Science \& Business Media, 2006.

\bibitem[Hall \& Samworth(2005)Hall and Samworth]{hall_samworth_2005}
Hall, P. and Samworth, R.~J.
\newblock Properties of bagged nearest neighbour classifiers.
\newblock \emph{Journal of the Royal Statistical Society: Series B (Statistical
  Methodology)}, 67\penalty0 (3):\penalty0 363--379, 2005.

\bibitem[Hao et~al.(2021)Hao, Hao, Andersen-Nissen, Mauck~III, Zheng, Butler,
  Lee, Wilk, Darby, Zager, et~al.]{hao2021integrated}
Hao, Y., Hao, S., Andersen-Nissen, E., Mauck~III, W.~M., Zheng, S., Butler, A.,
  Lee, M.~J., Wilk, A.~J., Darby, C., Zager, M., et~al.
\newblock Integrated analysis of multimodal single-cell data.
\newblock \emph{Cell}, 2021.

\bibitem[Hastie(2020)]{hastie2020ridge}
Hastie, T.
\newblock Ridge regularization: An essential concept in data science.
\newblock \emph{Technometrics}, 62\penalty0 (4):\penalty0 426--433, 2020.

\bibitem[Hastie et~al.(2009)Hastie, Tibshirani, and
  Friedman]{friedman_hastie_tibshirani_2009}
Hastie, T., Tibshirani, R., and Friedman, J.
\newblock \emph{The Elements of Statistical Learning}.
\newblock Springer Series in Statistics, 2009.
\newblock Second edition.

\bibitem[Hastie et~al.(2022)Hastie, Montanari, Rosset, and
  Tibshirani]{hastie2022surprises}
Hastie, T., Montanari, A., Rosset, S., and Tibshirani, R.~J.
\newblock Surprises in high-dimensional ridgeless least squares interpolation.
\newblock \emph{The Annals of Statistics}, 50\penalty0 (2):\penalty0 949--986,
  2022.

\bibitem[Hoerl \& Kennard(1970{\natexlab{a}})Hoerl and
  Kennard]{hoerl_kennard_1970_1}
Hoerl, A.~E. and Kennard, R.~W.
\newblock Ridge regression: Biased estimation for nonorthogonal problems.
\newblock \emph{Technometrics}, 12\penalty0 (1):\penalty0 55--67,
  1970{\natexlab{a}}.

\bibitem[Hoerl \& Kennard(1970{\natexlab{b}})Hoerl and
  Kennard]{hoerl_kennard_1970_2}
Hoerl, A.~E. and Kennard, R.~W.
\newblock Ridge regression: applications to nonorthogonal problems.
\newblock \emph{Technometrics}, 12\penalty0 (1):\penalty0 69--82,
  1970{\natexlab{b}}.

\bibitem[Kale et~al.(2011)Kale, Kumar, and
  Vassilvitskii]{kale_kumar_vassilvitskii_2011}
Kale, S., Kumar, R., and Vassilvitskii, S.
\newblock Cross-validation and mean-square stability.
\newblock In \emph{In Proceedings of the Second Symposium on Innovations in
  Computer Science}, 2011.

\bibitem[Krogh \& Sollich(1997)Krogh and Sollich]{krogh1997statistical}
Krogh, A. and Sollich, P.
\newblock Statistical mechanics of ensemble learning.
\newblock \emph{Physical Review E}, 55\penalty0 (1):\penalty0 811, 1997.

\bibitem[Kumar et~al.(2013)Kumar, Lokshtanov, Vassilvitskii, and
  Vattani]{kumar_lokshtanov_vassilviskii_vattani_2013}
Kumar, R., Lokshtanov, D., Vassilvitskii, S., and Vattani, A.
\newblock Near-optimal bounds for cross-validation via loss stability.
\newblock In \emph{International Conference on Machine Learning}, 2013.

\bibitem[LeJeune et~al.(2020)LeJeune, Javadi, and
  Baraniuk]{lejeune2020implicit}
LeJeune, D., Javadi, H., and Baraniuk, R.
\newblock The implicit regularization of ordinary least squares ensembles.
\newblock In \emph{International Conference on Artificial Intelligence and
  Statistics}, 2020.

\bibitem[Li(1985)]{li_1985}
Li, K.-C.
\newblock From {Stein's} unbiased risk estimates to the method of generalized
  cross validation.
\newblock \emph{The Annals of Statistics}, pp.\  1352--1377, 1985.

\bibitem[Li(1986)]{li_1986}
Li, K.-C.
\newblock Asymptotic optimality of $ c_l $ and generalized cross-validation in
  ridge regression with application to spline smoothing.
\newblock \emph{The Annals of Statistics}, 14\penalty0 (3):\penalty0
  1101--1112, 1986.

\bibitem[Li(1987)]{li_1987}
Li, K.-C.
\newblock Asymptotic optimality for $ c_p, c_l $, cross-validation and
  generalized cross-validation: Discrete index set.
\newblock \emph{The Annals of Statistics}, 15\penalty0 (3):\penalty0 958--975,
  1987.

\bibitem[Loureiro et~al.(2022)Loureiro, Gerbelot, Refinetti, Sicuro, and
  Krzakala]{loureiro2022fluctuations}
Loureiro, B., Gerbelot, C., Refinetti, M., Sicuro, G., and Krzakala, F.
\newblock Fluctuations, bias, variance \& ensemble of learners: Exact
  asymptotics for convex losses in high-dimension.
\newblock In \emph{International Conference on Machine Learning}, pp.\
  14283--14314. PMLR, 2022.

\bibitem[Mei \& Montanari(2022)Mei and Montanari]{mei_montanari_2022}
Mei, S. and Montanari, A.
\newblock The generalization error of random features regression: Precise
  asymptotics and the double descent curve.
\newblock \emph{Communications on Pure and Applied Mathematics}, 75\penalty0
  (4):\penalty0 667--766, 2022.

\bibitem[Miolane \& Montanari(2021)Miolane and
  Montanari]{miolane_montanari_2021}
Miolane, L. and Montanari, A.
\newblock The distribution of the lasso: Uniform control over sparse balls and
  adaptive parameter tuning.
\newblock \emph{The Annals of Statistics}, 49\penalty0 (4):\penalty0
  2313--2335, 2021.

\bibitem[Muthukumar et~al.(2020)Muthukumar, Vodrahalli, Subramanian, and
  Sahai]{muthukumar_vodrahalli_subramanian_sahai_2020}
Muthukumar, V., Vodrahalli, K., Subramanian, V., and Sahai, A.
\newblock Harmless interpolation of noisy data in regression.
\newblock \emph{IEEE Journal on Selected Areas in Information Theory},
  1\penalty0 (1):\penalty0 67--83, 2020.

\bibitem[Obuchi \& Kabashima(2016)Obuchi and Kabashima]{obuchi_kabashima_2016}
Obuchi, T. and Kabashima, Y.
\newblock Cross validation in {LASSO} and its acceleration.
\newblock \emph{Journal of Statistical Mechanics: Theory and Experiment}, 2016.

\bibitem[Patil et~al.(2021)Patil, Wei, Rinaldo, and
  Tibshirani]{patil2021uniform}
Patil, P., Wei, Y., Rinaldo, A., and Tibshirani, R.
\newblock Uniform consistency of cross-validation estimators for
  high-dimensional ridge regression.
\newblock In \emph{International Conference on Artificial Intelligence and
  Statistics}. PMLR, 2021.

\bibitem[Patil et~al.(2022{\natexlab{a}})Patil, Du, and
  Kuchibhotla]{patil2022bagging}
Patil, P., Du, J.-H., and Kuchibhotla, A.~K.
\newblock Bagging in overparameterized learning: Risk characterization and risk
  monotonization.
\newblock \emph{arXiv preprint arXiv:2210.11445}, 2022{\natexlab{a}}.

\bibitem[Patil et~al.(2022{\natexlab{b}})Patil, Kuchibhotla, Wei, and
  Rinaldo]{patil2022mitigating}
Patil, P., Kuchibhotla, A.~K., Wei, Y., and Rinaldo, A.
\newblock Mitigating multiple descents: A model-agnostic framework for risk
  monotonization.
\newblock \emph{arXiv preprint arXiv:2205.12937}, 2022{\natexlab{b}}.

\bibitem[Patil et~al.(2022{\natexlab{c}})Patil, Rinaldo, and
  Tibshirani]{patil2022estimating}
Patil, P., Rinaldo, A., and Tibshirani, R.
\newblock Estimating functionals of the out-of-sample error distribution in
  high-dimensional ridge regression.
\newblock In \emph{International Conference on Artificial Intelligence and
  Statistics}. PMLR, 2022{\natexlab{c}}.

\bibitem[Rad \& Maleki(2020)Rad and Maleki]{rad_maleki_2020}
Rad, K.~R. and Maleki, A.
\newblock A scalable estimate of the out-of-sample prediction error via
  approximate leave-one-out cross-validation.
\newblock \emph{Journal of the Royal Statistical Society: Series B (Statistical
  Methodology)}, 82\penalty0 (4):\penalty0 965--996, 2020.

\bibitem[Rad et~al.(2020)Rad, Zhou, and Maleki]{rad_zhou_maleki_2020}
Rad, K.~R., Zhou, W., and Maleki, A.
\newblock Error bounds in estimating the out-of-sample prediction error using
  leave-one-out cross validation in high-dimensions.
\newblock In \emph{International Conference on Artificial Intelligence and
  Statistics}. PMLR, 2020.

\bibitem[Rubio \& Mestre(2011)Rubio and Mestre]{rubio_mestre_2011}
Rubio, F. and Mestre, X.
\newblock Spectral convergence for a general class of random matrices.
\newblock \emph{Statistics \& probability letters}, 81\penalty0 (5):\penalty0
  592--602, 2011.

\bibitem[Rudin(1976)]{rudin_1976}
Rudin, W.
\newblock \emph{Principles of Mathematical Analysis}.
\newblock McGraw-Hill New York, 1976.

\bibitem[Samworth(2012)]{samworth2012optimal}
Samworth, R.~J.
\newblock Optimal weighted nearest neighbour classifiers.
\newblock \emph{The Annals of Statistics}, 40\penalty0 (5):\penalty0
  2733--2763, 2012.

\bibitem[Sollich \& Krogh(1995)Sollich and Krogh]{sollich1995learning}
Sollich, P. and Krogh, A.
\newblock Learning with ensembles: How overfitting can be useful.
\newblock \emph{Advances in neural information processing systems}, 8, 1995.

\bibitem[Stone(1974)]{stone_1974}
Stone, M.
\newblock Cross-validatory choice and assessment of statistical predictions.
\newblock \emph{Journal of the Royal Statistical Society: Series B},
  36\penalty0 (2):\penalty0 111--133, 1974.

\bibitem[Stone(1977)]{stone_1977}
Stone, M.
\newblock Asymptotics for and against cross-validation.
\newblock \emph{Biometrika}, 64\penalty0 (1):\penalty0 29--35, 1977.

\bibitem[Sur et~al.(2019)Sur, Chen, and Cand{\`e}s]{sur_chen_candes_2019}
Sur, P., Chen, Y., and Cand{\`e}s, E.~J.
\newblock The likelihood ratio test in high-dimensional logistic regression is
  asymptotically a rescaled chi-square.
\newblock \emph{Probability Theory and Related Fields}, 175\penalty0
  (1):\penalty0 487--558, 2019.

\bibitem[Thrampoulidis et~al.(2015)Thrampoulidis, Oymak, and
  Hassibi]{thrampoulidis_oymak_hassibi_2015}
Thrampoulidis, C., Oymak, S., and Hassibi, B.
\newblock Regularized linear regression: A precise analysis of the estimation
  error.
\newblock In \emph{Conference on Learning Theory}, pp.\  1683--1709. PMLR,
  2015.

\bibitem[Thrampoulidis et~al.(2018)Thrampoulidis, Abbasi, and
  Hassibi]{thrampoulidis_abbasi_hassibi_2018}
Thrampoulidis, C., Abbasi, E., and Hassibi, B.
\newblock Precise error analysis of regularized {$M$}-estimators in high
  dimensions.
\newblock \emph{IEEE Transactions on Information Theory}, 64\penalty0
  (8):\penalty0 5592--5628, 2018.

\bibitem[Wang et~al.(2018)Wang, Zhou, Lu, Maleki, and
  Mirrokni]{wang_zhou_lu_maleki_mirrokni_2018}
Wang, S., Zhou, W., Lu, H., Maleki, A., and Mirrokni, V.
\newblock Approximate leave-one-out for fast parameter tuning in high
  dimensions.
\newblock \emph{arXiv preprint arXiv:1807.02694}, 2018.

\bibitem[Wasserman(2006)]{wasserman2006all}
Wasserman, L.
\newblock \emph{Olive Nonparametric Statistics}.
\newblock Springer, 2006.

\bibitem[Wei et~al.(2022)Wei, Hu, and Steinhardt]{wei_hu_steinhardt}
Wei, A., Hu, W., and Steinhardt, J.
\newblock More than a toy: Random matrix models predict how real-world neural
  representations generalize.
\newblock \emph{arXiv preprint arXiv:2203.06176}, 2022.

\bibitem[Xu et~al.(2019)Xu, Maleki, and Rad]{xu_maleki_rad_2019}
Xu, J., Maleki, A., and Rad, K.~R.
\newblock Consistent risk estimation in high-dimensional linear regression.
\newblock \emph{arXiv preprint arXiv:1902.01753}, 2019.

\bibitem[Zhang \& Yang(2015)Zhang and Yang]{zhang_yang_2015}
Zhang, Y. and Yang, Y.
\newblock Cross-validation for selecting a model selection procedure.
\newblock \emph{Journal of Econometrics}, 187\penalty0 (1):\penalty0 95--112,
  2015.

\end{thebibliography}
\bibliographystyle{icml2023}

\newcommand\blfootnote[1]{%
  \begingroup
  \renewcommand\thefootnote{}\footnote{#1}%
  \addtocounter{footnote}{-1}%
  \endgroup
}

\blfootnote{Title in \citet{wasserman2006all}
may seem like a typo, but it is not. 
If you have a moment to chuckle, peek at column 2 of page 266!}

\newpage
\appendix
\onecolumn
\begin{center}
\Large
{\bf \framebox{Appendix}}
\end{center}

This serves as an appendix to the paper
``Subsample Ridge Ensembles: Equivalences and Generalized Cross-Validation.''
Below we provide an outline for the appendix
along with a summary of the notation used in the main paper and the appendix.

\section*{Organization}
The content of the appendix is organized as follows.
\begin{itemize}
   
    \item \Cref{app:tune-k} presents proofs of results in \Cref{sec:equiv}.
    \begin{itemize}
        \item \Cref{app:subsec:full-ensemble} relates the ridge estimator in the full ensemble to the $M$-ensemble ridge estimator as the ensemble size $M$ tends to infinity.
        Specifically, we provide proof of the fact [mentioned in \Cref{sec:equiv}] that the estimator $\tbeta_{k,\infty}^\lambda$ as defined in \eqref{eq:def-full-ensemble}
        is almost surely equivalent to letting the ensemble size $M \to \infty$
        for the estimator $\tbeta_{k,M}^\lambda$ 
        as defined in \eqref{eq:def-M-ensemble}.
        The main ingredients are results in \Cref{sec:appendix-concerntration}.
        
        \item \Cref{app:sec:risk} gathers known results from \citet{patil2022bagging} in the form of \Cref{thm:ver-with-replacement} that characterize the asymptotic prediction risks of ridge ensembles used in the remaining sections. 

        \item \Cref{app:subsec:equiv} proves \Cref{thm:comparison_optimal_ridge}.
        The main ingredients are \Cref{thm:ver-with-replacement} and results in \Cref{sec:calculus_asymptotic_equivalents}.
        See \Cref{fig:schematic-thm:comparison_optimal_ridge}.

    \begin{figure}[!ht]
    \centering
    \begin{tikzpicture}[node distance=3cm]
    \node (ver) [block] {\Cref{thm:ver-with-replacement}};
    \node (calc) [block, right of=ver] {\Cref{sec:calculus_asymptotic_equivalents}};
    \node (ridge) [block, above of=ver, node distance=2.5cm] {\Cref{thm:comparison_optimal_ridge}};
    \draw [arrow] (ver) -- (ridge) ;
    \draw [arrow] (calc) -- (ridge) ;
    \end{tikzpicture}
    \caption{Schematic for the proof of \Cref{thm:comparison_optimal_ridge}.}
    \label{fig:schematic-thm:comparison_optimal_ridge}
    \end{figure}
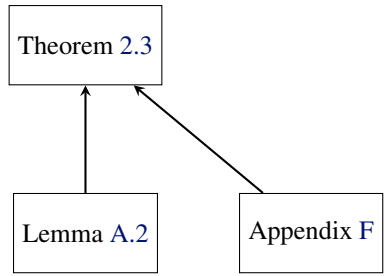

    \end{itemize}

    \item \Cref{app:gcv-consistency} presents proofs of results in \Cref{sec:gcv-consistency}.
    \begin{itemize}
        
        \item \Cref{app:subsec:thm:uniform-consistency-k} proves \Cref{thm:uniform-consistency-k}.
        The main ingredients are
        \Cref{lem:gcv-inf} (proved in \Cref{app:gcv-consistency})
        and results in \Cref{sec:appendix-concerntration}.
        The main ingredients that prove \Cref{lem:gcv-inf}
        are \Cref{prop:gcv} (proved in \Cref{app:proof-ridge})
        and \Cref{thm:ver-with-replacement}.
        See \Cref{fig:schematic-thm:uniform-consistency-k}.

    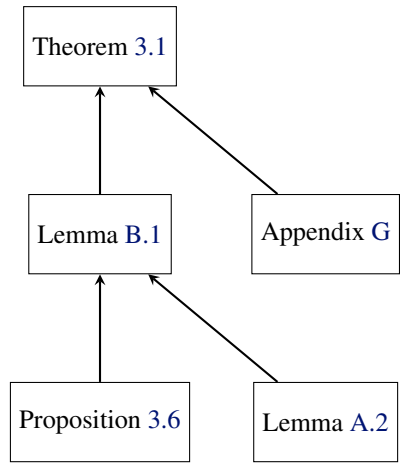
\begin{figure}[!ht]
    \centering
    \begin{tikzpicture}[node distance=3cm]
      \node (thm) [block] {\Cref{thm:uniform-consistency-k}};
      \node (prop1) [block, below of=thm, node distance=2.5cm] {\Cref{lem:gcv-inf}};
      \node (lem2) [block, right of=prop1] {\Cref{sec:appendix-concerntration}};
      \node (prop36) [block, below of=prop1, node distance=2.5cm] {\Cref{prop:gcv}};
      \node (lema2) [block, right of=prop36] {\Cref{thm:ver-with-replacement}};
      
      \draw [arrow] (prop1) -- (thm);
      \draw [arrow] (lem2) -- (thm);
      \draw [arrow] (prop36) -- (prop1);
      \draw [arrow] (lema2) -- (prop1);
    \end{tikzpicture}
    \caption{Schematic for the proof of \Cref{thm:uniform-consistency-k}.}
    \label{fig:schematic-thm:uniform-consistency-k}
    \end{figure}
        \item \Cref{app:subsec:gcv-opt-ridge} proves \Cref{cor:gcv-opt-ridge}.
        The main ingredient in \Cref{thm:uniform-consistency-k}.

        \item \Cref{app:subsec:inconsistency} proves \Cref{prop:inconsistency}.
        The main ingredients are \Cref{lem:gcv-den}
        and results in \Cref{sec:appendix-concerntration}.
        See \Cref{fig:prop:inconsistency}.

            \begin{figure}[!ht]
        \centering
        \begin{tikzpicture}[node distance=3cm]
        \node (prop33) [block] {\Cref{prop:inconsistency}};
        \node (lem34) [block, below of =prop33, node distance=2.5cm] {\Cref{lem:gcv-den}};
        \node (appG) [block, right of=lem34] {\Cref{sec:appendix-concerntration}};
        \draw [arrow] (lem34) -- (prop33) ;
        \draw [arrow] (appG) -- (prop33) ;
        \end{tikzpicture}
        \caption{Schematic for the proof of \Cref{prop:inconsistency}.}
        \label{fig:prop:inconsistency}
        \end{figure}
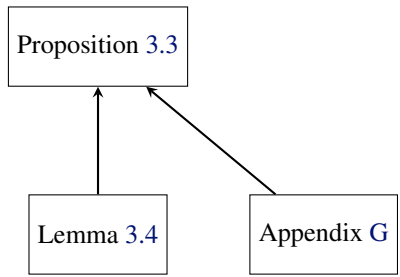
        
    \end{itemize}

    \item \Cref{app:sec:gcv-den} proves \Cref{lem:gcv-den}.
    The main ingredient is \Cref{lem:ridge-D} (proved in \Cref{app:sec:gcv-den}).
    See \Cref{fig:lem:gcv-den}.

    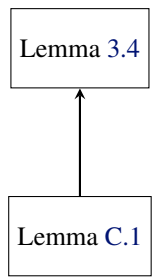
\begin{figure}[!ht]
    \centering
    \begin{tikzpicture}[node distance=3cm]
      \node (thm) [block] {\Cref{lem:gcv-den}};
      \node (prop1) [block, below of=thm, node distance=2.5cm] {\Cref{lem:ridge-D}};
      
      \draw [arrow] (prop1) -- (thm);
    \end{tikzpicture}
    \caption{Schematic for the proof of \Cref{lem:gcv-den}.}
    \label{fig:lem:gcv-den}
    \end{figure}

    \item \Cref{app:sec:gcv-num} 
    proves
    \Cref{lem:gcv-num}.
    The main ingredients are a series of lemmas, \Cref{lem:decomp-train-err,lem:conv-test-err,lem:conv-train-err} (proved in \Cref{app:sec:decomp-train-err,app:sec:conv-test-err,app:sec:conv-train-err}).
    These lemmas provide structural decompositions for the ensemble train error and obtain the limiting behaviors of the terms in the decompositions.
    See \Cref{fig:schematic-lem:gcv-num}.
    \begin{itemize}
        \item \Cref{app:sec:decomp-train-err} proves \Cref{lem:decomp-train-err} that shows a certain decomposition of the ensemble train error into in-sample train and out-of-sample test error components.

        \item \Cref{app:sec:conv-test-err} proves \Cref{lem:conv-test-err} on the convergence of out-of-sample test error components.
        The main ingredient is \Cref{thm:ver-with-replacement}.

        \item \Cref{app:sec:conv-train-err} proves \Cref{lem:conv-train-err} on the convergence of in-sample train error components.
        The main ingredients are \Cref{lem:ridge-conv-C0,lem:ridge-conv-V0} (proved in \Cref{app:sec:comp-concen}) and \Cref{lem:ridge-B0,lem:ridge-V0} (proved in \Cref{app:sec:comp-deter}).

        \item \Cref{app:sec:comp-concen} proves \Cref{lem:ridge-conv-C0,lem:ridge-conv-V0} on component concentrations of certain cross and variance terms arising in 
        the decompositions above.
        The main ingredients are results in \Cref{sec:appendix-concerntration}.

        \item \Cref{app:sec:comp-deter} proves \Cref{lem:ridge-B0,lem:ridge-V0} on component deterministic approximations for the concentrated bias and variance functionals
        in the steps above.
        The main ingredients are results in \Cref{sec:calculus_asymptotic_equivalents}.
    \end{itemize}

\begin{figure}[!ht]
\centering
    \begin{tikzpicture}[node distance=3cm]
    \node (lem35) [block] {\Cref{lem:gcv-num}};
    \node (lemd1) [block, below of=lem35, node distance=2.5cm] {\Cref{lem:decomp-train-err}};
    \node (lemd2) [block, right of=lemd1, node distance=3cm] {\Cref{lem:conv-test-err}};
    \node (lemd3) [block, right of=lemd2, node distance=3cm] {\Cref{lem:conv-train-err}};
    \node (lema2) [block, below of=lemd2, node distance=2.5cm] {\Cref{thm:ver-with-replacement}};
    \node (lemd45) [block, below of=lemd3, node distance=2.5cm] {\Cref{lem:ridge-conv-C0,lem:ridge-conv-V0}};
    \node (lemd67) [block, right of=lemd45, node distance=4cm] {\Cref{lem:ridge-B0,lem:ridge-V0}};
    \node (appG) [block, below of=lemd45, node distance=2.5cm] {\Cref{sec:appendix-concerntration}};
    \node (appF) [block, right of=appG, node distance=4cm] {\Cref{sec:calculus_asymptotic_equivalents}};
  \draw [arrow] (lemd1) -- (lem35);
  \draw [arrow] (lemd2) -- (lem35);
  \draw [arrow] (lemd3) -- (lem35);
  \draw [arrow] (lemd45) -- (lemd3);
  \draw [arrow] (lemd67) -- (lemd3);
  \draw [arrow] (appG) -- (lemd45);
  \draw [arrow] (appF) -- (lemd67);
  \draw [arrow] (lema2) -- (lemd2);
    \end{tikzpicture}
    \caption{Schematic for the proof of \Cref{lem:gcv-num}.}
    \label{fig:schematic-lem:gcv-num}
\end{figure}
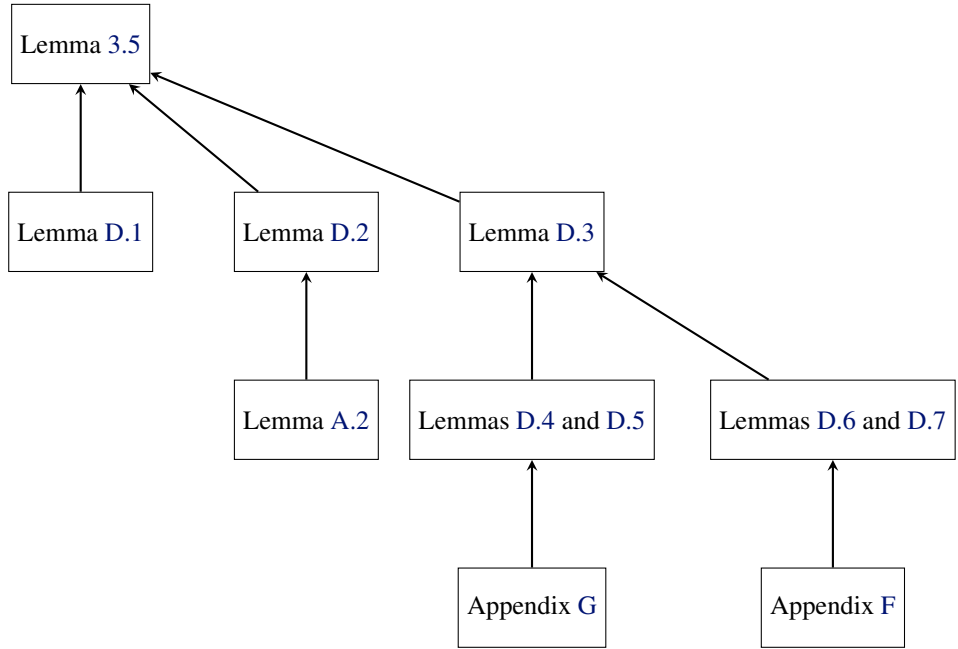

    \item \Cref{app:proof-ridge} 
    proves
    \Cref{prop:gcv}.
    The main ingredients are components proved in \Cref{lem:gcv-den} and \Cref{lem:gcv-num}, \Cref{prop:Rdet-ridge-infinity,prop:Rdet-lam-0} that handle certain boundary cases not covered by \Cref{lem:gcv-den} and \Cref{lem:gcv-num} (proved in \Cref{sec:proof-ridge,sec:proof-ridgeless}), and results in \Cref{sec:appendix-concerntration}.
    See \Cref{fig:prop:gcv}.
    \begin{itemize}
        \item 
        \Cref{sec:proof-ridge} proves \Cref{prop:Rdet-ridge-infinity} that considers the boundary case as the subsampling ratio $\phi_s \to \infty$ for ridge regression ($\lambda > 0$).
        \item
        \Cref{prop:Rdet-lam-0} proves \Cref{prop:Rdet-lam-0} that handles the limiting case of ridgeless regression ($\lambda = 0$),
        by justifying and taking a suitable limit as $\lambda \to 0^{+}$
        of the corresponding results for ridge regression.
    \end{itemize}

\begin{figure}[!ht]
\centering
    \begin{tikzpicture}[node distance=3cm]
    \node (prop36) [block] {\Cref{prop:gcv}};
    \node (lem34) [block, below of=prop36, node distance=2.5cm] {\Cref{lem:gcv-den}};
    \node (lem35) [block, right of=lem34, node distance=3cm] {\Cref{lem:gcv-num}};
    \node (leme1e2) [block, right of=lem35, node distance=4cm] {\Cref{prop:Rdet-ridge-infinity,prop:Rdet-lam-0}};
    \node (appG) [block, right of=leme1e2, node distance=5cm] {\Cref{sec:appendix-concerntration}};
  \draw [arrow] (lem34) -- (prop36);
  \draw [arrow] (lem35) -- (prop36);
  \draw [arrow] (leme1e2) -- (prop36);
  \draw [arrow] (appG) -- (prop36);
    \end{tikzpicture}
    \caption{Schematic for the proof of \Cref{prop:gcv}.}
    \label{fig:prop:gcv}
\end{figure}
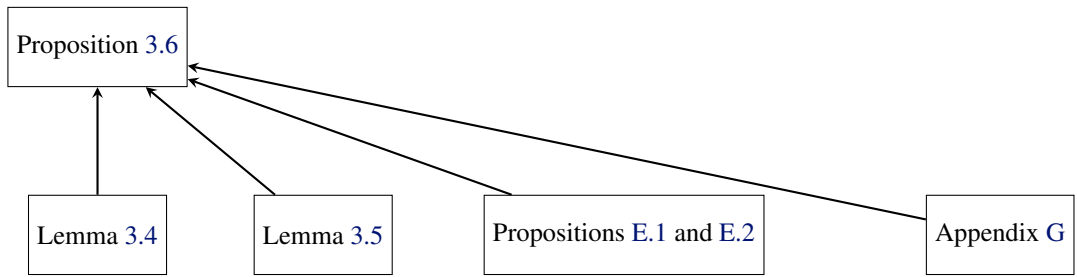

    \item \Cref{sec:calculus_asymptotic_equivalents} summarizes auxiliary asymptotic equivalency results used in the proofs throughout.
    \begin{itemize}
        \item \Cref{app:sec:preliminary-background} provides background on the notion of asymptotic matrix equivalents and various calculus rules that such notion of equivalency obeys.
        \item \Cref{append:det-equi-resol} gathers some known asymptotic matrix equivalents and derives some novel asymptotic matrix equivalents that arise in our analysis.
        \item \Cref{app:sec:analytic-properties} gathers various known analytic properties of certain fixed-point equations and proves some additional properties that arise in our analysis.
    \end{itemize}

    \item \Cref{sec:appendix-concerntration} collects several helper concentration results used in the proofs throughout.

    \begin{itemize}
        \item \Cref{sec:size-intersection} provides lemmas deriving the asymptotic proportion of shared observations when subsampling.
        
        \item \Cref{sec:concen-linform-quadfrom} provides lemmas establishing concentrations for linear and quadratic forms of random vectors.
        
        \item \Cref{sec:cesaro-mean-max} provides lemmas for lifting original convergences to converges of Ce\`saro-type mean and max for triangular arrays.
    \end{itemize}

    \item \Cref{app:correct} discusses the bias correction of GCV for finite ensembles [mentioned in \Cref{sec:discussion}].

    \item \Cref{app:numerical-details} describes additional numerical details for experiments [mentioned in \Cref{sec:gcv-consistency}].
\end{itemize}

\section*{Notation}
An overview of some general notation used in the main paper and the appendix is as follows.

\begin{enumerate}
    \item \underline{General notation}: 
    We denote scalars in non-bold lower or upper case, vectors in bold lower case, and matrices in bold upper case.
    For a real number $x$, $(x)_{+}$ denotes its positive part, $\lfloor x \rfloor$ its floor, and $\lceil x \rceil$ its ceiling.
    For a vector $\ba$, $\| \ba \|_{2}$ denotes its $\ell_2$ norm.
    For a pair of vectors $\bb$ and $\bc$, $\langle \bb, \bc \rangle$ denotes their inner product.
    For an event $E$, $\ind_E$ denotes the associated indicator random variable.
    We indicate convergence in probability using ``$\pto$", almost sure convergence using ``$\asto$", and weak convergence using ``$\dto$".

    \item \underline{Set notation}: We denote sets using calligraphic letters. 
    We use blackboard letters to denote some special sets: $\NN$ denotes the set of positive integers, $\RR$ denotes the set of real numbers, $\RR_{+}$ indicates the set of non-negative real numbers, $\CC$ represents the set of complex numbers, $\CC^{+}$ represents the set of complex numbers with positive imaginary parts, and $\CC_{-}$ indicates the set of complex numbers with negative imaginary parts.
    For a natural number $n$, we use $[n]$ to denote the set $\{ 1, \dots, n \}$.

    \item \underline{Matrix notation}: For a matrix $\bA \in \RR^{n \times p}$, $\bA^\top \in \RR^{p \times n}$ denotes its transpose, and $\bA^{+} \in \RR^{p \times n}$ denotes its Moore-Penrose inverse.
    For a square matrix $\bB \in \RR^{p \times p}$, $\tr[\bB]$ denotes its trace, and $\bB^{-1} \in \RR^{p \times p}$ denotes its inverse (assuming it is invertible).
    For a positive semidefinite matrix $\bC$, $\bC^{1/2}$ denotes its principal square root.
    A $p \times p$ identity matrix is denoted $\bI_p$, or simply by $\bI$, whenever it is clear from the context.

    For a real matrix $\bD$, its operator norm (or spectral norm) concerning $\ell_2$ vector norm is denoted by $\| \bD \|_{\mathrm{op}}$, and its trace norm (or nuclear norm) is denoted by $\| \bD \|_{\mathrm{tr}}$ (recall that $\| \bD \|_{\mathrm{tr}} = \tr[(\bD^\top \bD)^{1/2}]$).
    For a positive semidefinite matrix $\bU \in \RR^{p \times p}$ with eigenvalue decomposition $\bU = \bQ \bR \bQ^{-1}$ for an orthonormal matrix $\bQ \in \RR^{p \times p}$ and a diagonal matrix $\bR \in \RR^{p \times p}$ with non-negative entries, and a function $f : \RR_{+} \to \RR_{+}$, we denote by $f(\bU)$ the $p \times p$ positive semidefinite matrix $\bQ f(\bR) \bQ^{-1}$. Here $f(\bR)$ is a $p \times p$ diagonal matrix obtained by applying the function $f$ to each diagonal entry of $\bR$.

    For symmetric matrices $\bV$ and $\bW$, $\bV \preceq \bW$ denotes the Loewner ordering.
    For sequences of matrices $\bY_n$ and $\bZ_n$, $\bY_n \asympequi \bZ_n$ denotes a particular notion of asymptotic equivalence. 
    See \Cref{sec:calculus_asymptotic_equivalents} for more details.
\end{enumerate}

Finally, in the subsequent sections, we will be proving the results where $n,k,p$ are sequence of integers $\{n_m\}_{m=1}^{\infty}$, $\{k_m\}_{m=1}^{\infty}$, $\{p_m\}_{m=1}^{\infty}$.
One can also view $k$ and $p$ as sequences $\{ k_n \}$ and $\{ p_n \}$ that are indexed by $n$.
We will omit the subscripts (denoting the index sets) for notational simplicity whenever it is clear from the context.

\section{Proofs of results in \Cref{sec:equiv}}\label{app:tune-k}

\subsection[Full-ensemble versus limiting M-ensemble]{Full-ensemble versus limiting $M$-ensemble}\label{app:subsec:full-ensemble}
    \begin{lemma}[Almost sure equivalence of full-ensemble and limiting $M$-ensemble]\label{lem:full-ensemble}
        For $k,n$ fixed, for the ensemble estimator defined in \eqref{eq:def-M-ensemble},
        it holds that
        $$\tbeta^{\lambda}_{k,M}(\cD_n;\{I_{\ell}\}_{\ell=1}^M) \asto \EE[\hbeta^{\lambda}_{k}(\cD_{I}) \mid \cD_n]=\frac{1}{|\cI_k|} \sum_{I \in \cI_k} \hbeta^{\lambda}_{k}(\cD_{I}),$$
        as $M\rightarrow\infty$.
    \end{lemma}
    \begin{proof}[Proof of \Cref{lem:full-ensemble}]
        Note that for $k,n$ fixed, the cardinality of $\cI_k$ is $\binom{n}{k}$. Thus,
        we have
        \begin{align*}
            \tbeta^{\lambda}_{k,M}(\cD_n;\{I_{\ell}\}_{\ell=1}^M) = \sum_{I\in\cI_k} \frac{n_{M,I}}{M}\hbeta^{\lambda}_{k}(\cD_I)
        \end{align*}
        for random variables $n_{M,i}$'s.
        Since when sampling with replacement $n_{M,I}\sim \text{Binomial}(M, 1/\binom{n}{k})$ with mean $M/\binom{n}{k}$, from the strong law of large numbers, we have that as $M\rightarrow\infty$,
        \begin{align}
            \frac{n_{M,I}}{M} \asto \frac{1}{\binom{n}{k}},\qquad \forall\ I\in \cI_{k}.
        \end{align}
        For sampling without replacement, $n_{M,I}\sim \text{Hypergeometric}(M, 1, \binom{n}{k})$ for $M\leq \binom{n}{k}$ (see \Cref{def:hypergeometric}) with mean $M/\binom{n}{k}$.
        When $M= \binom{n}{k}$, ${n_{M,I}}/{M}={1}/{\binom{n}{k}}$.        
        In both cases,
        we have
        \begin{align*}
            \tbeta^{\lambda}_{k,M}(\cD_n;\{I_{\ell}\}_{\ell=1}^{\infty}) := \lim_{M\rightarrow\infty} \tbeta^{\lambda}_{k,M}(\cD_n;\{I_{\ell}\}_{\ell=1}^M) \overset{\as}{=} \frac{1}{\binom{n}{k}}\sum_{I\in\cI_k}\hbeta^{\lambda}_{k}(\cD_I),
        \end{align*}
        which concludes the proof.
    \end{proof}

\subsection{Risk characterization of ridge ensembles}\label{app:sec:risk}

    In analyzing ridge ensembles under proportional asymptotics, we often encounter the solution to a fixed-point equation.
    For every $\lambda > 0$ and $\theta > 0$, let $v(-\lambda; \theta)$ denote the unique nonnegative solution to the following fixed-point equation: 
   \begin{align}
       \label{eq:v_ridge}v(-\lambda;\theta)^{-1} &= \lambda + \theta\int r(1 + v(-\lambda;\theta)r)^{-1}\rd H(r).
    \end{align}
    When $\lambda = 0$, we define $v(0; \theta):=\lim_{\lambda \to 0^{+}} v(-\lambda; \theta)$ for $\theta > 1$ and $+\infty$ otherwise.
    
    Previous work has featured such fixed-point equations. 
    For instance, see \citet{dobriban_wager_2018,hastie2022surprises,mei_montanari_2022} for the context of ridge regression.
    In the context of $M$-estimators, see \citet{karoui_2013,elkaroui_2018,thrampoulidis_oymak_hassibi_2015,thrampoulidis_abbasi_hassibi_2018,sur_chen_candes_2019,miolane_montanari_2021}, among others. The uniqueness of the solution to the fixed-point equation \eqref{eq:v_ridge} is affirmed by \citet[Lemma S.6.14]{patil2022mitigating}.

    We then introduce the nonnegative constants $\tv(-\lambda;\vartheta,\theta)$, and $\tc(-\lambda;\theta)$ based on the following equations:
    \begin{align}
    \label{eq:tv_tc_ridge}
    \tv(-\lambda;\vartheta,\theta)
    &= \ddfrac{\vartheta\int r^2 (1+ v(-\lambda; \theta)r)^{-2}\rd H(r)}{v(-\lambda; \theta)^{-2}-\vartheta \int r^2 (1+ v(-\lambda; \theta)r)^{-2}\rd H(r)}, ~ \text{ and } ~
    \tc(-\lambda;\theta)=\int r (1+v(-\lambda; \theta))r)^{-2} \rd G(r).
    \end{align}
    
    \begin{lemma}[Risk characterization of ridge ensembles, adapted from \citet{patil2022bagging}]\label{thm:ver-with-replacement}
        Suppose 
        Assumptions \ref{asm:rmt-feat}-\ref{asm:lin-mod} hold for the dataset $\cD_n$.
        Then,
        as $k,n,p\rightarrow\infty$
        such that
        $p/n\rightarrow\phi\in(0,\infty)$ and $p/k\rightarrow\phi_s\in[\phi,\infty]$
        (and $\phi_s \neq 1$ if $\lambda = 0$),
        there exist deterministic functions $\RlamM$ for $M \in \NN$, such that for $I_1, \ldots, I_M \overset{\SRS}{\sim} \mathcal{I}_k$, 
        \begin{align}
            \label{eq:ridge-wr-guarantee}
            \sup_{M\in\NN}| R(\tfWR{M}{\cI_k}; \cD_n,\{I_{\ell}\}_{\ell = 1}^{M}) - \RlamM| &\pto 0.
        \end{align}
        Furthermore, the function $\RlamM$
        decomposes as
        \begin{align}
            \RlamM = \sigma^2 + \BlamM{M}{\phi}   + \VlamM{M}{\phi} , \label{eq:risk-det-with-replacement}
        \end{align}
        where 
        the bias and variance terms are given by
        \begin{align}
            \BlamM{M}{\phi}
            &= M^{-1} B_{\lambda}(\phi_s, \phi_s)
            + (1 -  M^{-1}) B_{\lambda}(\phi,\phi_s),
            \label{eq:bias-component-decomposition}\\
            \VlamM{M}{\phi}
            &= M^{-1} V_{\lambda}(\phi_s, \phi_s)
            + (1 - M^{-1})
            V_{\lambda}(\phi,\phi_s)
            \label{eq:var-component-decomposition},
        \end{align}
        and the functions $B_{\lambda}(\cdot, \cdot)$
        and $V_{\lambda}(\cdot, \cdot)$ are defined as
        \begin{align}
            B_{\lambda}(\vartheta, \theta)
            = \rho^2 (1+\tv(-\lambda; \vartheta, \theta)) \tc(-\lambda;\theta),\qquad
            V_{\lambda}(\vartheta, \theta)
            =\sigma^2\tv (-\lambda; \vartheta, \theta),\qquad \theta \in (0, \infty], \, \vartheta \le \theta. \label{eq:Blam_V_lam}
        \end{align}
    \end{lemma}

\subsection{Proof of \Cref{thm:comparison_optimal_ridge}}\label{app:subsec:equiv}
    \begin{proof}[Proof of \Cref{thm:comparison_optimal_ridge}]
        Define $\phi_s^*(\phi):=\argmin_{\phi_s \ge \phi} \sR_{\lambda,\infty}(\phi, \phi_s)$ and $\lambda^*(\phi):=\argmin_{\lambda\geq 0}\RlamMe{1}{\phi}{\phi}$.
        For simplicity, we will write $\phi_s^*$ and $\lambda^*$ and split the proof into different cases.

        \paragraph{Part (1)}
        \underline{Case of $\SNR>0$ ($\rho^2>0,\sigma^2>0$)}:
        
        From \citet[Proposition 5.7]{patil2022bagging} we have that $\phi_s^*\in(\phi\vee 1,\infty)$.
        From \Cref{lem:fixed-point-v-properties}~\ref{lem:fixed-point-v-properties-item-v-properties}, the function $\phi_s\mapsto v(0;\phi_s)$ is strictly decreasing over $\phi_s\in[1,\infty]$ with range
        \begin{align*}
            v(0;\phi_s\vee 1)=\begin{cases}  
            v(0;\phi_s),&\phi\in(1,\infty)\\
            \lim_{\phi_s\rightarrow1^+}v(0;\phi_s)=+\infty,&\phi\in(0,1]
            \end{cases},\qquad v(0;+\infty):=\lim_{\phi_s\rightarrow+\infty}v(0;\phi_s)=0.
        \end{align*}
        From \Cref{lem:fixed-point-v-lambda-properties}~\ref{lem:fixed-point-v-lambda-properties-item-monotonicity}, the function $\lambda\mapsto v(-\lambda;\phi)$ is strictly decreasing over $\lambda\in[0,\infty]$ with range
        \begin{align*}
            v(0;\phi\vee 1)=\begin{cases}  
            v(0;\phi),&\phi\in(1,\infty)\\
            \lim_{\lambda\rightarrow0^+}v(-\lambda;\phi)=+\infty,&\phi\in(0,1]
            \end{cases}
            ,\qquad v(-\infty;\phi):=\lim_{\lambda\rightarrow+\infty}v(-\lambda;\phi)=0.
        \end{align*}
        By the intermediate value theorem, there exists unique $\lambda_0\in(0,\infty)$ such that $v(-\lambda_0;\phi)=v(0;\phi_s^*)$.
        Then we also have $\tc(-\lambda_0;\phi)=\tc(0;\phi_s^*)$ and $\tv(-\lambda_0;\phi,\phi) = \tv(0;\phi_s^*)$.
        Substituting this into the optimal ensemble risk, we have
        \begin{align*}
            \min_{\phi_s\geq \phi}\RzeroM{\infty}{\phi} &=\RlamM[0][\infty][\phi][\phi_s^*]\\
            &=(\sigma^2+\rho^2  \tc(0;\phi_s^*) )(1+\tv(0; \phi, \phi_s^*))\\
            &=(\sigma^2+\rho^2  \tc(-\lambda_0;\phi) )(1+\tv(-\lambda_0; \phi, \phi))\\
            &= \RlamM[\lambda_0][1][\phi][\phi]\\
            &\leq \min_{\lambda\geq 0}\RlamMe{1}{\phi}{\phi}.
        \end{align*}
        On the other hand, there exists unique $\phi_0\in[1,\infty)$ such that $v(-\lambda^*;\phi)=v(0;\phi_0)$, and thus,
        we have
        \begin{align*}
            \min_{\lambda\geq 0}\RlamMe{1}{\phi}{\phi}&= \RlamM[\lambda^*][1][\phi][\phi]\\
            &=(\sigma^2+\rho^2  \tc(-\lambda^*;\phi) )(1+\tv(-\lambda^*; \phi, \phi))\\
            &=(\sigma^2+\rho^2  \tc(0;\phi_0) )(1+\tv(0; \phi, \phi_0))\\
            &= \RlamM[0][\infty][\phi][\phi_0] \\
            &\leq \min_{\phi_s\geq \phi}\RzeroM{\infty}{\phi}.
        \end{align*}
        Combining the above two inequalities, we have that $\min_{\phi_s\geq \phi}\RzeroM{\infty}{\phi}= \min_{\lambda\geq 0}\RlamMe{1}{\phi}{\phi}.$

        \paragraph{Part (2)}
        \underline{Case of $\SNR=0$ ($\rho^2=0,\sigma^2>0$)}:
        
        From \citet[Proposition 5.7]{patil2022bagging} we have that $\phi_s^*=+\infty$, which implies that $v(0;\phi_s^*)=0$.
        Then, from \Cref{lem:fixed-point-v-properties}~\ref{lem:fixed-point-v-properties-item-v-properties} we have $v(0;+\infty):=\lim_{\phi_s\rightarrow+\infty}v(0;\phi_s)=0$,
        \begin{align*}
            \lim_{\phi_s\rightarrow+\infty}\tv(0; \phi, +\infty) &= \lim_{\phi_s\rightarrow+\infty} \ddfrac{\phi \int{r^2}{(1+ v(0; \phi_s)r)^{-2}}\rd H(r)}{v(0; \phi_s)^{-2}-\phi \int{r^2}{(1+ v(0; \phi_s)r)^{-2}}\rd H(r)} \\
            &=\lim_{\phi_s\rightarrow+\infty} \ddfrac{\phi \int{(v(0; \phi_s)r)^2}{(1+ v(0; \phi_s)r)^{-2}}\rd H(r)}{1-\phi \int{(v(0; \phi_s)r)^2}{(1+ v(0; \phi_s)r)^{-2}}\rd H(r)} \\
            &= \ddfrac{\phi \int{(v(0; +\infty)r)^2}{(1+ v(0; +\infty)r)^{-2}}\rd H(r)}{1-\phi \int{(v(0; +\infty)r)^2}{(1+ v(0; +\infty)r)^{-2}}\rd H(r)} \\
            &=0,
        \end{align*}
        and thus,
        \begin{align*}
            \min_{\phi_s\geq \phi}\RzeroM{\infty}{\phi} &=\RlamM[0][\infty][\phi][\infty] = \sigma^2(1+\tv(0; \phi, +\infty)) = \sigma^2.
        \end{align*}
        On the other hand, 
        \begin{align*}
            \min_{\lambda\geq 0}\RlamMe{1}{\phi}{\phi}&= \RlamM[\lambda^*][1][\phi][\phi]=\sigma^2\tv(-\lambda^*; \phi, \phi) \geq \sigma^2
        \end{align*}
        where the equality holds when $\lambda^*=+\infty$ because $\tv(-\lambda^*; \phi, \phi)\geq 0$ from \Cref{lem:ridge-fixed-point-v-properties}~\ref{lem:fixed-point-v-properties-item-vb-properties}.
        Thus, the optimal parameters to the two optimization problems are given by $\phi_s^*=\lambda^*=+\infty$, with $v(0;\phi_s^*)=v(-\lambda^*;\phi)=0$.

        \paragraph{Part (3)} 
        \underline{Case of $\SNR=\infty$ ($\rho^2>0,\sigma^2=0$)}:
        
        When $\phi\leq 1$, from \citet[Proposition 5.7]{patil2022bagging} we have that any $\phi_s^*\in[\phi,1]$ minimizes $\min_{\phi_s\geq \phi}\RzeroM{\infty}{\phi}$ and the minimum is $0$, which is also the smallest possible prediction risk.
        As $\RlamMe{1}{\phi}{\phi}=0$ for $\lambda=0$, the conclusion still holds.
        
        When $\phi\in(1,\infty)$, we know that $\phi_s^*\in(1,\infty)$ from \citet[Proposition 5.7]{patil2022bagging}.
        Analogous to Part (1), we have that $\min_{\phi_s\geq \phi}\RzeroM{\infty}{\phi}= \min_{\lambda\geq 0}\RlamMe{1}{\phi}{\phi}.$

        \paragraph{Part (4)}
        \underline{Relationship between $\phi^*$ and $\lambda^*$}:
        
        Each pair of the optimal solution $(\phi^*,\lambda^*)$ satisfies that $v(0;\phi_s^*) = v(-\lambda^*;\phi)=:v^*$, where $v(0;\phi_s^*)$ and $v(-\lambda^*;\phi)$ are non-negative solutions to the following fixed-point equations:
        \begin{align*}
            \frac{1}{v(0;\phi_s^*)} &= \phi_s^* \int \frac{r}{1+v(0;\phi_s^*)r} \rd H(r),\qquad 
            \frac{1}{v(-\lambda^*;\phi)} = \lambda^* + \phi \int \frac{r}{1+v(-\lambda^*;\phi)r} \rd H(r)
        \end{align*}
        
        From the previous parts, if $\SNR=0$, $\lambda^*=\phi_s^*=+\infty$ and $v^*= 0$.
        Otherwise, we have
        \begin{align*}
            \frac{1}{v^*} &= \phi_s^* \int \frac{r}{1+v^*r} \rd H(r) = \lambda^* + \phi \int \frac{r}{1+v^*r} \rd H(r),
        \end{align*}
        which yields that
        \begin{align*}
            \lambda^* &= (\phi_s^*-\phi) \int \frac{r}{1+v^*r} \rd H(r).
        \end{align*}

        \paragraph{Part (5)} 
        \underline{Individual and joint optimization}:
        
        Note that from \Cref{lem:ridge-fixed-point-v-properties}~\ref{lem:ridge-fixed-point-v-properties-item-v-properties} and \Cref{lem:fixed-point-v-properties}~\ref{lem:fixed-point-v-properties-item-v-properties}, the function $\phi_s\mapsto v(-\lambda;\phi_s)$ is decreasing with the range $[0,\lambda^{-1}]$ for $\lambda\in[0,\infty]$.
        Then the function $(\lambda,\phi_s)\mapsto v(-\lambda;\phi_s)$ has the range $[0, +\infty]$, which is the same as $v(0;\phi_s)$.
        It follows that $\min_{\phi_s\geq \phi}\RzeroM{\infty}{\phi}= \min_{\phi_s\geq\phi, \lambda\geq 0}\RlamMe{1}{\phi}{\phi}$ by the analogous argument in Part (1)-(3).

        When $\lambda^*=0$, the curve reduces to a singleton, which is a trivial case.
        When $\lambda^*>0$, for any $t\in[0,\lambda^*]$, let $\lambda = \lambda^* - t$ and $\phi_s = \phi + t(\phi_s^* - \phi)/\lambda^*$.
        Note that
        \begin{align*}
           \frac{1}{v(\lambda;\phi_s)} &=  \lambda+\phi_s \int \frac{r}{1+v^*r} \rd H(r) \\
           & = \lambda^* - t + (\phi + t(\phi_s^* - \phi)/\lambda^*) \int \frac{r}{1+v(0;\phi_s^*)r} \rd H(r) \\
            &= \lambda^* + \phi \int \frac{r}{1+v(0;\phi_s^*)r} \rd H(r) +\frac{t}{\lambda^*} (\phi_s^* - \phi - \lambda) \int \frac{r}{1+v(0;\phi_s^*)r} \rd H(r)\\
            &=  \frac{1}{v^*} + \frac{t}{\lambda^*} \left(\frac{1}{v^*} - \frac{1}{v^*}\right) \\
            &= \frac{1}{v^*},
        \end{align*}
        which implies that $v(\lambda;\phi_s) = v^*$. Then, we have
        \begin{align*}
            \tc(-\lambda;\phi_s) &= \int  \frac{r}{(1+v(-\lambda; \phi_s))r)^{2}} \rd G(r) = \tc(-\lambda^*;\phi) = \tc(0;\phi_s^*). 
        \end{align*}
        and
        \begin{align*}
            \tv(-\lambda;\phi,\phi_s) &= \ddfrac{\phi\int r^2 (1+ v(-\lambda; \phi_s)r)^{-2}\rd H(r)}{v(-\lambda; \phi_s)^{-2}-\phi \int r^2 (1+ v(-\lambda; \phi_s)r)^{-2}\rd H(r)} = \tv(-\lambda^*;\phi,\phi) = \tv(0;\phi,\phi_s^*).
        \end{align*}
        It then follows that $\RlamM[\lambda][\infty] = \RlamM[\lambda^*][\infty][\phi][\phi] =  \RzeroMe{\infty}{\phi}{\phi_s^*}$, which completes the proof for \Cref{thm:comparison_optimal_ridge}.

        \paragraph{Part (6)$^*$}        
        \underline{Extension of risk equivalence}:
        
        Here we extend the results in \Cref{thm:comparison_optimal_ridge} to a more general equivalence of $(\lambda,\phi_s)$,
        as indicated following \Cref{thm:comparison_optimal_ridge} 
        towards the end of \Cref{subsec:connect-subsample-ridge}.

        For any $\bar{\phi}_s\in[\phi,+\infty]$, let $\bar{\lambda} = (\bar{\phi}_s-\phi) \int r(1+v(0;\phi_s)r)^{-1} \rd H(r)\geq 0. $
        Then,
        we have
        \begin{align*}
            \frac{1}{v(0;\bar{\phi}_s)} &= \bar{\phi}_s \int \frac{r}{1+v(0;\bar{\phi}_s)r} \rd H(r) = \bar{\lambda} + \phi \int \frac{r}{1+v(0;\bar{\phi}_s)r} \rd H(r),
        \end{align*}
        It follows that $v(-\bar{\lambda};\phi) = v(0;\bar{\phi}_s)$, and consequently, $\RzeroMe{\infty}{\phi}{\bar{\phi}_s}= \RlamM[\bar{\lambda}][\infty][\phi][\phi]$.
        
    \end{proof}

\section{Proofs of results in \Cref{sec:gcv-consistency}}\label{app:gcv-consistency}
    
    \subsection{Proof of \Cref{thm:uniform-consistency-k}}\label{app:subsec:thm:uniform-consistency-k}

    To prove \Cref{thm:uniform-consistency-k},
    we first prove pointwise convergence (over $k$ and $\lambda$) 
    as stated in \Cref{lem:gcv-inf}, which is based on \Cref{prop:gcv} proved in \Cref{app:proof-ridge}.

    \begin{lemma}[Consistency of GCV in full ensemble]\label{lem:gcv-inf}
        Under Assumptions \ref{asm:rmt-feat}-\ref{asm:lin-mod}, as $k,n,p\rightarrow\infty$, $p/n\rightarrow \phi\in(0,\infty)$ and $p/k\rightarrow\phi_s\in[\phi,+\infty]$, for $\lambda\geq 0$, it holds that
        \begin{align}
            |\gcv_k^{\lambda} - R_{k,\infty}^{\lambda}| &\asto 0.
        \end{align}
    \end{lemma}
    \begin{proof}[Proof of \Cref{lem:gcv-inf}]
        We will first show that proof for $\lambda>0$ and $\phi_s<\infty$ and then extend the results to these boundary cases.
        
        Recall that from \Cref{prop:gcv},
        we have
        \begin{align*}
             \gcvdet =
             \ddfrac{\frac{2\phi(2\phi_s-\phi)}{\phi_s^2}\RlamMtr[\lambda][2] + \frac{2(\phi_s-\phi)^2}{\phi_s^2}\RlamM[\lambda][2]
            - \frac{\phi}{\phi_s}\RlamMtr[\lambda][1] - \frac{\phi_s-\phi}{\phi_s}\RlamM[\lambda][1]}{\Ddet}.
        \end{align*}
        We next simplify the expression of the numerator:
        \begin{align*}
            &\frac{2\phi(2\phi_s-\phi)}{\phi_s^2}\RlamMtr[\lambda][2] + \frac{2(\phi_s-\phi)^2}{\phi_s^2}\RlamM[\lambda][2] - \frac{\phi}{\phi_s}\RlamMtr[\lambda][1] - \frac{\phi_s-\phi}{\phi_s}\RlamM[\lambda][1]\\
            &=  \frac{\phi(\phi_s-\phi)}{\phi_s^2}\RlamM[\lambda][1] + \Ddet[][\lambda][\phi_s]  \left(\frac{\phi}{\phi_s}\RlamM[\lambda][1] + \left(\frac{2\phi(\phi_s-\phi)}{\phi_s^2} \frac{1}{\lambda v(-\lambda;\phi_s)}+ \frac{\phi^2}{\phi_s^2}\right)\RlamM[\lambda][\infty]\right) \\
            &\qquad + \frac{2(\phi_s-\phi)^2}{\phi_s^2}\RlamM[\lambda][2] - \frac{\phi}{\phi_s}\Ddet[][\lambda][\phi_s] \RlamM[\lambda][1] - \frac{\phi_s-\phi}{\phi_s}\RlamM[\lambda][1]\\
            &= \left(\frac{2\phi(\phi_s-\phi)}{\phi_s^2} \lambda v(-\lambda;\phi_s) + \frac{\phi^2}{\phi_s^2}\lambda^2 v(-\lambda;\phi_s)^2\right)\RlamM[\lambda][\infty] + \frac{2(\phi_s-\phi)^2}{\phi_s^2}(\RlamM[\lambda][2] - \RlamM[\lambda][1])\\
            &=\left(\frac{(\phi_s-\phi)^2}{\phi_s^2}+\frac{2\phi(\phi_s-\phi)}{\phi_s^2} \lambda v(-\lambda;\phi_s) + \frac{\phi^2}{\phi_s^2}\lambda^2 v(-\lambda;\phi_s)^2\right)\RlamM[\lambda][\infty]\\
            &= \Ddet  \RlamM[\lambda][\infty].
        \end{align*}
        Then, it follows that $\gcvdet = \RlamM[\lambda][\infty]$.
        From \Cref{thm:ver-with-replacement} and \Cref{prop:gcv}, we have that $\gcv_k^{\lambda} \asto 
 \gcvdet$ and $R_{k,\infty}^{\lambda}\asto \RlamM[\lambda][\infty]$, which finishes the proof.
    \end{proof}

    We are now ready to prove \Cref{thm:uniform-consistency-k}.
    
    \begin{proof}[Proof of \Cref{thm:uniform-consistency-k}]
        Let $R_{n,k} = \gcv_{k,\infty}^{\lambda} - R_{k,\infty}^{\lambda}$ for $n\in\NN$ and $k\in\cK_n$.
        From \Cref{lem:gcv-inf} we have that $R_{n,k}\asto 0$ as $k,n,p\to\infty$, $p/n\to \phi\in(0,\infty)$ and $p/k\to \phi_s\in[\phi,\infty]$.
        Here we view $k$ and $p$ as $k_n$ and $p_n$ indexed by $n$.
        Then from \Cref{lem:conv_cond_expectation}~\ref{lem:conv_cond_expectation_sampling-1} the conclusion follows.
    \end{proof}

    \subsection{Proof of \Cref{cor:gcv-opt-ridge}}\label{app:subsec:gcv-opt-ridge}
    \begin{proof}[Proof of \Cref{cor:gcv-opt-ridge}]
        From \Cref{thm:uniform-consistency-k}, we have 
        $$\max_{k\in\cK_n} |\gcv_{k,\infty}^{\lambda} - R_{k,\infty}^{\lambda}| \asto 0.$$
        This implies that
        \begin{align*}
            \min_{k\in\cK_n} \gcv_{k}^{0} &= \min_{k\in\cK_n} \RlamM[0][\infty][p/n][p/k]\left( 1 + \frac{\gcv_{k}^{0} - \RlamM[0][\infty][p/n][p/k]}{\RlamM[0][\infty][p/n][p/k]} \right)\\
            & \lessgtr \min_{k\in\cK_n} \RlamM[0][\infty][p/n][p/k]\left( 1 \pm \max_{k\in\cK_n}\left|\frac{\gcv_{k}^{0} - \RlamM[0][\infty][p/n][p/k]}{\RlamM[0][\infty][p/n][p/k]} \right|\right)\\
            & \lessgtr \min_{k\in\cK_n} \RlamM[0][\infty][p/n][p/k]\left( 1 \pm \frac{1}{\sigma^2}\max_{k\in\cK_n}|\gcv_{k}^{0} - \RlamM[0][\infty][p/n][p/k]|\right)\\
            &\asto \inf_{\phi_s\geq\phi} \RlamM[0][\infty]\\
            &=\inf_{\phi_s\geq\phi,\lambda\geq 0} \RlamM[\lambda][\infty],
        \end{align*}
        where the last equality is from \Cref{thm:comparison_optimal_ridge}.
        This finishes the proof.
    \end{proof}

    \subsection{Proof of \Cref{prop:inconsistency}}\label{app:subsec:inconsistency}
    \begin{proof}[Proof of \Cref{prop:inconsistency}]
        From the proof of \Cref{lem:gcv-den}, we have
        \begin{align*}
            \frac{1}{k} \tr\left(\bM_{m}\bSigma_{m}\right) &\asto (1 -\lambda v(-\lambda;\phi_s )).
        \end{align*}
        Then, as $k,n,p\rightarrow\infty$, $p/n\rightarrow\phi$ and $p/k\rightarrow\phi_s$, we have
        \begin{align*}
            D_{k,2}^{\lambda} &= \left(1 - \frac{1}{|I_1\cup I_2|} \frac{1}{2}\sum_{m=1}^2\tr\left(\bM_{m}\bSigma_{m}\right) \right)^2\\
            &= \left(1 - \frac{k}{|I_1\cup I_2|} \frac{1}{k}\frac{1}{2}\sum_{m=1}^2\tr\left(\bM_{m}\bSigma_{m}\right) \right)^2\\
            &\asto \left(1 - \frac{\phi_s}{2\phi_s-\phi}(1-\lambda v(-\lambda;\phi_s))\right)^2 =:\sD_{2}^{\lambda}(\phi,\phi_s).
        \end{align*}
        where the convergence of $k/|I_1\cup I_2|$ is from \Cref{lem:i0_mean}.        
        It then follows that
        \begin{align*}
            \gcv_{k,2}^{\lambda} &= \frac{T_{k,2}^{\lambda}}{D_{k,2}^{\lambda}} \asto \gcvdet[\lambda][\phi][\phi_s][2]:=\frac{\sT_{2}^{\lambda}}{ \sD_{2}^{\lambda} },
        \end{align*}
        where $\sT_{2}^{\lambda}$ defined in \eqref{eq:train-err} has the following expression:
        \begin{align*}
            \RlamMtr[\lambda][2]& = \frac{1}{2}\frac{\phi_s-\phi}{2\phi_s-\phi}\RlamM[\lambda][1] + \frac{1}{2}\Ddet[][\lambda][\phi_s] \\
                \left(\frac{\phi_s}{2\phi_s-\phi}\RlamM[\lambda][1] \right.+ &\left(\frac{2(\phi_s-\phi)}{2\phi_s-\phi} \right. \frac{1}{\lambda v(-\lambda;\phi_s)}+\left.\left. \frac{\phi}{2\phi_s-\phi}\right)\RlamM[\lambda][\infty]\right) .
        \end{align*}
        On the other hand,
        we have
        \begin{align*}
            \RlamM[\lambda][2] &= \frac{1}{2}\RlamM[\lambda][1] + \frac{1}{2} \RlamM[\lambda][\infty].
        \end{align*}
        Note that $\RlamM[\lambda][1]>\RlamM[\lambda][\infty]>\sigma^2$ when $\phi<\phi_s<\infty$ because $\rho^2>0$.
        When $\lambda=0$ and $\phi_s>1$, we have $\lambda v(-\lambda;\phi_s) = 0$ and
        \begin{align*}
            \gcvdet[0][\phi][\phi_s][2] &= \ddfrac{\frac{\phi_s-\phi}{2\phi_s-\phi} \RlamM[\lambda][1] + \frac{2(\phi_s-\phi)}{2\phi_s-\phi} \RlamM[\lambda][\infty]}{2\left(1 - \frac{\phi_s}{2\phi_s-\phi}\right)^2} = \frac{1}{2}\cdot \frac{2\phi_s-\phi}{\phi_s-\phi} \left(\RlamM[\lambda][1] +  2\RlamM[\lambda][\infty]\right) .
        \end{align*}
        It follows that
        \begin{align*}
            \gcvdet[0][\phi][\phi_s][2] - \RlamM[\lambda][2] &= \frac{1}{2(\phi_s-\phi)}\left(\phi_s\RlamM[\lambda][1] + (3\phi_s-\phi) \RlamM[\lambda][\infty]\right) =: c,
        \end{align*}
        and $\gcv_{k,2}^{\lambda} - \RlamM[\lambda][2] \pto c>0$, which completes the proof.
    \end{proof}

\section{Proof of \Cref{lem:gcv-den} (convergence of the GCV denominator functional) }\label{app:sec:gcv-den}
    \begin{proof}[Proof of \Cref{lem:gcv-den}]
        By definition, the smooth matrix for $M=\infty$ is given by
        \begin{align*}
            \bS_{\lambda,\infty}(\{\cD_{I_{\ell}}\}_{\ell=1}^M) := \lim_{M\rightarrow\infty}\frac{1}{M}\sum_{\ell=1}^M\bS_{\lambda}(\cD_{I_{\ell}}). 
        \end{align*}
        For the denominator, note that for any fixed $n\in\NN$, as $M\rightarrow\infty$, $I_{1:M}\asto [n]$. Then from \Cref{lem:ridge-D} (stated and proved below),
    \begin{align*}
        \frac{1}{|I_{1:M}|}\tr(\bS_{\lambda}(\cD_{I_{\ell}})) &\asto  \frac{1}{n} \tr\left(\bX\bM_{\ell}\frac{\bX^{\top} \bL_{\ell}}{k}\right) = \frac{1}{n} \tr\left(\bM_{\ell}\bSigma_{\ell}\right) \asto \frac{\phi}{\phi_s}(1 -\lambda v(-\lambda;\phi_s )).
    \end{align*}
    By continuous mapping theorem, 
    we have
    \begin{align}
        \left(1 - \frac{1}{n}\tr(\bS_{\lambda,\infty})\right)^2 \asto \Ddet :=
        \left(\frac{\phi_s-\phi}{\phi_s}+ \frac{\phi}{\phi_s}\lambda v(-\lambda;\phi_s)\right)^2. \label{eq:Ddet}
    \end{align}
    \end{proof}

    \begin{lemma}[Deterministic approximation of the denominator functional]\label{lem:ridge-D}
        Under Assumption \ref{asm:rmt-feat}, for all $m\in[M]$ and $I_m\in \cI_k$, let $\hSigma_m=\bX^{\top}\bL_m\bX/k$, $\bL_m\in\RR^{n\times n}$ be a diagonal matrix with $(\bL_{m})_{ll}=1$ if $l\in I_m$ and 0 otherwise, and $\bM_m= (\bX^{\top}\bL_m\bX/k+\lambda\bI_p)^{-1}$.
        Then, it holds that for all $m\in[M]$ and $I_m\in\cI_k$,
        \begin{align*}
            \frac{1}{n} \tr\left(\bM_{m}\hSigma_{m}\right) \asto \frac{\phi}{\phi_s}(1 -\lambda v(-\lambda;\phi_s )),
        \end{align*}
        as $n,k,p\rightarrow\infty$, $p/n\rightarrow\phi\in(0,\infty)$, and $p/k\rightarrow\phi_s\in[\phi,\infty)$, where the nonnegative constant $\tv(\lambda  ;\phi,\phi_s)$ is as defined in \eqref{eq:tv_tc_ridge}.
    \end{lemma}
    \begin{proof}[Proof of \Cref{lem:ridge-D}]
        Note that $\bM_{m}\hSigma_{m} = \bI_p-\lambda \bM_m$. From \Cref{cor:asympequi-scaled-ridge-resolvent}, we have that $\lambda \bM_m\asympequi (v(-\lambda;\phi_s)\bSigma+\bI_p)^{-1}$. Then by \Cref{lem:calculus-detequi}~\ref{lem:calculus-detequi-item-trace}, it follows that
        \begin{align*}
            \frac{1}{n} \tr\left(\bM_{m}\bSigma_{m}\right) &\asto \phi\lim_{p\rightarrow\infty} \frac{1}{p} \tr\left(\bI_p-(v(-\lambda;\phi_s)\bSigma+\bI_p)^{-1}\right) \\
            &= \phi\lim_{p\rightarrow\infty} \left( 1 - \int \frac{1}{1+v(-\lambda;\phi_s ) r} \rd H_p(r)\right) \\
            &= \phi  \int \frac{v(-\lambda;\phi_s ) r}{1+v(-\lambda;\phi_s ) r} \rd H(r)\\
            &=\frac{\phi}{\phi_s}(1 -\lambda v(-\lambda;\phi_s ))
        \end{align*}
        where in the second last line, we used the fact that $H_p$ and $H$ have compact supports and Assumption \ref{asm:lin-mod} and the last equality is due to the definition of $v(-\lambda;\phi_s )$ in \eqref{eq:v_ridge}.
    \end{proof}

\section{Proof of \Cref{lem:gcv-num} (convergence of the GCV numerator functional) 
}\label{app:sec:gcv-num}

        \begin{proof}[Proof of \Cref{lem:gcv-num}]
             For any $m\in[M]$, let $I_m$ be a sample from $\cI_k$, and $\bL_{I_m}\in\RR^{n\times n}$ be a diagonal matrix with $(\bL_{I_m})_{ll}=1$ if $l\in I_m$ and 0 otherwise.
            The ingredient estimator takes the form:
            \begin{align*}
                \tbeta^{\lambda}_{k,M}(\{\cD_{I_\ell}\}_{\ell=1}^M) &=\frac{1}{M}\sum_{m=1}^M\betaridge(\cD_{I_m})\\
                &= \frac{1}{M}\sum_{m=1}^M (\bX^{\top}\bL_{I_m}\bX/k+\lambda\bI_p)^{-1}(\bX^{\top}\bL_{I_m}\by/k)\\
                &= \frac{1}{M}\sum_{m=1}^M \left[
                \left(\frac{\bX^{\top}\bL_{I_m}\bX}{k}+\lambda\bI_p\right)^{-1}\frac{\bX^{\top}\bL_{I_m}}{k} \bbeta_0  + \left(\frac{\bX^{\top}\bL_{I_m}\bX}{k}+\lambda\bI_p\right)^{-1}\frac{\bX^{\top}\bL_{I_m}}{k} \bepsilon\right].
            \end{align*}
            We will write $\tilde{\bbeta}_{\lambda,M} = \tbeta^{\lambda}_{k,M}$ and $\bL_m=\bL_{I_m}$ for simplicity when they are clear from the context. 
            The set operation will be propagated to such notations, e.g., $\bL_{m\cup l}=\bL_{I_m\cup I_l}$, $\bL_{m\cap l}=\bL_{I_m\cap I_l}$, $\bL_{m\setminus l}=\bL_{I_m\setminus I_l}$, etc.
            Let $\bM_m= (\bX^{\top}\bL_{I_m}\bX/k+\lambda\bI_p)^{-1}$ for $m\in[M]$, we have
            \begin{align}
                \tilde{\bbeta}_{\lambda,M}&=\frac{1}{M}\sum_{m=1}^M (\bI_p-\lambda \bM_m) \bbeta_0  + \frac{1}{M}\sum_{m=1}^M \bM_m(\bX^{\top}\bL_{m}/k)\bepsilon.\label{eq:tbeta}
            \end{align}
            
            The proof follows by combing the squared error decomposition in \Cref{lem:decomp-train-err}, with the component convergence 
            of test errors in \Cref{lem:conv-test-err} 
            and of train errors in \Cref{lem:conv-train-err}.
            To prove \Cref{lem:conv-train-err},
            we further make of the component concentration results presented in \Cref{app:sec:comp-concen,app:sec:comp-deter},
            and component deterministic approximation results presented in \Cref{app:sec:comp-deter}.
        \end{proof}

    \subsection{Decomposition of the mean squared error (\Cref{lem:decomp-train-err})}
    \label{app:sec:decomp-train-err}
    
        \begin{lemma}[Decomposition of the mean squared error for the $M$-ensemble estimator]\label{lem:decomp-train-err}
            For a dataset $\cD_n$, let $\bX\in\RR^{n\times p}$ and $\by\in\RR^n$ be the design matrix and response vector.
            Let $\bL_I\in\RR^{n\times n}$ be a diagonal matrix with $(\bL_{I})_{ll}=1$ if $l\in I$ and 0 otherwise.
            Then the mean squared error evaluated on $\cD_n$ decomposes as
            \begin{align}
                \|\by-\bX\tbeta^{\lambda}_{k}\|_2^2 &= -\EE\left[ \Errtrain(\hbeta_k^{\lambda}(\cD_{I})) + \Errtest(\hbeta_k^{\lambda}(\{\cD_{I}))\}) \mid \cD_n\right]\notag\\
                &\qquad +2 \EE\left[\Errtrain(\tbeta_{k,2}^{\lambda}(\{\cD_{I},\cD_{J}\})) + \Errtest(\tbeta_{k,2}^{\lambda}(\{\cD_{I},\cD_{J}\})) \mid \cD_n\right]
            \end{align}
            where the training and test errors are defined by
            \begin{align}
                \begin{split}
                    \Errtrain(\hbeta_k^{\lambda}(\cD_{I_{\ell}})) &=\|\bL_{I_{\ell}}(\by-\bX\hbeta_k^{\lambda}(\cD_{I_{\ell}}))\|_2^2 \\
                \Errtest(\hbeta_k^{\lambda}(\cD_{I_{\ell}}))  &= \|\bL_{I_{\ell}^c}(\by-\bX\hbeta_k^{\lambda}(\cD_{I_{\ell}}))\|_2^2\\
                \Errtrain(\tbeta_{k,2}^{\lambda}(\{\cD_{I_{m}},\cD_{I_{\ell}}\})) &= \|\bL_{I_{m}\cup \bL_{\ell}}(\by-\bX\tbeta_{k,2}^{\lambda}(\{\cD_{I_{m}},\cD_{I_{\ell}}\}))\|_2^2 \\
                \Errtest(\tbeta_{k,2}^{\lambda}(\{\cD_{I_{m}},\cD_{I_{\ell}}\})) &= \|\bL_{I_{m}^c\cap \bL_{\ell}^c}(\by-\bX\tbeta_{k,2}^{\lambda}(\{\cD_{I_{m}},\cD_{I_{\ell}}\}))\|_2^2.
                \end{split}          \label{eq:train-test-err}
            \end{align}            
        \end{lemma}
    
        From \Cref{lem:decomp-train-err}, the numerator of the GCV estimate for a $M$-ensemble estimator decomposes into a linear combination of the training and test error of all possible $1$-ensemble and $2$-ensemble estimators.
        Then the asymptotics of the numerator can be obtained, if we can show that the limits of all components exist and their linear combination remains invariable when $M$ goes off to infinity.
        \begin{proof}[Proof of \Cref{lem:decomp-train-err}]
            We first decompose the training error into the linear combination of the mean squared errors (evaluated on $\cD_n$) for 1-ensemble and 2-ensemble estimators:
            \begin{align*}
                &\frac{1}{n}\|\by-\bX\tbeta^{\lambda}_{k,M}\|_2^2 \\
                &= \frac{1}{n}\sum_{i=1}^n \left(\frac{1}{M}\sum_{\ell=1}^M(y_i-\bx_i\hbeta^{\lambda}_{k}(\cD_{I_{\ell}})\right)^2\\
                &= \frac{1}{n}\sum_{i=1}^n \frac{1}{M^2}\sum_{\ell=1}^M(y_i-\bx_i\hbeta^{\lambda}_{k}(\cD_{I_{\ell}}))^2 + 
                \frac{1}{n}\sum_{i=1}^n \frac{1}{M^2}\sum_{\substack{m,\ell\in[M]\\i\neq j}}(y_i-\bx_i\hbeta^{\lambda}_{k}(\cD_{I_{m}}))(y_i-\bx_i\hbeta^{\lambda}_{k}(\cD_{I_{\ell}})) \\
                &= \frac{1}{n}\frac{1}{M^2}\sum_{\ell=1}^M\|\by-\bX\hbeta^{\lambda}_{k}(\cD_{I_{\ell}})\|_2^2 \\
                &\qquad + \frac{1}{n}\sum_{i=1}^n \frac{1}{M^2}\sum_{\substack{m,\ell\in[M]\\i\neq j}}\frac{1}{2}
                [4(y_i-\bx_i\tbeta_{\lambda,2}(\{\cD_{I_{m}},D_{I_{\ell}}\}))^2-
                (y_i-\bx_i\hbeta^{\lambda}_{k}(\cD_{I_{m}}))^2 -
                (y_i-\bx_i\hbeta^{\lambda}_{k}(\cD_{I_{\ell}}))^2]\\
                &= -\left(\frac{1}{M}- \frac{2}{M^2}\right) \sum_{\ell=1}^M\frac{1}{n}\|\by-\bX\hbeta^{\lambda}_{k}(\cD_{I_{\ell}})\|_2^2 + \frac{2}{M^2} \sum_{\substack{m,\ell\in[M]\\i\neq j}} \frac{1}{n} \|\by-\bX\tbeta_{\lambda,2}(\{\cD_{I_{m}},\cD_{I_{\ell}}\})\|_2^2.
            \end{align*}
            Next, we further decompose the MSE into training and test errors for 1-ensemble and 2-ensemble estimators:
            \begin{align*}
                \frac{1}{n}\|\by-\bX\hbeta^{\lambda}_{k}(\cD_{I_{\ell}})\|_2^2 &= \frac{1}{n}\|\bL_{I_{\ell}}(\by-\bX\hbeta^{\lambda}_{k}(\cD_{I_{\ell}}))\|_2^2 + \frac{1}{n}\|\bL_{I_{\ell}^c}(\by-\bX\hbeta^{\lambda}_{k}(\cD_{I_{\ell}}))\|_2^2,\\
                \frac{1}{n} \|\by-\bX\tbeta_{\lambda,2}(\{\cD_{I_{m}},\cD_{I_{\ell}}\})\|_2^2 &= \frac{1}{n} \|\bL_{I_{m}\cup \bL_{\ell}}(\by-\bX\tbeta_{\lambda,2}(\{\cD_{I_{m}},\cD_{I_{\ell}}\}))\|_2^2 + \frac{1}{n} \|\bL_{I_{m}^c\cap \bL_{\ell}^c}(\by-\bX\tbeta_{\lambda,2}(\{\cD_{I_{m}},\cD_{I_{\ell}}\}))\|_2^2
            \end{align*}
            The conclusion then readily follows.
        \end{proof}

    \subsection{Convergence of test errors (\Cref{lem:conv-test-err})}
    \label{app:sec:conv-test-err}
    
     \begin{lemma}[Convergence of test errors]\label{lem:conv-test-err}
        Under Assumptions \ref{asm:rmt-feat}-\ref{asm:lin-mod}, for the test error defined in \eqref{eq:train-test-err} with $I_1,I_2\overset{\SRS}{\sim}\cI_k$, we have that as $k,n,p\rightarrow\infty$, $p/n\rightarrow \phi\in(0,\infty)$ and $p/k\rightarrow\phi_s\in[\phi,+\infty]$,
        \begin{align}
            \frac{\Errtest(\hbeta_k^{\lambda}(\cD_{I_{\ell}}))}{n-k} &\asto  \sR_{1}^{\lambda}(\phi,\phi_s)\\
            \frac{\Errtest(\tbeta_{k,2}^{\lambda}(\{\cD_{I_{m}},\cD_{I_{\ell}}\}))}{|I_m^c\cap I_{\ell}^c|} &\asto \sR_{2}^{\lambda}(\phi,\phi_s),
        \end{align}
        where the deterministic functions $\sR_M$ is defined in \Cref{thm:ver-with-replacement}.
    \end{lemma}
    
    \begin{proof}[Proof of \Cref{lem:conv-test-err}]
        From the strong law of large numbers, we have
        \begin{align*}
            \frac{1}{k}\Errtest(\hbeta^{\lambda}_{k}(\cD_{I_{\ell}})) &\asto \EE\left[
            (y_0 - \bx_0^{\top} \hbeta^{\lambda}_{k}(\cD_{I_{1}}))^2 \given \cD_n\right]\\
            \frac{1}{|I_m^c\cap I_{\ell}^c|}\Errtest(\hbeta^{\lambda}_{k}(\{\cD_{I_{m}},\cD_{I_{\ell}}\})) &\asto \EE\left[(y_0 - \bx_0^{\top} \hbeta^{\lambda}_{k}(\{\cD_{I_{m}},\cD_{I_{\ell}}\}))^2\given \cD_n\right].
        \end{align*}
        From \Cref{thm:ver-with-replacement} \citep[Theorem 4.1]{patil2022bagging}, the condition prediction risks converge in the sense that
        \begin{align*}
            \EE\left[
            (y_0 - \bx_0^{\top} \hbeta^{\lambda}_{k}(\cD_{I_{1}}))^2 \given \cD_n\right] &\asto \sR_1(\phi,\phi_s)\\
            \EE\left[(y_0 - \bx_0^{\top} \tbeta^{\lambda}_{k,2}(\{\cD_{I_{m}},\cD_{I_{\ell}}\}))^2\given \cD_n\right] &\asto \sR_2(\phi,\phi_s),
        \end{align*}
        and the conclusions follow.
    \end{proof}

    \subsection{Convergence of train errors (\Cref{lem:conv-train-err})}
    \label{app:sec:conv-train-err}

    \begin{lemma}[Convergence of train errors]\label{lem:conv-train-err}
        Under Assumptions \ref{asm:rmt-feat}-\ref{asm:lin-mod}, for the train error defined in \eqref{eq:train-test-err} with $I_1,I_2\overset{\SRS}{\sim}\cI_k$, we have that as $k,n,p\rightarrow\infty$, $p/n\rightarrow \phi\in(0,\infty)$ and $p/k\rightarrow\phi_s\in[\phi,+\infty)$,
        \begin{align}\label{eq:train-err}
            \begin{split}
                \frac{1}{k}\Errtrain(\hbeta_k^{\lambda}(\cD_{I_{\ell}})) &\asto \RlamMtr[\lambda][1]:= \Ddet[][\lambda][\phi_s] \RlamM[\lambda][1]  \\
                \frac{1}{|I_m\cup I_{\ell}|}\Errtrain(\tbeta_{k,1}^{\lambda}(\{\cD_{I_{m}},\cD_{I_{\ell}}\}))&\asto \RlamMtr[\lambda][2]:= \frac{1}{2}\frac{\phi_s-\phi}{2\phi_s-\phi}\RlamM[\lambda][1] + \frac{1}{2}\Ddet[][\lambda][\phi_s] \\
                \left(\frac{\phi_s}{2\phi_s-\phi}\RlamM[\lambda][1] \right.+ &\left(\frac{2(\phi_s-\phi)}{2\phi_s-\phi} \right. \frac{1}{\lambda v(-\lambda;\phi_s)}+\left.\left. \frac{\phi}{2\phi_s-\phi}\right)(2\RlamM[\lambda][2]-\RlamM[\lambda][1]\right),
                \end{split}
        \end{align}
        where the deterministic function $\sR_M$ is defined in \Cref{thm:ver-with-replacement}.
    \end{lemma}
    
        \begin{proof}[Proof of \Cref{lem:conv-train-err}]
        From \eqref{eq:tbeta}, we have
        \begin{align}
            \bbeta_0-\tilde{\bbeta}_{\lambda,M}&=\frac{1}{M}\sum_{m=1}^M \lambda \bM_m \bbeta_0  - \frac{1}{M}\sum_{m=1}^M \bM_m(\bX^{\top}\bL_m/k)\bepsilon. \label{eq:diff-beta-hbeta}
        \end{align}

        \paragraph{Part (1)}
        \underline{Case of $M=1$}:
        
        From \eqref{eq:diff-beta-hbeta}, the training error can be decomposed as follows:
        \begin{align*}
            \frac{1}{k}\Errtrain(\hbeta^{\lambda}_{k}(\cD_{I_{\ell}})) &= \|\bL_{\ell}\bepsilon + \bL_{\ell}\bX (\bbeta_0-\hbeta^{\lambda}_{k}(\cD_{I_{\ell}}))\|_2^2/k\\
            &= \|(\bL_{\ell} - \bL_{\ell}\bX\bM_{\ell}\bX^{\top}\bL_{\ell}/k)\bepsilon + \lambda\bL_{\ell}\bX \bM_{\ell}\bbeta_0\|_2^2/k\\
            &= T_C + T_B + T_V, 
        \end{align*}
        where the constant term $T_C$, bias term $T_B$, and the variance term $T_V$ are given by
        \begin{align}
            T_C&= \frac{2\lambda}{k} \bepsilon^{\top}\bL_{\ell}\left(\bI_n-\bX\bM_{\ell}\frac{\bX^{\top}\bL_{\ell}}{k}\right)^{\top} \bL_{\ell}\bX \bM_{\ell} \bbeta_0, \label{eq:ridge-C0}\\
            T_B &= \lambda^2\bbeta_0^{\top}\bM_{\ell}\hSigma_{\ell}\bM_{\ell}\bbeta_0, \label{eq:ridge-B0}\\
            T_V&= \frac{1}{k} \bepsilon^{\top}\bL_{\ell}\left(\bI_n-\bX\bM_{\ell}\frac{\bX^{\top}\bL_{\ell}}{k}\right)^{\top} \bL_{\ell}\left(\bI_n-\bX\bM_{\ell}\frac{\bX^{\top}\bL_{\ell}}{k}\right) \bL_{\ell}\bepsilon. \label{eq:ridge-V0}
        \end{align}
        Next, we analyze the three terms separately.
        From Lemmas \ref{lem:ridge-conv-C0} and \ref{lem:ridge-conv-V0} with $n=k$, we have that $T_C\asto0$, and
        \begin{align*}
            T_V 
            &\asto \sigma^2\left(1 - \frac{2}{k}\tr(\bM_{m}\hSigma_m) + \frac{1}{k}\tr(\bM_m\hSigma_m\bM_m\hSigma_m)\right) := T_{VT}.
        \end{align*}
        Thus, it remains to obtain the asymptotic equivalent for the bias term $T_B$ and the trace term $T_{VT}$.

        From Lemma \ref{lem:ridge-B0} and Lemma \ref{lem:ridge-V0}, we have that for all $I_1\in\cI_k$,
        \begin{align*}
            T_B&\asto \rho^2\Ddet[][\lambda][\phi_s](1+\tv(-\lambda;\phi_s,\phi_s)) \tc(-\lambda;\phi_s) \\
            T_{VT}&\asto \sigma^2\Ddet[][\lambda][\phi_s](1+\tv(-\lambda;\phi_s,\phi_s)).
        \end{align*}
        Then, we have 
        \begin{align*}
            \frac{1}{k}\Errtrain(\hbeta^{\lambda}_{k}(\cD_{I_{\ell}})) \asto \sR_{\lambda,1}(\phi,\phi_s) \Ddet[][\lambda][\phi_s],
        \end{align*}
        where $\sR_{\lambda,1}$ and $\sD_{\lambda}$ are defined in \eqref{eq:risk-det-with-replacement} and \eqref{eq:Ddet}, respectively.

        \paragraph{Part (2)}
        \underline{Case of $M=2$}:
        
        From \eqref{eq:diff-beta-hbeta}, the training error can be analogously decomposed as follows:
        \begin{align*}
            \frac{1}{|I_{m}\cup I_{l}|}\Errtrain(\tbeta^{\lambda}_{k,2}(\{\cD_{I_{1}},\cD_{I_{2}})\}) = \frac{1}{|I_{m}\cup I_{l}|}\|\bL_{m\cup l}(\bepsilon + \bX (\bbeta_0-\hbeta_{\lambda,2}))\|_2^2= T_C' + T_B' + T_V', 
        \end{align*}
        where the constant term $T_C$, bias term $T_B$, and the variance term $T_V$ are given by
        \begin{align}
            T_C'&= \frac{\lambda}{2|I_{m}\cup I_{l}|}\sum_{i,j\in\{m,l\}}\bepsilon^{\top}\left(\bI_n-\bX\bM_{i}\frac{\bX^{\top}\bL_{i}}{k}\right)^{\top} \bL_{m\cup l}\bX \bM_{j} \bbeta_0, \label{eq:ridge-C0-M2}\\
            T_B' &= \frac{\lambda^2}{4|I_{m}\cup I_l|}\sum_{i,j\in\{m,l\}} \bbeta_0^{\top}\bM_i\hSigma_{m\cup l}\bM_j\bbeta_0 \notag\\
            &=\frac{\lambda^2k}{4|I_{m}\cup I_l|}\sum_{i\in\{m,l\}} \bbeta_0^{\top}\bM_i\hSigma_i\bM_i\bbeta_0 + \frac{\lambda^2}{4|I_{m}\cup I_l|}\sum_{i\in\{m,l\}}|I_{m+l-i}\setminus I_i| \bbeta_0^{\top}\bM_i\hSigma_{(m+l-i)\setminus i}\bM_i\bbeta_0 \notag\\
            &\qquad + \frac{\lambda^2}{4} \sum_{i\in\{m,l\}} \bbeta_0^{\top}\bM_i\hSigma_{m\cup l}\bM_{m+l-i}\bbeta_0 , \label{eq:ridge-B0-M2}\\
            T_V'&= \frac{1}{4|I_{m}\cup I_l|}\sum_{i,j\in\{m,l\}}\bepsilon^{\top}\left(\bI_n-\bX\bM_i\frac{\bX^{\top}\bL_i}{k}\right)^{\top} \bL_{m\cup l} \left(\bI_n-\bX\bM_j\frac{\bX^{\top}\bL_j}{k}\right) \bepsilon. \label{eq:ridge-V0-M2}
        \end{align}

        Next, we analyze the three terms separately.
        From Lemmas \ref{lem:ridge-conv-C0} and \ref{lem:ridge-conv-V0}, we have that $T_C\asto0$, and
        \begin{align*}
            T_V' 
            &\asto   \frac{\sigma^2}{4}\sum_{i\in\{m ,l\}}\left(1 - \frac{2}{|I_m\cup I_l|}\tr(\bM_{i}\hSigma_{i}) + \frac{1}{k}\tr(\bM_i\hSigma_i\bM_i\hSigma_{m\cup l})\right)\\
            &\qquad + \frac{\sigma^2}{2}\left(1 - \frac{1}{|I_m\cup I_l|} \sum_{j\in\{m,l\}}\tr(\bM_{j}\hSigma_{j}) + \frac{1}{n}\tr(\bM_l\hSigma_{m\cap l}\bM_l\hSigma_{m\cup l})\right)\\
            &\quad = \frac{\phi_s}{2\phi_s-\phi}\frac{T_{VT}}{2} +  \frac{\sigma^2}{4}\frac{\phi_s-\phi}{2\phi_s-\phi}\left(2 + \frac{1}{k} \sum_{i\in\{m,l\}}\tr(\bM_i\hSigma_{i}\bM_i\hSigma_{(m+l-i) \setminus i}) \right)\\
            &\qquad + \frac{\sigma^2}{2}\left(1 - \frac{1}{|I_m\cup I_l|} \sum_{j\in\{m,l\}}\tr(\bM_{j}\hSigma_{j}) + \frac{1}{n}\tr(\bM_l\hSigma_{m\cap l}\bM_m\hSigma_{m\cup l})\right) := T_{VT}'.
        \end{align*}
        Thus, it remains to obtain the asymptotic equivalent for the bias term $T_B'$ and the trace term $T_{VT}'$.
                
        From Lemma \ref{lem:ridge-B0} and Lemma \ref{lem:ridge-V0}, and the convergence of the cardinality from \Cref{lem:i0_mean}, we have that for all $I_m,I_l\overset{\SRS}{\sim}\cI_k$,
        \begin{align*}
            T_B'&\asto \frac{\rho^2}{2}\tilde{t}(\phi,\phi_s)\tc(-\lambda;\phi_s),
            \quad 
            \text{and}
            \quad
            T_V'\asto \frac{\sigma^2}{2}\tilde{t}(\phi,\phi_s),
        \end{align*}
        where
        \begin{align*}
             \tilde{t}(\phi,\phi_s) &= \frac{\phi_s}{2\phi_s-\phi}\Ddet[][\lambda][\phi_s](1+\tv(-\lambda;\phi_s,\phi_s))  +
             \frac{\phi_s-\phi}{2\phi_s-\phi} (1+\tv(-\lambda;\phi_s,\phi_s)) \\
             &\qquad + 
             \Ddet[][\lambda][\phi_s]\left(\frac{2(\phi_s-\phi)}{2\phi_s-\phi} \frac{1}{\lambda v(-\lambda;\phi_s)} + \frac{\phi}{2\phi_s-\phi}\right)(1+\tv(-\lambda;\phi,\phi_s)) .
        \end{align*}
        Then, we have 
        \begin{align*}
            & \frac{1}{|I_{m}\cup I_{l}|}\Errtrain(\tbeta_{k,2}^{\lambda}(\{\cD_{I_{1}},\cD_{I_{2}})\})\\ &\asto \frac{\rho^2}{2}\tilde{t}(\phi,\phi_s)\tc(-\lambda;\phi_s) + \frac{\sigma^2}{2}\tilde{t}(\phi,\phi_s)\\
            &= \frac{1}{2}\Ddet[][\lambda][\phi_s]\left(\frac{\phi_s}{2\phi_s-\phi}\sR_{\lambda,1}(\phi,\phi_s) + \left(\frac{2(\phi_s-\phi)}{2\phi_s-\phi} \frac{1}{\lambda v(-\lambda;\phi_s)} + \frac{\phi}{2\phi_s-\phi}\right) (2\sR_{\lambda,2}(\phi,\phi_s) - \sR_{\lambda,2}(\phi,\phi_s))\right) \\
            &\qquad + \frac{1}{2}\frac{\phi_s-\phi}{2\phi_s-\phi}\sR_1(\phi,\phi_s),
        \end{align*}   
        which finishes the proof.
    \end{proof}

    \subsection{Component concentrations}
    \label{app:sec:comp-concen}
    
        In this subsection, we will show that the cross-term $T_C$ converges to zero and the variance term $T_V$ converges to its corresponding trace expectation.
    
        \subsubsection{Convergence of the cross term}
        \begin{lemma}[Convergence of the cross term]\label{lem:ridge-conv-C0}
            Under Assumptions \ref{asm:rmt-feat}-\ref{asm:lin-mod},
            for $T_C$ and $T_C'$ as defined in \eqref{eq:ridge-C0} and \eqref{eq:ridge-C0-M2}),
            we have $T_C\asto 0$ and $T_C'\asto 0$ as $k,p\rightarrow\infty$ and $p/k\rightarrow\phi_s$.
        \end{lemma}
        \begin{proof}[Proof of Lemma \ref{lem:ridge-conv-C0}]
            We first prove the result for $T_C'$.
            Note that
            \begin{align*}
                T_C' &= - \frac{\lambda}{M^2} \cdot \frac{1}{|I_1\cup I_2|}\left\langle \sum_{m=1}^2\left(\bI_n-\bX\frac{\bM_m\bX^{\top}\bL_m}{k}\right)^{\top}\bL_{m\cup l} \bX \sum_{m=1}^2\bM_m\bbeta_0
                , \bepsilon\right\rangle.
            \end{align*}
            We next bound the squared norm:
            \begin{align*}
                &\frac{1}{|I_m\cup I_l|}\norm{\frac{1}{2}\sum_{m=1}^2\left(\bI_n-\frac{\bX\bM_m\bX^{\top}\bL_m}{k}\right)^{\top} \bL_{m\cup l}\bX \sum_{m=1}^2\bM_m\bbeta_0}_2^2\\
                & \leq  \sum_{j=1}^2\sum\limits_{l=1}^2\left[\frac{|I_m\cup I_l|}{4k^2}\norm{( \bM_j\bX^{\top}\bL_j)^{\top} \hSigma_{m\cup l} \bM_l \bbeta_0}_2^2 + \frac{1}{4|I_1\cup I_2|}\norm{ \bL_j \bX \bM_l \bbeta_0}_2^2\right]\\
                & \leq 
                \frac{\norm{\bbeta_0}_2^2}{4}\cdot \sum_{j=1}^2\sum\limits_{l=1}^2\left[\frac{|I_m\cup I_l|}{k^2}
                \norm{\bM_l \hSigma_{m\cup l}  \bM_j  \bX^{\top} \bL_j\bX\bM_j \hSigma_{m\cup l} \bM_l}_{\oper} + \frac{k}{|I_1\cup I_2|}\norm{\hSigma_j}_{\oper}\norm{  \bM_l }_{\oper} \right]\\
                & \leq \frac{\norm{\bbeta_0}_2^2}{4} \cdot \sum_{j=1}^2\sum\limits_{l=1}^2\left[\frac{|I_m\cup I_l|}{k}
                \norm{\bM_l}_{\oper}^2 \norm{\hSigma_l}_{\oper}^2\norm{ \bM_j (\bX^{\top} \bL_j\bX/k)\bM_j }_{\oper} + \frac{k}{|I_1\cup I_2|} \norm{\bM_l}_{\oper} \norm{\hSigma_l}_{\oper}\right]\\
                & = \frac{\norm{\bbeta_0}_2^2}{4}\cdot \sum_{j=1}^2\sum\limits_{l=1}^2\left[
                \frac{|I_m\cup I_l|}{k}
                \norm{\bM_l}_{\oper}^2 \norm{\hSigma_l}_{\oper}^2\norm{\bM_j}_{\oper}\norm{ \bI_p - \lambda\bM_j }_{\oper}+ \frac{k}{|I_1\cup I_2|} \norm{\bM_l}_{\oper} \norm{\hSigma_l}_{\oper}\right]\\
                & \leq \frac{\norm{\bbeta_0}_2^2}{\lambda }\left(\frac{|I_m\cup I_l|}{k\lambda^2}
                 + \frac{k}{|I_1\cup I_2|} \right) \norm{\hSigma_l}_{\oper}^2, 
            \end{align*}
            where the last inequality is due to the fact that $\|\bM_j\|_{\oper}\leq1/\lambda$ and $\norm{ \bI_p - \lambda\bM_j }_{\oper}\leq 1$.
            By Assumption \ref{asm:lin-mod}, $\norm{\bbeta_0}_2^2$ is uniformly bounded in $p$.
            From \citet{bai2010spectral}, we have $\limsup\norm{\hSigma}_{\oper}\leq \limsup\max_{1\leq i\leq p} s_i^2\leq r_{\max}(1+\sqrt{\phi_s})^2$ almost surely as $k,p\rightarrow\infty$ and $p/k\rightarrow\phi_s\in(0,\infty)$.
            From \Cref{lem:i0_mean}, we have that $ k/|I_1\cup I_2| \asto k/(2k-k^2/n)$, which is $\phi_s/(2\phi_s-\phi)$ almost surely.
            Then we have that the square norm is almost surely upper bounded by some constant.
            Applying 
            \Cref{lem:concen-linform},
            we thus have that $T_C'\asto 0$.

            Note that when $I_1=I_2$, $T_C'$ reduces to $T_C$; thus, the conclusion for $T_C$ also holds.
        \end{proof}

        \subsubsection{Convergence of the variance term}
        \begin{lemma}[Convergence of the variance term]\label{lem:ridge-conv-V0}
            Under Assumptions \ref{asm:rmt-feat}-\ref{asm:lin-mod}, let $M\in\NN$ and $\hSigma=\bX^{\top}\bX/n$. For all $m\in[M]$ and $I_m\in \cI_k$, let $\hSigma_m=\bX^{\top}\bL_m\bX/k$, $\bL_m\in\RR^{n\times n}$ be a diagonal matrix with $(\bL_{m})_{ll}=1$ if $l\in I_m$ and 0 otherwise, and $\bM_m= (\bX^{\top}\bL_m\bX/k+\lambda\bI_p)^{-1}$.
            Then, for all
            $m,l\in[M]$ and $m\neq l$, it holds that:

            \begin{align}
                &\frac{1}{k} \bepsilon^{\top}\left(\bI_n-\bX\bM_m\frac{\bX^{\top}\bL_m}{k}\right)^{\top} \bL_{m}\left(\bI_n-\bX\bM_m\frac{\bX^{\top}\bL_m}{k}\right) \bepsilon\notag\\
                &\qquad -
                \sigma^2\left(1 - \frac{2}{k}\tr(\bM_{m}\hSigma_m) + \frac{1}{k}\tr(\bM_m\hSigma_m\bM_m\hSigma_m)\right)
                \asto 0, \label{eq:TV_tr_1}\\                
                &\frac{1}{|I_m\cup I_l|} \bepsilon^{\top}\left(\bI_n-\bX\bM_i\frac{\bX^{\top}\bL_i}{k}\right)^{\top} \bL_{m\cup l}\left(\bI_n-\bX\bM_i\frac{\bX^{\top}\bL_i}{k}\right) \bepsilon\notag\\
                &\qquad -\sigma^2\left(1- \frac{2}{|I_m\cup I_l|} \tr(\bM_{i}\hSigma_{i}) + \frac{1}{k}\tr(\bM_i\hSigma_{i}\bM_i\hSigma_{m\cup l})\right)\asto 0,\label{eq:TV_tr_2}\\
                &\frac{1}{|I_m\cup I_l|} \bepsilon^{\top}\left(\bI_n-\bX\bM_m\frac{\bX^{\top}\bL_m}{k}\right)^{\top} \bL_{m\cup l}\left(\bI_n-\bX\bM_l\frac{\bX^{\top}\bL_l}{k}\right) \bepsilon\notag\\
                &\qquad -\sigma^2\left(1- \frac{1}{|I_m\cup I_l|} \sum_{\ell\in\{i,j\}}\tr(\bM_{\ell}\hSigma_{\ell}) + \frac{1}{n}\tr(\bM_i\hSigma_{i\cap j}\bM_j\hSigma_{m\cup l})\right)\asto 0,\label{eq:TV_tr_3}
            \end{align}
            where $i,j\in\{m,l\}$, $i\neq j$, and  $\hSigma_{m\cup l} = \bX^{\top}\bL_{m\cup l}\bX/|I_m\cup I_l|$, as $n,k,p\rightarrow\infty$, $p/n\rightarrow\phi\in(0,\infty)$, and $p/k\rightarrow\phi_s\in[\phi,\infty)$.
        \end{lemma}
        \begin{proof}[Proof of Lemma \ref{lem:ridge-conv-V0}]
            We first prove the last convergence result.
            Note that
            \begin{align*}
                & \norm{\left(\bI_n-\bX\bM_m\frac{\bX^{\top}\bL_m}{k}\right)^{\top} \bL_{m\cup l} \left(\bI_n-\bX\bM_l\frac{\bX^{\top}\bL_l}{k}\right)}_{\oper} \\
                &= \norm{\bL_{m\cup l} - \frac{1}{k}\bL_m\bX\bM_m\bX^{\top}\bL_{m\cup l} - \frac{1}{k}\bL_{m\cup l}\bX\bM_l\bX^{\top}\bL_{l} +\frac{1}{k^2} \bL_m\bX\bM_m\bX^{\top}\bL_{m\cup l}\bX\bM_{l}\bX^{\top}\bL_l}_{\oper}\\
                &\leq 1 + \sqrt{\frac{|I_m\cup I_l|}{k}}\sum_{j\in\{m,l\}}\norm{\hSigma_j}_{\oper}^{\frac{1}{2}}  \norm{\bM_j}_{\oper} \norm{\hSigma_{m\cup l}}_{\oper}^{\frac{1}{2}} + \frac{|I_m\cup I_l|}{k}\norm{\hSigma_m}_{\oper}^{\frac{1}{2}}\norm{\bM_m}_{\oper}\norm{\hSigma_{m\cup l} }_{\oper}\norm{ \bM_l}_{\oper}\norm{\hSigma_l}_{\oper}^{\frac{1}{2}}\\
                &\leq 1 + \frac{2}{\lambda}\sqrt{\frac{|I_m\cup I_l|}{k}}\norm{\hSigma_m}_{\oper}^{\frac{1}{2}} \norm{\hSigma_l}_{\oper}^{\frac{1}{2}} +  \frac{1}{\lambda^2}\frac{|I_m\cup I_l|}{k}\norm{\hSigma_m}_{\oper}^{\frac{1}{2}}\norm{\hSigma_{m\cup l} }_{\oper}\norm{\hSigma_l}_{\oper}^{\frac{1}{2}}.
            \end{align*}
            Now, we have $\limsup\|\hSigma\|_{\oper}\leq \limsup\max_{1\leq i\leq p} s_i^2\leq r_{\max}(1+\sqrt{\phi})^2$ almost surely as $n,p\rightarrow\infty$ and $p/n\rightarrow\phi\in(0,\infty)$ from \citet{bai2010spectral}.
            Similarly, $\limsup\|\hSigma_m\|_{\oper}\leq r_{\max}(1+\sqrt{\phi_s})^2$ almost surely.
            From \Cref{lem:i0_mean}, $|I_m\cup I_l|/k\asto (2\phi_s-\phi)/\phi_s$.
            Then the above quantity is asymptotically upper bounded by some constant as $n,k,p\rightarrow\infty$, $p/n\rightarrow\phi\in(0,\infty)$ and $p/k\rightarrow\phi_s\in[\phi,\infty)$.
            From 
            \Cref{lem:concen-quadform},
            it follows that 
            \begin{align*}
                &\frac{1}{|I_m\cup I_l|} \bepsilon^{\top}\left(\bI_n-\bX\bM_i\frac{\bX^{\top}\bL_i}{k}\right)^{\top} \bL_{m\cup l} \left(\bI_n-\bX\bM_j\frac{\bX^{\top}\bL_j}{k}\right) \bepsilon\\
                &\quad- \frac{\sigma^2}{|I_m\cup I_l|}\tr\left(\left(\bI_n-\bX\bM_i\frac{\bX^{\top}\bL_i}{k}\right)^{\top} \bL_{m\cup l}\left(\bI_n-\bX\bM_j\frac{\bX^{\top}\bL_j}{k}\right)\right) \asto 0.
            \end{align*}
            Expanding the trace term above, we have
            \begin{align}
                &\frac{\sigma^2}{|I_m\cup I_l|}\tr\left(\left(\bI_n-\bX\bM_i\frac{\bX^{\top}\bL_i}{k}\right)^{\top} \bL_{m\cup l}\left(\bI_n-\bX\bM_j\frac{\bX^{\top}\bL_j}{k}\right)\right)\notag\\
                &= \sigma^2\left(1 - \frac{1}{|I_m\cup I_l|}\sum_{\ell\in\{i,j\}}\tr(\bM_{\ell}\hSigma_{\ell}) + \frac{|I_i\cap I_j|}{k^2}\tr(\bM_i\hSigma_{i\cup j}\bM_j\hSigma_{m\cap l})\right).\label{eq:eq:TV_tr_2_0}
            \end{align}            
            Since $I_m,I_l\overset{\SRS}{\sim}\cI_k$, from \Cref{lem:i0_mean} we have that $|I_m\cap I_l|/k \asto k/n$.
            Then,
            we have
            \begin{align*}
                &\frac{1}{|I_m\cup I_l|} \bepsilon^{\top}\left(\bI_n-\bX\bM_i\frac{\bX^{\top}\bL_i}{k}\right)^{\top} \bL_{m\cup l} \left(\bI_n-\bX\bM_j\frac{\bX^{\top}\bL_j}{k}\right) \bepsilon\\
                &\quad - \sigma^2\left(1 - \frac{1}{|I_m\cup I_l|} \sum_{\ell\in\{i,j\}}\tr(\bM_{\ell}\hSigma_{\ell}) + \frac{1}{n}\tr(\bM_i\hSigma_{i\cap j}\bM_j\hSigma_{m\cup l})\right)\asto 0,
            \end{align*}
            and thus \eqref{eq:TV_tr_3} follows.
            
            Setting $i=j$ in \eqref{eq:eq:TV_tr_2_0} yields \eqref{eq:TV_tr_2}.
            
            Finally, setting $i=j=l=m$ in \eqref{eq:eq:TV_tr_2_0} finishes the proof for \eqref{eq:TV_tr_1}.
        \end{proof}

    \subsection{Component deterministic approximations}
    \label{app:sec:comp-deter}
    
        \subsubsection{Deterministic approximation of the bias functional}
        \begin{lemma}[Deterministic approximation of the bias functional]\label{lem:ridge-B0}
            Under Assumptions \ref{asm:rmt-feat}-\ref{asm:lin-mod}, for all $m\in[M]$ and $I_m\in \cI_k$, let $\hSigma_m=\bX^{\top}\bL_m\bX/k$, $\bL_m\in\RR^{n\times n}$ be a diagonal matrix with $(\bL_{m})_{ll}=1$ if $l\in I_m$ and 0 otherwise, and $\bM_m= (\bX^{\top}\bL_m\bX/k+\lambda\bI_p)^{-1}$.
            Then, it holds that:
            \begin{enumerate}
                \item For all $m\in[M]$,
                $$\lambda^2 \bbeta_0^{\top}\bM_m\hSigma_m\bM_m\bbeta_0 \asto \rho^2\lambda^2 v(-\lambda;\phi_s)^2(1+\tv(-\lambda;\phi_s,\phi_s)) \tc(-\lambda;\phi_s).$$

                \item For all $m,l\in[M]$, $m\neq l$ and $I_m,I_l\overset{\texttt{\textup{SRSWR}}}{\sim}\cI_k$,
                \begin{align}
                    \lambda^2 \bbeta_0^{\top}\bM_l\hSigma_{m\setminus l}\bM_l\bbeta_0 \asto \rho^2 (1+\tv(-\lambda;\phi_s,\phi_s))\tc(-\lambda;\phi_s).
                \end{align}

                \item For all $m,l\in[M]$, $m\neq l$ and $I_m,I_l\overset{\texttt{\textup{SRSWR}}}{\sim}\cI_k$,
                \begin{align}
                    \lambda^2 \bbeta_0^{\top}\bM_m\hSigma_{m\cup l}\bM_l\bbeta_0 \asto \rho^2\lambda^2\left(\frac{2(\phi_s-\phi)}{2\phi_s-\phi} \frac{1}{\lambda v(-\lambda;\phi_s)} + \frac{\phi}{2\phi_s-\phi}\right)v(-\lambda;\phi_s)^2(1+\tv(-\lambda;\phi,\phi_s))\tc(-\lambda;\phi_s),
                \end{align}

            \end{enumerate}
            where $\hSigma_{m\cup l} = |I_m\cup I_l|^{-1}\bX^{\top}\bL_{m\cup l}\bX$, as $n,k,p\rightarrow\infty$, $p/n\rightarrow\phi\in(0,\infty)$, and $p/k\rightarrow\phi_s\in[\phi,\infty)$, where $\phi_0=\phi_s^2/\phi$, 
            $T_B$ is as defined in \eqref{eq:ridge-B0}, and the nonnegative constants $\tv(-\lambda  ;\phi,\phi_s)$ and $\tc(-\lambda;\phi_s)$ are as defined in \eqref{eq:tv_tc_ridge}.
        \end{lemma}
        \begin{proof}[Proof of Lemma \ref{lem:ridge-B0}]
            We split the proof into different parts.
            
            \paragraph{Part (1)}
            From Lemma \ref{lem:deter-approx-generalized-ridge}~\ref{eq:detequi-ridge-genvar} (with $\bA=\bI_p$), we have that %
            \begin{align}
                \lambda^2
                \bM_m \hSigma_m \bM_m
                \asympequi 
                v(-\lambda;\phi_s)^2(1+\tv(-\lambda;\phi_s,\phi_s)) \cdot (v(-\lambda; \phi_s) \bSigma + \bI_p)^{-1}
                 \bSigma 
                (v(-\lambda; \phi_s) \bSigma + \bI_p)^{-1}.\label{eq:lem-det-approx-B0-0}
            \end{align}
            By the definition of asymptotic equivalent, we have
            \begin{align}
                \lambda^2\bbeta_0^{\top} \bM_m \hSigma_m \bM_m\bbeta_0 &\asto \lim\limits_{p\rightarrow\infty}v(-\lambda;\phi_s)^2(1+\tv(-\lambda;\phi_s,\phi_s))\sum\limits_{i=1}^p \frac{r_i}{(1 + r_i v(-\lambda; \phi_s))^2} (\bbeta_0^{\top}w_i)^2 \notag\\
                &= \lim\limits_{p\rightarrow\infty}\|\bbeta_0\|_2^2v(-\lambda;\phi_s)^2(1+\tv(-\lambda;\phi_s,\phi_s))\int \frac{r}{(1 + v(-\lambda; \phi_s)r)^2} \rd G_{p}(r)\notag\\
                &= \rho^2v(-\lambda;\phi_s)^2(1+\tv(-\lambda;\phi_s,\phi_s)) \int \frac{r}{(1 + v(-\lambda; \phi_s)r)^2} \rd G(r),\label{eq:lem-det-approx-B0-1}
            \end{align}
            where the last equality holds since $G_p$ and $G$ have compact supports and invoking Assumption \ref{asm:lin-mod}.

        \paragraph{Part (2)}
        From \Cref{lem:resolv-insample}~\ref{item:lem-ridge-ind}, we have
        \begin{align*}
            \bM_l\hSigma_{m\setminus l}\bM_l \asympequi \bM_l\bSigma\bM_l .
        \end{align*}
        Then, from \citet[Lemma S.2.4]{patil2022bagging}, the conclusion follows.
        
        \paragraph{Part (3)}            
        For the cross term, it suffices to derive the asymptotic equivalent of $\bbeta_0^{\top}\bM_1\hSigma_{1\cup 2}\bM_2\bbeta_0$. 
        We begin with analyzing the asymptotic equivalent of $\bM_1\hSigma_{1\cup 2}\bM_2$.
        Let $i_0=\tr(\bL_1\bL_2)$ be the number of shared samples between $\cD_{I_1}$ and $\cD_{I_2}$, we use the decomposition
        \begin{align*}
            \bM_j^{-1}&=\frac{i_0}{k}(\hat{\bSigma}_0+\lambda\bI_p) +  \frac{k-i_0}{k}(\hat{\bSigma}_j^{\text{ind}} + \lambda\bI_p),\qquad j=1,2,
        \end{align*} where $\hat{\bSigma}_0=\bX^{\top}\bL_1\bL_2\bX/i_0$ and $\hat{\bSigma}_{j}^{\text{ind}}=\bX^{\top}(\bL_j-\bL_1\bL_2)\bX/(k-i_0)$ are the common and individual covariance estimators of the two datasets.
        Let $\bN_0=(\hat{\bSigma}_0+\lambda\bI_p)^{-1}$ and $\bN_j=(\hat{\bSigma}_j^{\text{ind}}+\lambda\bI_p)^{-1}$ for $j=1,2$. Then
        \begin{align}
            \bM_j&= \left(\frac{i_0}{k}\bN_0^{-1}+\frac{k-i_0}{k}\bN_j^{-1}\right)^{-1}, \qquad j=1,2, \label{eq:ridge-bias-term1}
        \end{align}
        where the equalities hold because $\bN_0$ is invertible when $\lambda>0$.
        Note that 
        \begin{align*}
            \hSigma_{1\cup 2} &= \frac{k}{2k-i_0}\hSigma_{1} + \frac{k}{2k-i_0}\hSigma_{2} - \frac{i_0}{2k-i_0}\hSigma_{0},\\
            &= \frac{k}{2k-i_0}\sum_{j=1}^2(\bM_{j}^{-1}-\lambda\bI_p)  - \frac{i_0}{2k-i_0}\hSigma_{0}.
        \end{align*}
        We have that
        \begin{align}
            \lambda^2\bM_1\hSigma_{1\cup 2}\bM_2 &=  \frac{k}{2k-i_0}\lambda^2\sum_{j=1}^2 \bM_j - \frac{2k}{2k-i_0}\lambda^3 \bM_1\bM_2 - \frac{i_0}{2k-i_0}\lambda^2\bM_1\hSigma_{0}\bM_2 .\label{eq:bias-M1ShM2}
        \end{align}
        Next, we derive the asymptotic equivalents for the three terms in \eqref{eq:bias-M1ShM2}.
        From \Cref{cor:asympequi-scaled-ridge-resolvent}, the first term admits 
        \begin{align}
            \lambda\bM_j\asympequi (v(-\lambda;\phi_s)\bSigma+\bI_p)^{-1}.\label{eq:bias-M1ShM2-term-1}
        \end{align}

        Note that
        \begin{align}
            \lambda^2\bM_1\bM_2 &\asympequi \left(v(-\lambda; \phi_s) \bSigma + \bI_p\right)^{-1} (\tv(-\lambda;\phi,\phi_s,\bI_p)\bSigma+\bI_p)\left(v(-\lambda; \phi_s) \bSigma + \bI_p\right)^{-1}\\
            &\asympequi \left(v(-\lambda; \phi_s) \bSigma + \bI_p\right)^{-2} (\tv(-\lambda;\phi,\phi_s,\bI_p)\bSigma+\bI_p),\label{eq:bias-M1ShM2-term-2}
        \end{align}
        where
        \begin{align*}
            \tv(-\lambda;\phi,\phi_s,\bI_p) &=\ddfrac{\phi \int\frac{r}{(1+v(-\lambda; \phi_s)r)^2}\rd H(r) }{v(-\lambda; \phi_s)^{-2}-\phi \int\frac{r^2}{(1+v(-\lambda; \phi_s)r)^2}\rd H(r)} .
        \end{align*}

        For the third term, 
         \begin{align}
            \bM_1 \hSigma_0 \bM_2 &\asympequi \tv_v(-\lambda;\phi,\phi_s)( v(-\lambda; \phi_s) \bSigma + \bI_p)^{-2}\bSigma,\label{eq:bias-M1ShM2-term-3}
        \end{align}
        where 
        \begin{align*}
            \tv_v(-\lambda;\phi,\phi_s) &:= \ddfrac{1}{v(-\lambda; \phi_s)^{-2}-\phi \int\frac{r^2}{(1+ v(-\lambda; \phi_s)r)^2}\,\rd H(r)}. 
        \end{align*}
        Combining \eqref{eq:bias-M1ShM2-term-1}-\eqref{eq:bias-M1ShM2-term-3}, we get 
        \begin{align}
            \lambda^2\bM_1\hSigma_{1\cup 2}\bM_2 &\asympequi  \left(\frac{2\phi_s}{2\phi_s-\phi}  (v(-\lambda; \phi_s)-\tv(-\lambda;\phi,\phi_s,\bI_p))  - \frac{\phi}{2\phi_s-\phi} \lambda\tv_v(-\lambda;\phi,\phi_s) \right)
            \lambda( v(-\lambda; \phi_s) \bSigma + \bI_p)^{-2}\bSigma\notag\\
            &= \lambda^2v(-\lambda; \phi_s)^2(1+\tv(-\lambda;\phi,\phi_s))\left(\frac{2(\phi_s-\phi)}{2\phi_s-\phi} \frac{1}{\lambda v(-\lambda;\phi_s)} + \frac{\phi}{2\phi_s-\phi}\right)( v(-\lambda; \phi_s) \bSigma + \bI_p)^{-2}\bSigma.
        \end{align}
        The last conclusion follows analogously as in \eqref{eq:lem-det-approx-B0-1}.
        \end{proof}

        \subsubsection{Deterministic approximation of the variance functional}
        \begin{lemma}[Deterministic approximation of the variance functional]\label{lem:ridge-V0}
            Under Assumptions \ref{asm:rmt-feat}-\ref{asm:lin-mod}, for all $m\in[M]$ and $I_m\in \cI_k$, let $\hSigma_m=\bX^{\top}\bL_m\bX/k$, $\bL_m\in\RR^{n\times n}$ be a diagonal matrix with $(\bL_{m})_{ll}=1$ if $l\in I_m$ and 0 otherwise,  and $\bM_m= (\bX^{\top}\bL_m\bX/k+\lambda\bI_p)^{-1}$.
            Then, it holds that:            
            \begin{enumerate}
                \item For all $m\in[M]$ and $I_m\in\cI_k$,
                \begin{align}
                    1 - \frac{2}{k}\tr(\bM_{m}\hSigma_m) + \frac{1}{k}\tr(\bM_m\hSigma_m\bM_m\hSigma_m) \asto \lambda^2v(-\lambda; \phi_s)^2(1+\tv(-\lambda;\phi_s,\phi_s)). \label{eq:TV-1}
                \end{align}

                \item For all $m,l\in[M]$, $m\neq l$ and $I_m,I_l\overset{\texttt{\textup{SRSWR}}}{\sim}\cI_k$,
                \begin{align}
                    \frac{1}{k}\tr(\bM_m\hSigma_{m}\bM_m\hSigma_{l\setminus m}) \asto \tv(-\lambda;\phi_s,\phi_s). \label{eq:TV-2}
                \end{align}

                \item For all $m,l\in[M]$, $m\neq l$ and $I_m,I_l\overset{\texttt{\textup{SRSWR}}}{\sim}\cI_k$,
                \begin{align}
                    1 - \frac{1}{|I_m\cup I_l|}&\sum_{j\in\{m,l\}}\tr(\bM_{j}\hSigma_{j}) + \frac{1}{n}\tr(\bM_l\hSigma_{m\cap l}\bM_m\hSigma_{m\cup l}) \asto\notag\\ 
                    & \Ddet[][\lambda][\phi_s]\left(\frac{2(\phi_s-\phi)}{2\phi_s-\phi} \frac{1}{\lambda v(-\lambda;\phi_s)} + \frac{\phi}{2\phi_s-\phi}\right)(1+\tv(-\lambda;\phi,\phi_s)),\label{eq:TV-3}
                \end{align}

            \end{enumerate}
            where $\hSigma_{l\setminus m} = |I_l\setminus I_m|^{-1}\bX^{\top}\bL_{l\setminus m}\bX$ and $\hSigma_{m\cup l} = |I_m\cup I_l|^{-1}\bX^{\top}\bL_{m\cup l}\bX$, as $n,k,p\rightarrow\infty$, $p/n\rightarrow\phi\in(0,\infty)$, and $p/k\rightarrow\phi_s\in[\phi,\infty)$, where the nonnegative constant $\tv(\lambda  ;\phi,\phi_s)$ is as defined in \eqref{eq:tv_tc_ridge}.
        \end{lemma}
        \begin{proof}[Proof of Lemma \ref{lem:ridge-V0}]
            We split the proof into different parts.
            
            \paragraph{Part (1)} Note that
            \begin{align*}
                \tr(\bM_m\hSigma_m\bM_m\hSigma_m) &= \tr(\bM_m\hSigma_m) - \lambda \tr(\bM_n^2\hSigma_m).
            \end{align*}
            We now have
            \begin{align*}
                1 - \frac{2}{k}\tr(\bM_{m}\hSigma_m) + \frac{1}{k}\tr(\bM_m\hSigma_m\bM_m\hSigma_m) &= 1 - \frac{1}{k}\tr(\bM_{m}\hSigma_m) - \frac{\lambda}{k}\tr(\bM_{m}^2\hSigma_m)\\
                &= 1 - \frac{p}{k} + \frac{\lambda}{k} \tr(\bM_m) - \frac{\lambda}{k} \tr(\bM_{m}^2\hSigma_m).
            \end{align*}
            From \Cref{cor:asympequi-scaled-ridge-resolvent} we have that $\lambda \bM_m\asympequi (v(-\lambda;\phi_s)\bSigma+\bI_p)^{-1}$.
            From \Cref{lem:deter-approx-generalized-ridge}~\ref{eq:detequi-ridge-genvar} (with $\bA=\bI$), we have that for $j\in[M]$,
            \begin{align}
                \bM_{m}\hSigma_m\bM_{m} \asympequi \tv_v(-\lambda; \phi_s) (v(-\lambda; \phi_s) \bSigma + \bI_p)^{-2} \bSigma. \label{eq:lem-V0-term-1}
            \end{align}
            By the trace rule \Cref{lem:calculus-detequi}~\ref{lem:calculus-detequi-item-trace} , we have
            \begin{align}
                \frac{\lambda}{k} \tr(\bM_m) - \frac{\lambda}{k} \tr(\bM_{m}^2\hSigma_m) &\asto \lim\limits_{p\rightarrow\infty}\frac{p}{k}\cdot \frac{1}{p}\left(\tr((v(-\lambda;\phi_s)\bSigma+\bI_p)^{-1}) -
                \tr(\lambda\tv_v(-\lambda; \phi_s) (v(-\lambda; \phi_s) \bSigma + \bI_p)^{-2} \bSigma) \right)\notag\\
                &= \phi_s
                 \lim\limits_{p\rightarrow\infty}\frac{1}{p}\sum\limits_{i=1}^p\frac{1 + v(-\lambda; \phi_s)r_i-\lambda\tv_v(-\lambda; \phi_s) r_i}{(v(-\lambda; \phi_s) r_i + 1)^2}\notag\\
                &=\phi_s \lim\limits_{p\rightarrow\infty}\int \frac{1 + v(-\lambda; \phi_s)r - \lambda\tv_v(-\lambda; \phi_s)r}{(1 + v(-\lambda; \phi_s) r)^2}\rd H_p(r)\notag\\
                &=\phi_s\int \frac{1 + v(-\lambda; \phi_s)r - \lambda\tv_v(-\lambda; \phi_s)r}{(1 + v(-\lambda; \phi_s) r)^2}\rd H(r),\qquad j=1,2, \label{eq:lem:ridge-V0-1}
            \end{align}
            where in the last line we used the fact that $H_p$ and $H$ have compact supports and Assumption \ref{asm:lin-mod}.
            Then, we have
            \begin{align*}
                &1 - \frac{2}{k}\tr(\bM_{m}\hSigma_m) + \frac{1}{k}\tr(\bM_m\hSigma_m\bM_m\hSigma_m)\\ &\asto 1 - \phi_s + \phi_s\int \frac{1 + v(-\lambda; \phi_s)r - \lambda\tv_v(-\lambda; \phi_s)r}{(1 + v(-\lambda; \phi_s) r)^2}\rd H(r) \\
                &= 1 - \phi_s\int \frac{v(-\lambda;\phi_s)r}{1 + v(-\lambda; \phi_s) r}\rd H(r) - \phi_s\int \frac{\lambda\tv_v(-\lambda; \phi_s)r}{(1 + v(-\lambda; \phi_s) r)^2}\rd H(r) \\
                &= \lambda v(-\lambda;\phi_s) - \phi_s\int \frac{\lambda\tv_v(-\lambda; \phi_s)r}{(1 + v(-\lambda; \phi_s) r)^2}\rd H(r) \\
                &= \lambda \tv_v(-\lambda; \phi_s) \left(v(-\lambda; \phi_s)^{-1} - \phi_s \int \frac{ v(-\lambda; \phi_s)r^2}{(1 + v(-\lambda; \phi_s) r)^2}\rd H(r) - \phi_s\int \frac{r}{(1 + v(-\lambda; \phi_s) r)^2}\rd H(r)\right)\\
                &= \lambda \tv_v(-\lambda; \phi_s) \left(v(-\lambda; \phi_s)^{-1} - \phi_s \int \frac{r}{1 + v(-\lambda; \phi_s) r}\rd H(r) \right)\\
                &= \lambda^2 \tv_v(-\lambda; \phi_s)\\
                &= \lambda^2v(-\lambda; \phi_s)^2(1+\tv(-\lambda;\phi_s,\phi_s)) ,
            \end{align*}
            and thus, \eqref{eq:TV-1} follows.

            \paragraph{Part (2)} Since $\bM_m\hSigma_{m}\bM_m$ and $\hSigma_{l\setminus m}$ are independent, from \Cref{lem:resolv-insample}  \ref{item:lem-ridge-ind}, we have
            \begin{align*}
                \bM_m\hSigma_{m}\bM_m\hSigma_{l\setminus m} \asympequi \bM_m\hSigma_{m}\bM_m\bSigma.
            \end{align*}
            Then, by the definition of asymptotic equivalents, it follows that
            \begin{align*}
                \frac{1}{k}\tr(\bM_m\hSigma_{m}\bM_m\hSigma_{l\setminus m}) &= \frac{1}{k}\tr( \bM_m\hSigma_{m}\bM_m\bSigma) \asto \tv(-\lambda;\phi_s,\phi_s),
            \end{align*}
            where the convergence is due to \citet[Lemma S.2.5]{patil2022bagging}.
            
            \paragraph{Part (3)} 
            Let $i_0=|I_m\cap I_l|$. The first two terms in \eqref{eq:lem:ridge-V0-1} satisfy that
            \begin{align*}
                1 - \frac{1}{|I_m\cup I_l|} \sum_{j\in\{m,l\}}\tr(\bM_j\hSigma_j) &= 1 - \frac{p}{2k-i_0} \sum_{j\in\{m,l\}}\frac{1}{p}\tr(\bM_j\hSigma_j)
                \asto 1 - \frac{2\phi_s^2}{2\phi_s-\phi} \int \frac{v(-\lambda;\phi_s)r}{1 + v(-\lambda;\phi_s)r}\rd H(r),
            \end{align*}
            where the convergence is from Part 1.
            The last term can be further decomposed because
            \begin{align*}
                \tr(\bM_l\hSigma_{m\cap l}\bM_m\hSigma_{m\cup l}) &= \frac{k}{2k-i_0}\sum_{j\in\{m,l\}}\tr(\bM_l\hSigma_{m\cap l}\bM_m\hSigma_{j}) - \frac{i_0}{2k-i_0} \tr(\bM_l\hSigma_{m\cap l}\bM_m\hSigma_{m\cap l})\\
                &= \frac{k}{2k-i_0}\sum_{j\in\{m,l\}}[\tr(\bM_j\hSigma_{m\cap l}) - \lambda \tr(\bM_m\hSigma_{m\cap l}\bM_l)] - \frac{i_0}{2k-i_0} \tr(\bM_l\hSigma_{m\cap l}\bM_m\hSigma_{m\cap l}).
            \end{align*}
            From \Cref{lem:resolv-insample}  \ref{item:lem-ridge-V0-term-det-1}, \ref{item:lem-ridge-B0-term-det-2}, and \ref{item:lem-ridge-V0-term-det-2}, we have
            \begin{align*}
                \bM_j\hSigma_{m\cap l} &\asympequi \bI_p - (v(-\lambda;\phi_s) \bSigma+\bI_p)^{-1} \\
                \bM_m\hSigma_{m\cap l}\bM_l&\asympequi \tv_v(-\lambda;\phi,\phi_s)( v(-\lambda; \phi_s) \bSigma + \bI_p)^{-2}\bSigma\\
                \bM_l\hSigma_{m\cap l}\bM_m\hSigma_{m\cap l} &\asympequi \left(\frac{\phi_s}{\phi} v(-\lambda;\phi_s)-\frac{\phi_s-\phi}{\phi}\lambda \tv_v(-\lambda;\phi,\phi_s) \right)(v(-\lambda;\phi_s)\bSigma+\bI_p)^{-1}\bSigma \\
                &\qquad - \lambda \tv_v(-\lambda;\phi,\phi_s)(v(-\lambda;\phi_s)\bSigma+\bI_p)^{-2}\bSigma.
            \end{align*}
            Combining the above terms and by Assumption \ref{asm:lin-mod}, we have
            \begin{align*}
                &1- \frac{1}{|I_m\cup I_l|}\sum_{j\in\{m,l\}}\tr(\bM_{j}\hSigma_{j}) + \frac{1}{n}\tr(\bM_l\hSigma_{m\cap l}\bM_m\hSigma_{m\cup l})\\ 
                &\quad \asto 1 - \frac{2\phi_s^2}{2\phi_s-\phi} \int \frac{v(-\lambda;\phi_s)r}{1+v(-\lambda;\phi_s)r}\rd H(r) \\
                &\quad \qquad  + \phi \frac{2\phi_s}{2\phi_s-\phi}\left(\int \frac{v(-\lambda;\phi_s)r}{1+v(-\lambda;\phi_s)r}\rd H(r)  - \lambda\tv_v(-\lambda;\phi,\phi_s) \int \frac{r}{(1+v(-\lambda;\phi_s)r)^2}\rd H(r)  \right)\\
                &\qquad \qquad - \phi \frac{\phi}{2\phi_s-\phi}\left(\frac{\phi_s}{\phi}\int \frac{r}{1+v(-\lambda;\phi_s)r}\rd H(r)  - \frac{\phi_s-\phi}{\phi} \lambda\tv_v(-\lambda;\phi,\phi_s) \int \frac{r}{1+v(-\lambda;\phi_s)r}\rd H(r) \right.\\
                &\qquad \qquad \qquad  \left. -\lambda\tv_v(-\lambda;\phi,\phi_s)\int \frac{r}{(1+v(-\lambda;\phi_s)r)^2}\rd H(r) \right)\\
                &\quad = 1 - \phi_s \int \frac{v(-\lambda;\phi_s)r}{1+v(-\lambda;\phi_s)r}\rd H(r) + \frac{\phi(\phi_s-\phi)}{2\phi_s-\phi}\lambda\tv_v(-\lambda;\phi,\phi_s) \int \frac{r}{1+v(-\lambda;\phi_s)r}\rd H(r) \\
                &\qquad - \phi \lambda\tv_v(-\lambda;\phi,\phi_s)\int \frac{r}{(1+v(-\lambda;\phi_s)r)^2}\rd H(r)\\
                &\quad = \lambda v(-\lambda;\phi_s) + \frac{\phi(\phi_s-\phi)}{2\phi_s-\phi} \int \frac{\lambda\tv_v(-\lambda;\phi,\phi_s)r}{1+v(-\lambda;\phi_s)r}\rd H(r)  - \phi \int \frac{\lambda\tv_v(-\lambda;\phi,\phi_s)r}{(1+v(-\lambda;\phi_s)r)^2}\rd H(r)\\
                &\quad = \frac{\lambda^2 \tv_v(-\lambda;\phi,\phi_s)}{2\phi_s-\phi}\left(\frac{2\phi_s-\phi}{\lambda v(-\lambda;\phi_s)} - \frac{\phi (2\phi_s-\phi)}{\lambda}\int \frac{v(-\lambda;\phi_s)r^2 }{(1+v(-\lambda;\phi_s) r)^2}\rd H(r) \right.\\
                &\qquad\left. + \frac{\phi(\phi_s-\phi)}{\lambda} \int \frac{r}{1+v(-\lambda;\phi_s)r}\rd H(r)  - \frac{\phi}{\lambda} \int \frac{r}{(1+v(-\lambda;\phi_s)r)^2}\rd H(r) \right)\\
                &\quad=\frac{\lambda^2 \tv_v(-\lambda;\phi,\phi_s)}{2\phi_s-\phi}\left(\frac{2\phi_s-\phi}{\lambda v(-\lambda;\phi_s)} - \frac{\phi (2\phi_s-\phi)}{\lambda}\int \frac{r }{1+v(-\lambda;\phi_s) r}\rd H(r) \right.\\
                &\qquad\left. + \frac{\phi(\phi_s-\phi)}{\lambda} \int \frac{r}{1+v(-\lambda;\phi_s)r}\rd H(r)   \right)\\
                &\quad = \lambda^2\tv_v(-\lambda;\phi, \phi_s)^2\left(\frac{2(\phi_s-\phi)}{2\phi_s-\phi} \frac{1}{\lambda v(-\lambda;\phi_s)}+\frac{\phi}{2\phi_s-\phi}\right)\\
                &\quad = \Ddet[][\lambda][\phi_s]\left(\frac{2(\phi_s-\phi)}{2\phi_s-\phi} \frac{1}{\lambda v(-\lambda;\phi_s)} + \frac{\phi}{2\phi_s-\phi}\right)(1+\tv(-\lambda;\phi,\phi_s)).
            \end{align*}
        \end{proof}

\section{Proof of \Cref{prop:gcv}}\label{app:proof-ridge}

        \begin{proof}[Proof of \Cref{prop:gcv}]
            From \Cref{lem:decomp-train-err}, we have
            \begin{align*}
                \frac{1}{n}\|\by-\bX\tbeta^{\lambda}_{k}\|_2^2 &= - \frac{1}{n} \EE_{I\overset{\SRS}{\sim}\cI_k}\left[ \Errtrain(\hbeta_k^{\lambda}(\cD_{I})) + \Errtest(\hbeta_k^{\lambda}(\{\cD_{I}\})) \right]\notag\\
                &\qquad + \frac{2}{n}\EE_{(I_m,I_{\ell})\overset{\SRS}{\sim}\cI_k}  \left[\Errtrain(\tbeta_k^{\lambda}(\{\cD_{I_{m}},\cD_{I_{\ell}}\})) + \Errtest(\tbeta_k^{\lambda}(\{\cD_{I_{m}},\cD_{I_{\ell}}\})) \right].
            \end{align*}
            Since by \Cref{lem:conv-test-err}, \Cref{lem:conv-train-err} and \Cref{lem:conv_cond_expectation_sampling}, each expectation converges, we have that
            \begin{align*}
                \frac{1}{n}\|\by-\bX\tbeta^{\lambda}_{k}\|_2^2 &\asto 
                \frac{2\phi(2\phi_s-\phi)}{\phi_s^2}\RlamMtr[\lambda][2] + \frac{2(\phi_s-\phi)^2}{\phi_s^2}\RlamM[\lambda][2]  - \frac{\phi}{\phi_s}\RlamMtr[\lambda][1] - \frac{\phi_s-\phi}{\phi_s}\RlamM[\lambda][1] ,
            \end{align*}
            where the convergence of the averages is from \Cref{lem:conv_cond_expectation} and the convergence of coefficients is from \Cref{lem:i0_mean}.
            Since the denominator converges from \Cref{lem:gcv-den}, we further have
            \begin{align*}
                \gcv_{k,\infty}^{\lambda} \asto \gcvdet = \ddfrac{\frac{2\phi(2\phi_s-\phi)}{\phi_s^2}\RlamMtr[\lambda][2] + \frac{2(\phi_s-\phi)^2}{\phi_s^2}\RlamM[\lambda][2] - \frac{\phi}{\phi_s}\RlamMtr[\lambda][1] - \frac{\phi_s-\phi}{\phi_s}\RlamM[\lambda][1]}{\Ddet},
            \end{align*}
            for $\lambda>0$ and $\phi_s\in[\phi,+\infty)$.

            For the boundary case when $\lambda>0$ but $\phi_s=+\infty$, we require \Cref{prop:Rdet-ridge-infinity}; for the boundary case when $\lambda=0$, we require \Cref{prop:Rdet-lam-0}.
            Applying \Cref{prop:Rdet-ridge-infinity} and \Cref{prop:Rdet-lam-0} finishes the proof.
        \end{proof}  
    
    \subsection{Boundary case: diverging subsample aspect ratio for the ridge predictor}
    \label{sec:proof-ridge}
    
    \begin{proposition}[Risk approximation when $\phi_s\rightarrow+\infty$]\label{prop:Rdet-ridge-infinity}
    Under Assumptions \ref{asm:rmt-feat}-\ref{asm:lin-mod}, for all $\lambda>0$, we have $$\gcv_k^{\lambda}\asto \gcvdet[\lambda][\phi][\infty],$$
    as $k,n,p\rightarrow\infty$, $p/n\rightarrow\phi\in(0,\infty)$ and $p/k\rightarrow\infty$, where $\gcvdet[\lambda][\cdot][\cdot]$ is defined in \Cref{prop:gcv}.
    \end{proposition}
    \begin{proof}[Proof of \Cref{prop:Rdet-ridge-infinity}]
        Recall that
        \begin{align*}
            \frac{1}{n}\|\by-\bX\tbeta_k^{\lambda}\|_2^2 &= \lim_{M\rightarrow\infty}\frac{1}{n}\|\by-\bX\tbeta^{\lambda}_{k,M}\|_2^2 = (\bbeta_0 -\tbeta^{\lambda}_{k})^{\top} \hSigma (\bbeta_0 -\tbeta^{\lambda}_{k}) + \frac{1}{n}\bepsilon^{\top}\bepsilon + \frac{2}{n}(\bbeta_0 -\tbeta^{\lambda}_{k})^{\top} \hSigma \bepsilon
        \end{align*}
        From \Cref{lem:concen-linform} and \Cref{lem:concen-quadform}, we have that $\bbeta_0^{\top} \bX^{\top} \bepsilon/n\asto 0$ and $\bepsilon^{\top}\bepsilon/n\asto \sigma^2$ as $n\rightarrow\infty$.
        For the other term, note that for any $(I_1,\ldots,I_M)\overset{\SRS}{\sim}\cI_k$,
        \begin{align*}
            \|\tbeta^{\lambda}_{k}\|_2 &\leq \lim_{M\rightarrow\infty}\EE_{(I_1,\ldots,I_M)\overset{\SRS}{\sim}\cI_k}\left[ \frac{1}{M}\sum_{m=1}^M\|(\bX^{\top}\bL_m\bX/k+\lambda\bI_p)^{-1}(\bX^{\top}\bL_m\by/k)\|_2\right] \\
            &\leq \lim_{M\rightarrow\infty}\EE_{(I_1,\ldots,I_M)\overset{\SRS}{\sim}\cI_k}\left[ \frac{1}{M}\sum_{m=1}^M\|(\bX^{\top}\bL_m\bX/k+\lambda\bI_p)^{-1}\bX^{\top}\bL_m/\sqrt{k}\|\cdot\|\bL_m\by/\sqrt{k}\|_2\right]\\
            &\leq C\sqrt{\rho^2+\sigma^2}\cdot \max_{I_m\in\cI_k} \|(\bX^{\top}\bL_m\bX/k+\lambda\bI_p)^{-1}\bX^{\top}\bL_m/\sqrt{k}\|_{\oper},
        \end{align*}
        where the last inequality holds eventually almost surely since Assumptions \ref{asm:rmt-feat}-\ref{asm:lin-mod} imply that the entries of $\by$ have bounded $4$-th moment, and thus from the strong law of large numbers, $\| \bL_m \by / \sqrt{k} \|_2$ is eventually almost surely
        bounded above by $C\sqrt{\EE[y_1^2]} = C\sqrt{\rho^2 + \sigma^2}$ for some constant $C$.
        Observe that operator norm of the matrix $(\bX^\top \bL_m\bX / k+ \lambda\bI_p)^{-1} \bX\bL_m / \sqrt{k}$ is upper bounded $\max_i s_i/(s_i^2+\lambda)\leq 1/s_{\min}$ where $s_i$'s are the singular values of $\bX$ and $s_{\min}$ is the smallest nonzero singular value.
        As $k, p \to \infty$ such that $p / k \to \infty$, $s_{\min} \to \infty$ almost surely (e.g., from results of \citet{bloemendal2016principal}) and therefore, $\| \tbeta^{\lambda}_{k} \|_2 \to 0$ almost surely.
        Because $\|\hSigma\|_{\oper}$ is upper bounded almost surely, we further have $\tbeta_{k}^{\top} \bX^{\top} \bepsilon/n\asto 0$.
        Consequently we have $(\bbeta_0-\tbeta_{k})^{\top} \bX^{\top} \bepsilon/n\asto 0$ and
        \begin{align*}
            \frac{1}{n}\|\by-\bX\tbeta^{\lambda}_{k}\|_2^2 & \asto \bbeta_0^{\top}\hSigma \bbeta_0 + \sigma^2 .
        \end{align*}
        Finally, from \Cref{lem:resolv-insample}~\ref{item:lem-ridge-ind} $\bbeta_0^{\top}\hSigma \bbeta_0\asto \bbeta_0^{\top}\bSigma \bbeta_0$ and from \Cref{asm:lin-mod}, we have
        \begin{align*}
            \frac{1}{n}\|\by-\bX\tbeta^{\lambda}_{k}\|_2^2 & \asto  \sigma^2 + \rho^2 \int r \rd G(r).
        \end{align*}
        Since $\bS_{k}^{\lambda} = \bX \tbeta^{\lambda}_{k}$, we have that $\tr(\bS_{k}^{\lambda} )/n\asto 0$.
        So the denominator converges to 1, almost surely.
        
        From \Cref{lem:ridge-fixed-point-v-properties}, we have $\gcvdet[\lambda][\phi][\infty]:=\lim_{\phi_s\rightarrow+\infty}\gcvdet= \sigma^2 + \rho^2\int r\rd G(r)$, which is also the limit of the GCV estimate.
        Thus, $\gcvdet[\lambda][\phi][\infty]$ is well defined and $\gcvdet$ is right continuous at $\phi_s=+\infty$.
    \end{proof}

    \subsection{Boundary case: the ridgeless predictor}
    \label{sec:proof-ridgeless}
        \begin{proposition}[Risk approximation when $\lambda=0$]\label{prop:Rdet-lam-0}
            Under Assumptions \ref{asm:rmt-feat}-\ref{asm:lin-mod}, suppose that the conclusion of \Cref{prop:gcv} holds for $\lambda>0$. Then, we have 
            $$\gcv_k^0 \asto \gcvdet[0]:= \lim_{\lambda\rightarrow0^+} \gcvdet,$$
            as $k,n,p\rightarrow\infty$, $p/n\rightarrow\phi\in(0,\infty) $ and $p/k\rightarrow[\phi,+\infty]$, where $\gcvdet[\lambda][\cdot][\cdot]$ is defined in \Cref{prop:gcv}.
        \end{proposition}
        \begin{proof}[Proof of \Cref{prop:Rdet-lam-0}]
            We analyze the numerator and the denominator separately.

            \paragraph{Part (1)} For the denominator, note that
            \begin{align*}
                P_{n,\lambda}:=(1 - \tr(\bS^{\lambda}_{k})/n )^2 &= \lim_{M\rightarrow\infty}(1 - \tr(\bS^{\lambda}_{k,M})/n )^2,
            \end{align*}
            where $\bS^{\lambda}_{k} = \lim_{M\rightarrow\infty}\bS^{\lambda}_{k,M}$ is the smoothing matrix.
            Since $\bS^{\lambda}_{k}\succeq \zero_{n\times n}$ and
            \begin{align}
                \|\bS^{\lambda}_{k}\|_{\oper} &\leq \max_{I_m\in\cI_k} \|\bX(\bX^{\top}\bL_m\bX/k+\lambda\bI_p)^{-1}\bX^{\top}\bL_m/\sqrt{k}\|_{\oper},\label{eq:sec:proof-ridgeless-1}
            \end{align}
            which is also upper bounded almost surely from the proof in \Cref{prop:Rdet-ridge-infinity} (when $\lambda=0$, the inverse in the above display is replaced by pseudo-inverse). Thus, we have $P_{n,\lambda}$ is almost surely upper bounded $\lambda\in\Lambda:=[0,\lambda_{\max}]$ for any $\lambda_{\max}\in(0,\infty)$ fixed.

            Next we inspect the boundedness of the derivative of $P_{n,\lambda}$:
            \begin{align*}
                \frac{\partial}{\partial \lambda}P_{n,\lambda} &= \frac{\partial}{\partial \lambda} \lim_{M\rightarrow\infty}(1 - \tr(\bS^{\lambda}_{k,M})/n )^2 =: \frac{\partial}{\partial \lambda} \lim_{M\rightarrow\infty} Q_{M,\lambda}.
            \end{align*}
            We claim that $$\frac{\partial}{\partial \lambda} \lim_{M\rightarrow\infty} Q_{M,\lambda} =  \lim_{M\rightarrow\infty} \frac{\partial}{\partial \lambda}Q_{M,\lambda}.$$
            To see this, we need to show that $Q_{M,\lambda}$ is equicontinuous in $\lambda$ over $\Lambda$.
            First we know that $Q_{M,\lambda}$ is differentiable in $\lambda$.
            From \eqref{eq:sec:proof-ridgeless-1}, we have that $Q_{M,\lambda}$ is uniformly upper bounded over $\lambda\in\Lambda$ almost surely.
            Note that
            \begin{align*}
                \frac{\partial}{\partial \lambda}Q_{M,\lambda} &= 
                2(1-\tr(\bS^{\lambda}_{k,M})) \tr\left(\frac{\partial}{\partial \lambda}\bS^{\lambda}_{k,M}\right) ,
            \end{align*}
            where $$\frac{\partial}{\partial \lambda}\bS^{\lambda}_{k,M}  =\frac{1}{M}\sum_{m=1}^M \bX \left(\frac{\bX^{\top}\bL_m\bX}{k}+\lambda\bI\right)^{-2} \frac{\bX^{\top}\bL_m}{k} .$$
            By the similar arguments as in \Cref{prop:Rdet-ridge-infinity}, we have that $\norm{\partial\bS^{\lambda}_{k,M}/\partial\lambda}_{\oper}$, and $\|\bS^{\lambda}_{k,M}\|_2^2$ are uniformly upper bounded almost surely over $\Lambda$, the equicontinuity conclusion follows.
            Then by Moore-Osgood theorem, we have 
            \begin{align*}
                \frac{\partial}{\partial \lambda}P_{n,\lambda} 
                &= \lim_{M\rightarrow\infty} 2(1-\tr(\bS^{\lambda}_{k,M})) \tr\left(\frac{\partial}{\partial \lambda}\bS^{\lambda}_{k,M}\right) 
            \end{align*}
            is uniformly upper bounded almost surely over $[0,+\infty]$ independent of $\lambda$ and $M$.
            Therefore, we conclude that $|\partial P_{n,\lambda}/\partial\lambda|$ is upper bounded almost surely over $\lambda\in\Lambda$.

            On the other hand, we know that $P_{n,\lambda}\asto \Ddet$ for $\lambda>0$.
            Define $\Ddet[][0]:=\lim_{\lambda\rightarrow0^+}\Ddet$.
            When $\lambda=0$ and $\phi_s>1$, we know that $\Ddet[][0]$ is well-defined because $v(-\lambda; \phi_s)$ is finite and continuous from \Cref{lem:fixed-point-v-lambda-properties}.
            When $\lambda=0$ and $\phi_s\in(0,1]$, from the definition of fixed-point solution \eqref{eq:v_ridge}, we have
            \begin{align*}
                1 &= v(-\lambda;\phi_s)\lambda + \phi_s \int \frac{v(-\lambda;\phi_s)r}{1+v(-\lambda;\phi_s)r}\rd H(r).
            \end{align*}
            In this case, $v(0;\phi_s)=+\infty$ from \Cref{lem:fixed-point-v-lambda-properties}.
            Let $\lambda\to 0^+$, we have 
            \begin{align*}
                1 &= \lim_{\lambda\to 0^+}v(-\lambda;\phi_s)\lambda + \phi_s \lim_{\lambda\to 0^+}\int \frac{v(-\lambda;\phi_s)r}{1+v(-\lambda;\phi_s)r}\rd H(r) = \lim_{\lambda\to 0^+}v(-\lambda;\phi_s)\lambda + \phi_s.
            \end{align*}
            Then we have $\lim_{\lambda\to 0^+}v(-\lambda;\phi_s)\lambda = 1 - \phi_s$ and $\Ddet = (1 - \phi_s)^2$.
            Thus, $\Ddet[][0]$ is always well-defined.

            From \Cref{lem:fixed-point-v-lambda-properties}, there exists $M'>0$ such that the magnitudes of $v(-\lambda; \phi_s)$ and its derivative with respect to $\lambda$ are continuous and bounded by $M'$ for all $\lambda\in[0,+\infty]$.
            It follows that $|\Ddet|$ and $|\partial \Ddet/ \partial \lambda|$ are uniformly upper bounded almost surely.
            From Moore-Osgood theorem and the continuity property from \Cref{lem:fixed-point-v-lambda-properties}, we have
            \begin{align*}
                \lim_{n\rightarrow \infty}\lim_{\lambda\rightarrow 0^+} P_{n,\lambda} &= \lim_{\lambda\rightarrow 0^+}\lim_{n\rightarrow \infty} P_{n,\lambda}
                =\lim_{\lambda\rightarrow 0^+}\Ddet
                =\Ddet[][0].
            \end{align*}

            \paragraph{Part (2)} For the numerator, note that
            \begin{align*}
                P_{n,\lambda}':=\frac{1}{n}\|\by-\bX\tbeta^{\lambda}_{k}\|_2^2 &= \frac{1}{n}\|(\bI_n-\bS^{\lambda}_{k})\by\|_2^2 .
            \end{align*}
            Assumptions \ref{asm:rmt-feat}-\ref{asm:lin-mod} imply that the entries of $\by$ have bounded $4$-th moment, and thus from the strong law of large numbers, $\| \by / \sqrt{n} \|_2$ is eventually almost surely bounded above by $C\sqrt{\EE[y_1^2]} = C\sqrt{\rho^2 + \sigma^2}$ for some constant $C$.
            On the other hand, $\bS^{\lambda}_{k}\succeq \zero_{n\times n}$ and $\|\bS^{\lambda}_{k}\|_{\oper}$ is also upper bounded almost surely from Part (1). Thus, we have $P_{n,\lambda}'$ is almost surely upper bounded $\lambda\in\Lambda$.

            Next we inspect the boundedness of the derivative of $P_{n,\lambda}'$:
            \begin{align*}
                \frac{\partial}{\partial \lambda}P_{n,\lambda}' &= \frac{2}{n}(\by-\bX\tbeta^{\lambda}_{k})^{\top} \frac{\partial}{\partial \lambda}\bS^{\lambda}_{k}\by\\
                &= \frac{\partial}{\partial \lambda} \lim_{M\rightarrow\infty}\frac{2}{n}(\by-\bX\tbeta^{\lambda}_{k,M})^{\top} \bS^{\lambda}_{k,M}\by\\
                &= \frac{\partial}{\partial \lambda} \lim_{M\rightarrow\infty}\frac{2}{n}\by^{\top}(\bI_n-\bS^{\lambda}_{k,M}) \bS^{\lambda}_{k,M}\by =: \frac{\partial}{\partial \lambda} \lim_{M\rightarrow\infty} Q_{M,\lambda}'.
            \end{align*}
            We claim that $$\frac{\partial}{\partial \lambda} \lim_{M\rightarrow\infty} Q_{M,\lambda}' =  \lim_{M\rightarrow\infty} \frac{\partial}{\partial \lambda}Q_{M,\lambda}'.$$
            To see this, we need to show that $Q_{M,\lambda}'$ is equicontinuous in $\lambda$ over $\Lambda$.
            First we know that $Q_{M,\lambda}'$ is differentiable in $\lambda$.
            From \eqref{eq:sec:proof-ridgeless-1}, we have that $Q_{M,\lambda}'$ is uniformly upper bounded over $\lambda\in\Lambda$ almost surely.
            Similarly, we have
            \begin{align*}
                \frac{\partial}{\partial \lambda}Q_{M,\lambda}' &= 
                \frac{2}{n} \by^{\top} (\bI_n- 2\bS^{\lambda}_{k,M}) \frac{\partial}{\partial \lambda} \bS^{\lambda}_{k,M} \by,
            \end{align*}
            and
            \begin{align*}
                \left|\frac{\partial}{\partial \lambda}Q_{M,\lambda}' \right| &\leq \norm{\bI_n- 2\bS^{\lambda}_{k,M}}_{\oper} \norm{\frac{\partial}{\partial \lambda}\bS^{\lambda}_{k,M}}_{\oper} \frac{1}{n}\|\by\|_2^2.
            \end{align*}
            and the equicontinuity conclusion follows analogously as in Part (1).
            Therefore, we conclude that $|\partial P_{n,\lambda}'/\partial\lambda|$ is upper bounded almost surely over $\lambda\in[0,+\infty]$.

            On the other hand, we know that $P_{n,\lambda}'\asto \gcvdet$ for $\lambda>0$.
            Define $\Ddet[][0]\gcvdet[0]:=\lim_{\lambda\rightarrow0^+}(\Ddet\gcvdet) = \Ddet[][0] \RlamM[0]$, which is well defined from Part (1) and \Cref{thm:ver-with-replacement}.
            From \Cref{lem:fixed-point-v-lambda-properties}, there exists $M'>0$ such that the magnitudes of $v(-\lambda; \phi_s)$, $\tv(\lambda;\phi_s,\phi)$ and $\tc(\lambda;\phi_s)$, and their derivatives with respect to $\lambda$ are continuous and bounded by $M'$ for all $\lambda\in[0,+\infty]$.
            It follows that $|\gcvdet|$ is upper bounded almost surely.
            Analogously, we have that $|\partial (\Ddet\gcvdet)/ \partial \lambda|$ is also upper bounded almost surely on $\lambda\in\Lambda$.
            From Moore-Osgood theorem and the continuity property from \Cref{lem:fixed-point-v-lambda-properties}, we have
            \begin{align*}
                \lim_{n\rightarrow \infty}\lim_{\lambda\rightarrow 0^+} P_{n,\lambda}' &= \lim_{\lambda\rightarrow 0^+}\lim_{n\rightarrow \infty} P_{n,\lambda}'
                =\lim_{\lambda\rightarrow 0^+}\gcvdet
                =\gcvdet[0].
            \end{align*}

        \end{proof}

\section{Auxiliary results on asymptotic equivalents}
\label{sec:calculus_asymptotic_equivalents}

\subsection{Preliminary background}
\label{app:sec:preliminary-background}

We use the notion of asymptotic equivalence of sequences of random matrices in various proofs.
This section provides a basic review of the related definitions and corresponding calculus rules.
See \citet{dobriban_wager_2018,dobriban_sheng_2021,patil2022mitigating,patil2022bagging} for more details.

\begin{definition}[Asymptotic equivalence]
    \label{def:deterministic-equivalent}
    Consider sequences $\{ \bA_p \}_{p \ge 1}$ and $\{ \bB_p \}_{p \ge 1}$ of (random or deterministic) matrices of growing dimensions.
    We say that $\bA_p$ and $\bB_p$ are asymptotically equivalent and write $\bA_p \asympequi \bB_p$ if $\lim_{p \to \infty} | \tr[\bC_p (\bA_p - \bB_p)] | = 0$ almost surely for any sequence of random matrices $\bC_p$ independent to $\bA_p$ and $\bB_p$, with bounded trace norm such that $\limsup_{p\rightarrow\infty} \| \bC_p \|_{\mathrm{tr}} < \infty$ almost surely.
\end{definition}

    The notion of asymptotic equivalence of two sequences of random matrices from \Cref{def:deterministic-equivalent} can be further extended to incorporate conditioning on another sequence of random matrices.

    \begin{definition}[Conditional asymptotic equivalence]
        \label{def:cond-deterministic-equivalent}
        Consider sequences $\{ \bA_p \}_{p \ge 1}$, $\{ \bB_p \}_{p \ge 1}$ and $\{ \bD_p \}_{p \ge 1}$ of (random or deterministic) matrices of growing dimensions.
        We say that $\bA_p$ and $\bB_p$ are equivalent given $\bD_p$ and write $\bA_p \asympequi \bB_p\mid \bD_p$ if $\lim_{p \to \infty} | \tr[\bC_p (\bA_p - \bB_p)] | = 0$ almost surely conditional on $\{\bD_p\}_{p\ge 1}$. In other words,
        \begin{align*}
            \PP\left(\lim\limits_{p\rightarrow\infty}|\tr[\bC_p(\bA_p-\bB_p)]| =0\given \{\bD_p\}_{p\ge 1}\right) =1,
        \end{align*}
        for any sequence of random matrices $\bC_p$, independent to $\bA_p$ and $\bB_p$ conditional on $\bD_p$, with bounded trace norm
        such that $\limsup \| \bC_p \|_{\mathrm{tr}} < \infty$
        as $p \to \infty$.
    \end{definition}

    Below we summarize the calculus rules for conditional asymptotic equivalence \Cref{def:cond-deterministic-equivalent} adapted from \citet[Lemma S.7.4 and S.7.6]{patil2022bagging}.
    \begin{lemma}[Calculus of asymptotic equivalents]
        \label{lem:calculus-detequi}
        Let $\bA_p$, $\bB_p$, $\bC_p$ and $\bD_p$ be sequences of random matrices.
        The calculus of asymptotic equivalents ($\asympequi_D$ and $\asympequi_R$) satisfies the following properties:
        \begin{enumerate}[label={(\arabic*)}]
            \item 
            \label{lem:calculus-detequi-item-equivalence}
            Equivalence:
            The relation $\asympequi$ is an equivalence relation.
            \item 
            \label{lem:calculus-detequi-item-sum}
            Sum:
            If $\bA_p \asympequi \bB_p\mid \bE_p$ and $\bC_p \asympequi \bD_p\mid \bE_p$, then $\bA_p + \bC_p \asympequi \bB_p + \bD_p\mid \bE_p$.
            \item 
            \label{lem:calculus-detequi-item-product}
            Product: If $\bA_p$ has bounded operator norms such that $\limsup_{p\rightarrow\infty}\| \bA_p \|_{\oper} < \infty$, $\bA_p$ is conditional independent to $\bB_p$ and $\bC_p$ given $\bE_p$ for $p\ge 1$, and $\bB_p \asympequi \bC_p\mid \bE_p$, then $\bA_p \bB_p \asympequi \bA_p \bC_p\mid \bE_p$.
            \item 
            \label{lem:calculus-detequi-item-trace}
            Trace:
            If $\bA_p \asympequi \bB_p\mid \bE_p$, then $\tr[\bA_p] / p - \tr[\bB_p] / p \to 0$ almost surely when conditioning on $ \bE_p$.
            \item 
            \label{lem:calculus-detequi-item-differentiation}
            Differentiation:
            Suppose $f(z, \bA_p) \asympequi g(z, \bB_p)\mid \bE_p$ where the entries of $f$ and $g$
            are analytic functions in $z \in S$ and $S$ is an open connected subset of $\CC$.
            Suppose for any sequence $\bC_p$ of deterministic matrices with bounded trace norm
            we have $| \tr[\bC_p (f(z, \bA_p) - g(z, \bB_p))] | \le M$ for every $p$ and $z \in S$.
            Then we have $f'(z, \bA_p) \asympequi g'(z, \bB_p)\mid \bE_p$ for every $z \in S$,
            where the derivatives are taken entrywise with respect to $z$.
            
            \item \label{lem:cond-calculus-detequi-item-uncond} Unconditioning: If $ \bA_p\asympequi \bB_p\mid \bE_p$, then $ \bA_p\asympequi \bB_p$.
            
            \item \label{lem:cond-calculus-detequi-item-substitute} 
            Substitution:
            Let $v:\RR^{p\times p}\rightarrow\RR$ and $f(v(\bC),\bC):\RR^{p\times p}\rightarrow\RR^{p\times p}$ be a matrix function for matrix $\bC\in\RR^{p\times p}$ and $p\in\NN$,  that is continuous in the first augment with respect to operator norm. If $v(\bC)\stackrel{\as}{=} v(\bD)$ such that $\bC$ is independent to $\bD$, then $f(v(\bC),\bC)\asympequi f(v(\bD),\bC)\mid \bC$.
        \end{enumerate}
    \end{lemma}

\subsection{Standard ridge resolvents and various extensions}
\label{append:det-equi-resol}

In this section,
we gather various asymptotic matrix equivalents.
\Cref{subsubsec:resolvent-standard} introduces the basic concepts and definitions.
The extended equivalents developed in the work of \citet{patil2022bagging} are summarized in \Cref{sec:asympequi-extended-subsample-resolvents}.
Based on the results in \Cref{subsubsec:resolvent-standard,sec:asympequi-extended-subsample-resolvents}, we prove some useful asymptotic equivalent relations in \Cref{subsubsec:resolvent-train}, which are subsequently used in the proof of \Cref{lem:conv-train-err} (that further relies on \Cref{lem:ridge-B0,lem:ridge-V0}).

\subsubsection{Standard ridge resolvents}\label{subsubsec:resolvent-standard}

The following lemma provides an asymptotic equivalent
for the standard ridge resolvent and implies \Cref{cor:asympequi-scaled-ridge-resolvent}. 
It is adapted from Theorem 1 of \citet{rubio_mestre_2011}.
See also Theorem 3 of \citet{dobriban_sheng_2021}.

\begin{lemma}[Asymptotic equivalent for standard ridge resolvent]
    \label{lem:basic-ridge-resolvent-equivalent}
    Suppose $\bx_i \in \RR^{p}$ for $i \in [n]$ are i.i.d.\ random vectors such that each $\bx_i = \bz_{i} \bSigma^{1/2}$, where $\bz_i$ is a random vector consisting of i.i.d.\ entries $z_{ij}$ for $j \in [p]$ satisfying $\EE[z_{ij}] = 0$, $\EE[z_{ij}^2] = 1$, and $\EE[|z_{ij}|^{8+\alpha}] \le M_\alpha$ for some constants $\alpha > 0$ and $M_\alpha < \infty$, and $\bSigma \in \RR^{p \times p}$ is a positive semidefinite matrix satisfying $0 \preceq \bSigma \preceq r_{\max} I_p$ for some constant $r_{\max} < \infty$ that is independent of $p$.
    Let $\bX \in \RR^{n \times p}$ the concatenated matrix with $\bx_i^\top$ for $i \in [n]$ as rows, and let $\hSigma \in \RR^{p \times p}$ denote the random matrix $\bX^\top \bX / n$.
    Let $\gamma := p / n$.
    Then, for $z \in \CC^{+}$, as $n, p \to \infty$ such that $0 < \liminf \gamma \le \limsup \gamma < \infty$, we have the following asymptotic equivalence:
    \begin{equation}
        (\hSigma - z \bI_p)^{-1}
        \asympequi
        (c(e(z; \gamma)) \bSigma - z \bI_p)^{-1}.
    \end{equation}
    Here the scalar $c(e(z; \gamma))$ is defined in terms of another scalar $e(z; \gamma)$ by the equation:
    \begin{equation}
        \label{eq:basic-ridge-equivalence-c-e-relation}
        c(e(z; \gamma))
        = \frac{1}{ 1 + \gamma e(z; \gamma)},
    \end{equation}
    and $e(z; \gamma)$ is the unique solution in $\CC^{+}$ to the following fixed-point equation:
    \begin{equation}
        \label{eq:basic-ridge-equivalence-e-fixed-point}
        e(z; \gamma)
        = 
        \tr[ \bSigma (c(e(z; \gamma)) \bSigma  - z I_p)^{-1} ] / p.
    \end{equation}
\end{lemma}

The following corollary is a simple consequence of \Cref{lem:basic-ridge-resolvent-equivalent}.
It supplies an asymptotic equivalent for the (regularization) scaled ridge resolvent.

\begin{corollary}[Asymptotic equivalent for scaled ridge resolvent]\label{cor:asympequi-scaled-ridge-resolvent}
    Assume the setting of \Cref{lem:basic-ridge-resolvent-equivalent}.
    For $\lambda > 0$, we have the following asymptotic equivalence:
    \[
        \lambda (\hSigma + \lambda \bI_p)^{-1}
        \asympequi
        (v(-\lambda; \gamma) \bSigma + \bI_p)^{-1}.
    \]
    Here $v(-\lambda; \gamma) > 0$ is the unique solution to the following fixed-point equation:
    \begin{align}
        \label{eq:basic-ridge-equivalence-v-fixed-point}
        \frac{1}{v(-\lambda; \gamma)}
        =
        \lambda
        + \gamma \int \frac{r}{1+v(-\lambda; \gamma)r} \, \rd H_n(r),
    \end{align}
    where $H_n$ is the empirical distribution of the eigenvalues of $\bSigma$ that is supported on $\RR_{+}$.
\end{corollary}

It is worth mentioning that the parameter $v(-\lambda; \gamma)$ in \Cref{cor:asympequi-scaled-ridge-resolvent} is the companion Stieltjes transform of the spectral distribution of the sample covariance matrix $\hSigma$. 
It is also the Stieltjes transform of the spectral distribution of the gram matrix $\bX \bX^\top / n$.

The following lemma uses \Cref{cor:asympequi-scaled-ridge-resolvent} along with calculus of asymptotic equivalents (from \Cref{lem:calculus-detequi}). It provides asymptotic equivalents for resolvents needed to obtain asymptotic bias and variance of standard ridge regression.
It is adapted from Lemma S.6.10 of \citet{patil2022mitigating}.

\begin{lemma}[Asymptotic equivalents for ridge resolvents associated with generalized bias and variance]
    \label{lem:deter-approx-generalized-ridge}
    Suppose $\bx_i \in \RR^{p}$ for $i \in [n]$ are i.i.d.\ random vectors with each $\bx_i = \bz_{i} \bSigma^{1/2}$, where $\bz_i \in \RR^{p}$ is a random vector that contains i.i.d.\ random variables $z_{ij}$ for $j \in [p]$ each with $\EE[z_{ij}] = 0$, $\EE[z_{ij}^2] = 1$, and $\EE[|z_{ij}|^{8+\alpha}] \le M_\alpha$ for some constants $\alpha > 0$ and $M_\alpha < \infty$, and $\bSigma \in \RR^{p \times p}$ is a positive semidefinite matrix with $r_{\min} \bI_p \preceq \bSigma \preceq r_{\max} \bI_p$ for some constants $r_{\min} > 0$ and $r_{\max} < \infty$  that is independent of $p$.
    Let $\bX \in \RR^{n \times p}$ be the concatenated random matrix with $\bx_i$, $1 \le i \le n$, as its rows, and define $\hSigma:=\bX^\top \bX / n \in \RR^{p \times p}$.
    Let $\gamma:= p / n$.
    Then, for $\lambda > 0$, as $n, p \to \infty$ with $0 < \liminf \gamma \le \limsup \gamma < \infty$, we have the following asymptotic equivalents:
    \begin{enumerate}[label={(\arabic*)}]
        \item\label{eq:detequi-ridge-genbias}  Bias of ridge regression:
        \begin{equation}
            \lambda^2
            (\hSigma + \lambda \bI_p)^{-1} \bA (\hSigma + \lambda \bI_p)^{-1}
            \asympequi 
            (v(-\lambda; \gamma,\bSigma) \bSigma + \bI_p)^{-1}
            (\tv_b(-\lambda; \gamma,\bSigma,\bA) \bSigma + \bA)
            (v(-\lambda; \gamma,\bSigma) \bSigma + \bI_p)^{-1}.
        \end{equation}

        \item\label{eq:detequi-ridge-genvar} Variance of ridge regression:
        \begin{equation}
            (\hSigma + \lambda \bI_p)^{-2} \hSigma \bA
            \asympequi
            \tv_v(-\lambda; \gamma,\bSigma) (v(-\lambda; \gamma,\bSigma) \bSigma + \bI_p)^{-2} \bSigma\bA.
        \end{equation}
    \end{enumerate}
    Here $v(-\lambda; \gamma,\bSigma) > 0$ is the unique solution to the fixed-point equation
    \begin{equation}
        \label{eq:def-v-ridge}
        \frac{1}{v(-\lambda; \gamma,\bSigma)}
        = \lambda+\int \frac{\gamma r}{1+v(-\lambda; \gamma,\bSigma) r} \, \rd H_n(r;\bSigma),
    \end{equation}
    and $\tv_b(-\lambda; \gamma,\bSigma)$ and $\tv_v(-\lambda; \gamma,\bSigma)$
    are defined through $v(-\lambda; \gamma,\bSigma)$ by the following equations:
    \begin{align}
        \tv_b(-\lambda; \gamma,\bSigma,\bA)
        &=
        \ddfrac{\gamma\tr[\bA \bSigma(v(-\lambda; \gamma,\bSigma)\bSigma + \bI_p)^{-2}]/p
        }{v(-\lambda; \gamma,\bSigma)^{-2}- \int \gamma r^2(1+v(-\lambda; \gamma,\bSigma) r)^{-2} \, \rd H_n(r;\bSigma)}
        , \label{eq:def-v-b-ridge}\\
        \tv_v(-\lambda; \gamma,\bSigma)^{-1}
       & = 
            v(-\lambda; \gamma,\bSigma)^{-2}
            - \int \gamma r^2(1+v(-\lambda; \gamma,\bSigma) r)^{-2} \, \rd H_n(r;\bSigma),
            \label{eq:def-tv-v-ridge}
    \end{align}
    where $H_n(\cdot;\bSigma)$ is the empirical distribution of the eigenvalues of $\bSigma$ that is supported on $[r_{\min}, r_{\max}]$.
\end{lemma}

Although \Cref{lem:deter-approx-generalized-ridge} states the dependency on $\bSigma$ explicitly, we will simply write $H_n(r)$, $v(-\lambda; \gamma)$, $\tv_b(-\lambda; \gamma, \bA)$, and $\tv_v(-\lambda; \gamma)$ to denote $H_n(r;\bSigma)$, $v(-\lambda; \gamma,\bSigma)$, $\tv_b(-\lambda; \gamma,\bSigma,\bA)$, and $\tv_v(-\lambda; \gamma,\bSigma)$, respectively, for simplifying notations when it is clear from the context.
When $\bA=\bSigma$, we simply write $\tv_b(-\lambda; \gamma) = \tv_b(-\lambda; \gamma,\bA)$.
The moment assumption of order $8 + \alpha$ for some $\alpha > 0$ in the above lemma can be relaxed to only requiring the existence of moments of order $4 + \alpha$ by a truncation argument as in the proof of Theorem 6 of \citet{hastie2022surprises} (in Appendix A.4 therein). We omit the details and refer the readers to \citet{hastie2022surprises}.

\subsubsection{Extended ridge resolvents}
\label{sec:asympequi-extended-subsample-resolvents}

The lemma below extends the asymptotic equivalents of the ridge resolvents in \Cref{lem:deter-approx-generalized-ridge} to provide asymptotic equivalents for Tikhonov resolvents, where the regularization matrix $\lambda \bI_p$ is replaced with $\lambda(\bI_p+ \bC)$ and $\bC\in\RR^{p\times p}$ is an arbitrary positive semidefinite random matrix.

\begin{lemma}[Tikhonov resolvents, adapted from \citet{patil2022bagging}]\label{lem:deter-approx-ridge-extend}
    Suppose the conditions in Lemma \ref{lem:deter-approx-generalized-ridge} holds.
    Let $\bC\in\RR^{p\times p}$ be any symmetric and positive semidefinite random matrix with uniformly bounded operator norm in $p$ that is independent to $\bX$ for all $n,p\in\NN$, and let $\bN=(\hSigma + \lambda \bI_p)^{-1}$.
    Then, for $\lambda > 0$, as $n, p \to \infty$ with $0 < \liminf \gamma \le \limsup \gamma < \infty$, we have the following asymptotic equivalents:
    \begin{enumerate}[label={(\arabic*)}]

        \item \label{eq:lem:deter-approx-ridge-extend}Tikhonov resolvent:
        \begin{align}
            \lambda(\bN^{-1} + \lambda\bC )^{-1} &\asympequi \tilde{\bSigma}_{\bC}^{-1}. 
        \end{align}
        
        \item \label{eq:lem:deter-approx-ridge-extend-gbias}Bias of Tikhonov regression:
        \begin{align}
            \lambda^2 (\bN^{-1} + \lambda\bC )^{-1}\bSigma (\bN^{-1} + \lambda\bC )^{-1} &\asympequi    \tilde{\bSigma}_{\bC}^{-1} (\tv_b(-\lambda; \gamma,\bSigma_{\bC})\bSigma+\bSigma)\tilde{\bSigma}_{\bC}^{-1}.
        \end{align}
        
        \item \label{eq:lem:deter-approx-ridge-extend-gvar}Variance of Tikhonov regression:
        \begin{align}
            (\bN^{-1} + \lambda\bC )^{-1}\hSigma(\bN^{-1} + \lambda\bC )^{-1}\bSigma  &\asympequi \tv_v(-\lambda; \gamma,\bSigma_{\bC})\tilde{\bSigma}_{\bC}^{-1}\bSigma\tilde{\bSigma}_{\bC}^{-1}\bSigma,
        \end{align}
    \end{enumerate}
    where $\bSigma_{\bC}= (\bI_p+ \bC)^{-\frac{1}{2}}\bSigma  (\bI_p+ \bC)^{-\frac{1}{2}}$, $\tilde{\bSigma}_{\bC}=v(-\lambda; \gamma,\bSigma_{\bC}) \bSigma + \bI_p+\bC $.
    Here $v(-\lambda; \gamma,\bSigma_{\bC})$, $\tilde{v}_b(-\lambda; \gamma,\bSigma_{\bC})$, and $\tilde{v}_v(-\lambda; \gamma,\bSigma_{\bC})$ defined by \eqref{eq:def-v-ridge}-\eqref{eq:def-tv-v-ridge} simplify to the following:
    \begin{align}
        \frac{1}{v(-\lambda; \gamma,\bSigma_{\bC})}
             &= \lambda
            + \gamma \tr[ (v(-\lambda; \gamma,\bSigma_{\bC}) \bSigma + \bI_p +\bC)^{-1}\bSigma] / p, \label{eq:def-v-c-ridge} \\
            \frac{1}{\tilde{v}_v(-\lambda; \gamma,\bSigma_{\bC})} &= \frac{1}{v(-\lambda; \gamma,\bSigma_{\bC})^2} - \gamma \tr[ (v(-\lambda; \gamma,\bSigma_{\bC}) \bSigma + \bI_p +\bC)^{-2}\bSigma^2] / p, \label{eq:def-tv-v-c-ridge}\\
            \tilde{v}_b(-\lambda; \gamma,\bSigma_{\bC}) &=  \gamma \tr[ (v(-\lambda; \gamma,\bSigma_{\bC}) \bSigma + \bI_p +\bC)^{-2}\bSigma^2] / p \cdot \tilde{v}_v(-\lambda; \gamma,\bSigma_{\bC}). \label{eq:def-tv-b-c-ridge}
    \end{align}
    \end{lemma}

\subsubsection{Resolvents for training error}\label{subsubsec:resolvent-train}    
The following lemma concerns the asymptotic equivalents of quantities that arise in the proof for \Cref{lem:conv-train-err}.

\begin{lemma}[Resolvents for in-sample error]\label{lem:resolv-insample} Suppose the conditions in Lemma \ref{lem:deter-approx-generalized-ridge} holds.
    Let $\bC\in\RR^{p\times p}$ be any symmetric and positive semidefinite random matrix with uniformly bounded operator norm in $p$ that is independent to $\bX$ for all $n,p\in\NN$.
    Let $I_1,I_2\overset{\SRS}{\sim} \cI_k$ and $\hSigma_j$ be the sample covariance matrix computed using $k$ observations of $\bX$ indexed by $I_j$ ($j=0,1$).
    For $j=1,2$, let $\bM_j=(\hSigma_j + \lambda \bI_p)^{-1}$ be the resolvent for $\hSigma_j$.
    Then, as $k,n,p\rightarrow\infty$ such that $p/n\rightarrow\phi\in(0,\infty)$ and $p/k\rightarrow\phi_s\in[\phi,\infty)$, we have the following asymptotic equivalents:
\begin{enumerate}[(1)]
    \item\label{item:lem-ridge-ind} Independent product with sample covariance:
    \begin{align*}
        \bC \hSigma_j \asympequi \bC \bSigma.
    \end{align*}

    \item\label{item:lem-ridge-B0-term-det-1} Bias term 1:
    \begin{align}
        \lambda^2\bM_1\bC\bM_2 \asympequi \left(v(-\lambda; \phi_s) \bSigma + \bI_p\right)^{-1} (\tv(-\lambda;\phi,\phi_s,\bC)\bSigma+\bC)\left(v(-\lambda; \phi_s) \bSigma + \bI_p\right)^{-1}. \label{eq:lem-ridge-B0-term-det-1}
    \end{align}

    \item\label{item:lem-ridge-B0-term-det-2} Bias term 2:

    \begin{align}
        \bM_1 \hSigma_{1\cap 2} \bM_2 \bC&\asympequi \tv_v(-\lambda;\phi,\phi_s)( v(-\lambda; \phi_s) \bSigma + \bI_p)^{-2}\bSigma\bC,\label{eq:lem-ridge-B0-term-det-2}
    \end{align}

    \item\label{item:lem-ridge-V0-term-det-1} Variance term 1:
    \begin{align}
        \bM_1 \hSigma_{1\cap 2} &\asympequi  \bI_p - (v(-\lambda;\phi_s) \bSigma+\bI_p)^{-1} ,\label{eq:lem-ridge-V0-term-det-1}
    \end{align}

    \item\label{item:lem-ridge-V0-term-det-2} Variance term 2:
    \begin{align}
        \bM_1 \hSigma_{1\cap 2} \bM_2 \hSigma_{1\cap 2} &\asympequi \frac{\phi_s}{\phi}\left(v(-\lambda;\phi_s)-\frac{\phi_s-\phi}{\phi_s}\lambda \tv_v(-\lambda;\phi,\phi_s) \right)(v(-\lambda;\phi_s)\bSigma+\bI_p)^{-1}\bSigma \notag\\
        &\qquad - \lambda \tv_v(-\lambda;\phi,\phi_s)(v(-\lambda;\phi_s)\bSigma+\bI_p)^{-2}\bSigma,\label{eq:lem-ridge-V0-term-det-2}
    \end{align}
\end{enumerate}
where 
    \begin{align*}
        \tv(-\lambda;\phi,\phi_s,\bC) &=\ddfrac{\lim\limits_{k,n,p} \phi \tr[\bC\bSigma(v(-\lambda; \phi_s)\bSigma+\bI_p)^{-2}]/p}{v(-\lambda; \phi_s)^{-2}-\phi \int\frac{r^2}{(1+v(-\lambda; \phi_s)r)^2}\rd H(r)},\\
        \tv_v(-\lambda;\phi,\phi_s) &:= \ddfrac{1}{v(-\lambda; \phi_s)^{-2}-\phi \int\frac{r^2}{(1+ v(-\lambda; \phi_s)r)^2}\,\rd H(r)}. 
    \end{align*}
\end{lemma}
\begin{proof}[Proof of \Cref{lem:resolv-insample}]
    We split the proof into different parts below.
    
    \paragraph{Part (1)} 
    
    Note that $\tr(\hSigma_j) =  \sum_{i\in I_j}\|\bx_i\|_2^2/k$ and $\bx_i = \bz_i^{\top}\bSigma \bz_i$.
    By \Cref{lem:concen-quadform}, we have that $\tr(\hSigma_j)/p-\tr(\bSigma)/p\asto 0$.
    Since $\norm{\bC}_{\oper}$ is uniformly upper bounded and
    \begin{align*}
        \left|\frac{1}{p}\tr(\bC\hSigma_j) -\frac{1}{p}\tr(\bC\bSigma)\right| & \leq \frac{1}{p}|\tr(\bC(\hSigma_j-\bSigma))| \leq \frac{1}{p} \norm{\bC}_{\oper} |\tr(\hSigma_j-\bSigma)|,
    \end{align*}
    it follows that $\frac{1}{p}\tr(\bC\hSigma_j) -\frac{1}{p}\tr(\bC\bSigma)\asto 0$, which implies that $\bC\hSigma_j\asympequi \bC\bSigma$

    \paragraph{Part (2)} This is a direct consequence of \citet[Part (c) of the proof for Lemma S.24]{patil2022bagging}.
    
    \paragraph{Part (3)} This is a direct consequence of \citet[Part (c) of the proof for Lemma S.25]{patil2022bagging}.

    \paragraph{Part (4)} Let $i_0=|I_1\cap I_2|$.
         Conditioning on $\hSigma_{1\cap 2}$ and $i_0$, from \Cref{def:cond-deterministic-equivalent} and \Cref{lem:deter-approx-ridge-extend}~\ref{eq:lem:deter-approx-ridge-extend} we have
        \begin{align*} 
            \lambda\bM_1 \asympequi\bM^{\det}_{\bM_{1\cap 2},i_0}&:= \frac{k}{k-i_0}\left(v_1\bSigma +\bI_p+ \bC_1\right)^{-1} \,\Big|\, i_0,
        \end{align*}
        where $v_1=v(-\lambda; \gamma_1,\bSigma_{\bC_1})$, $\bSigma_{\bC_1}=(\bI_p+\bC_1)^{-\frac{1}{2}}\bSigma (\bI_p+\bC_1)^{-\frac{1}{2}}$, $\bC_1=i_0(\lambda(k-i_0))^{-1}\bM_{1\cap 2}^{-1}$, and $\gamma_1=p/(k-i_0)$. Here the subscripts of $v_1$ and $\bC_1$ are related to the aspect ratio $\gamma_1$. 
        Because
        \begin{align*}
            \limsup\norm{\hSigma_{1\cap 2}}_{\oper} \leq r_{\max}(1+\sqrt{\phi_s^2/\phi})^2,
        \end{align*}
        almost surely as $k,n,p\rightarrow\infty$ such that $p/n\rightarrow \phi$ and $p/k\rightarrow \phi_s$, by \Cref{lem:calculus-detequi}~\ref{lem:calculus-detequi-item-product}, we have $$\bM_1\hSigma_{1\cap 2} \asympequi \lambda^{-1}\bM^{\det}_{\bM_{1\cap 2},i_0}\hSigma_{1\cap 2}\mid i_0.$$
        That is,
        \begin{align*}
            \bM_1\hSigma_{1\cap 2} \asympequi  \frac{k}{i_0} (\bM_{1\cap 2}^{-1}+\lambda\bC_0)^{-1}\hSigma_{1\cap 2} \mid i_0,
        \end{align*}
        where $\bC_0=(k-i_0)/i_0\cdot(v_1\bSigma+\bI_p)$.
        Define $\bSigma_{\bC_0}=(\bI+\bC_0)^{-\frac{1}{2}}\bSigma(\bI+\bC_0)^{-\frac{1}{2}}$.
        Conditioning on $i_0$, by \Cref{lem:deter-approx-ridge-extend}~\ref{eq:lem:deter-approx-ridge-extend}, 
        we have
        \begin{align*}
            \tr[\bSigma_{\bC_1}(v_1\bSigma_{\bC_1}+\bI_p)^{-1}] &= \tr[\bSigma(v_1\bSigma+\bI_p+\bC_1)^{-1}]\\
            &= \frac{\lambda(k-i_0)}{i_0}\tr\left[\bSigma\left(\bM_{1\cap 2}^{-1}+ \frac{\lambda(k-i_0)}{i_0}( v_1\bSigma+\bI_p)\right)^{-1}\right]\\
            &\stackrel{\as}{=} \frac{k-i_0}{i_0}\tr\left[\bSigma\left(v_0\bSigma+\bI_p+ \frac{k-i_0}{i_0}( v_1\bSigma+\bI_p)\right)^{-1}\right]\\
            &= \tr\left[\bSigma\left(
            \left(\frac{i_0}{k-i_0}v_0 + v_1 \right)\bSigma+ \frac{k}{k-i_0}\bI_p\right)^{-1}\right],
        \end{align*}
        where $v_0=v(-\lambda;\gamma_0,\bSigma_{\bC_0})$and $\gamma_0=p/i_0$.
        Note that the fixed-point solution $v_0$ depends on $v_1$. The fixed-point equations reduce to
        \begin{align*}
            \frac{1}{v_0} &=\lambda + \gamma_0 \tr[\bSigma_{\bC_0}(v_0\bSigma_{\bC_0}+\bI_p)^{-1}]/p= \lambda + \frac{p}{k}\tr \left[\bSigma\left(\left(\frac{i_0}{k}v_0+\frac{k-i_0}{k}v_1\right)\bSigma + \bI_p\right)^{-1}\right]/p \\
            \frac{1}{v_1} &= \lambda + \gamma_1 \tr[\bSigma_{\bC_1}(v_1\bSigma_{\bC_1}+\bI_p)^{-1}]/p= \lambda + \frac{p}{k} \tr\left[\bSigma\left(
            \left(\frac{i_0}{k}v_0 +\frac{k-i_0}{k} v_1\right)\bSigma+ \bI_p\right)^{-1}\right]/p
        \end{align*}
        almost surely.
        Note that the solution $(v_0, v_1)$ to the above equations is a pair of positive numbers and does not depend on samples.
        If $(v_0, v_1)$ is a solution to the above system, then $(v_1, v_0)$ is also a solution. Thus, any solution to the above equations must be unique.
        On the other hand, since $v_0=v_1=v(-\lambda;p/k)$ satisfies the above equations, it is the unique solution.
        By \Cref{lem:calculus-detequi}~\ref{lem:cond-calculus-detequi-item-substitute}, we can replace $v(-\lambda;\gamma_1,\bSigma_{\bC_1})$ by the solution $v_0=v_1=v(-\lambda; p/k)$ of the above system, which does not depend on samples. Thus,
        \begin{align}
            \bM_1\hSigma_{1\cap 2} \asympequi = \frac{k}{i_0} (\bM_{1\cap 2}^{-1}+\lambda\bC^*)^{-1}\hSigma_{1\cap 2} \mid i_0,\label{eq:ridge-B0-cross-term}
        \end{align}
        where $\bC^*= (k-i_0)/i_0\cdot(v(-\lambda; p/k)\bSigma+\bI_p)$.
        Again from \Cref{lem:deter-approx-ridge-extend}~\ref{eq:lem:deter-approx-ridge-extend} we have
        \begin{align*}
            (\bM_{1\cap 2}^{-1}+\lambda\bC^*)^{-1}\hSigma_{1\cap 2} &= \bI_p - \lambda (\bM_{1\cap 2}^{-1}+\lambda\bC^*)^{-1}(\bI_p+\bC^*)\\
            &\asympequi \bI_p - (v(-\lambda;p/k) \bSigma+\bI_p+\bC^*)^{-1}(\bI_p+\bC^*) \mid i_0\\
            &= \frac{i_0}{k}(\bI_p - (v(-\lambda;p/k) \bSigma+\bI_p)^{-1}).
        \end{align*}
        Finally, from \Cref{lem:calculus-detequi}~\ref{lem:cond-calculus-detequi-item-uncond}, we have
        $$\bM_1 \hSigma_{1\cap 2} \asympequi \bI_p - (v(-\lambda;p/k) \bSigma+\bI_p)^{-1} \asympequi \bI_p - (v(-\lambda;\phi_s) \bSigma+\bI_p)^{-1}.$$

    \paragraph{Part (5)} From \citet[Part (c) of the proof for Lemma S.2.5]{patil2022bagging}, we have that
    \begin{align*}
        \bM_1\hSigma_{1\cap 2}\bM_2\hSigma_{1\cap 2} \asympequi \frac{k^2}{i_0^2} (\bM_{1\cap 2}^{-1} + \lambda \bC^*)^{-1}\hSigma_{1\cap 2} (\bM_{1\cap 2}^{-1} + \lambda \bC^*)^{-1}\hSigma_{1\cap 2},
    \end{align*}
    where $\bM_{1\cap 2}= (\hSigma_{1\cap 2}+\lambda\bI_p)^{-1}$ and $\bC^*=(k-i_0)/i_0(v(-\lambda; \phi_s)\bSigma +\bI_p)$.
    Since
    \begin{align*}
        (\bM_{1\cap 2}^{-1} + \lambda \bC^*)^{-1}\hSigma_{1\cap 2}  =  \bI_p - \lambda (\bM_{1\cap 2}^{-1} + \lambda \bC^*)^{-1}  (\bI_p+\bC^*),
    \end{align*}
    we have
    \begin{align}
         &\bM_1\hSigma_{1\cap 2}\bM_2\hSigma_{1\cap 2} \notag\\
         &\asympequi  \frac{k^2}{i_0^2} (\bM_{1\cap 2}^{-1} + \lambda \bC^*)^{-1}\hSigma_{1\cap 2} - \lambda \frac{k^2}{i_0^2} (\bM_{1\cap 2}^{-1} + \lambda \bC^*)^{-1}\hSigma_{1\cap 2} (\bM_{1\cap 2}^{-1} + \lambda \bC^*)^{-1}  (\bI_p+\bC^*) \notag\\
         &= \frac{k^2}{i_0^2}(\bI_p - \lambda (\bM_{1\cap 2}^{-1} + \lambda \bC^*)^{-1}  (\bI_p+\bC^*))  - \lambda \frac{k^2}{i_0^2} (\bM_{1\cap 2}^{-1} + \lambda \bC^*)^{-1}\hSigma_{1\cap 2} (\bM_{1\cap 2}^{-1} + \lambda \bC^*)^{-1}  (\bI_p+\bC^*)\label{eq:item-5-1}
    \end{align}
    From \Cref{lem:deter-approx-ridge-extend} \ref{eq:lem:deter-approx-ridge-extend} and \ref{eq:lem:deter-approx-ridge-extend-gvar}, we have that 
    \begin{align*}
        \lambda (\bM_{1\cap 2}^{-1} + \lambda \bC^*)^{-1} &\asympequi \frac{\phi}{\phi_s} (v(-\lambda;\phi_s)\bSigma+\bI_p)^{-1} \\
        (\bM_{1\cap 2}^{-1} + \lambda \bC^*)^{-1}\hSigma_{1\cap 2} (\bM_{1\cap 2}^{-1} + \lambda \bC^*)^{-1}  &\asympequi \frac{\phi^2}{\phi_s^2}\tv_v(-\lambda;\phi,\phi_s) (v(-\lambda;\phi_s)\bSigma+\bI_p)^{-2}\bSigma.
    \end{align*}
    Combing the above two equivalents, the expression in \eqref{eq:item-5-1} can be further simplified as:
    \begin{align*}
        \bM_1\hSigma_{1\cap 2}\bM_2\hSigma_{1\cap 2} &\asympequi \frac{\phi_s}{\phi}\left(v(-\lambda;\phi_s)-\frac{\phi_s-\phi}{\phi_s}\lambda \tv_v(-\lambda;\phi,\phi_s) \right)(v(-\lambda;\phi_s)\bSigma+\bI_p)^{-1}\bSigma \\
        &\qquad - \lambda \tv_v(-\lambda;\phi,\phi_s)(v(-\lambda;\phi_s)\bSigma+\bI_p)^{-2}\bSigma.
    \end{align*}
\end{proof}

\subsection{Analytic properties of associated fixed-point equations}
\label{app:sec:analytic-properties}

In this section, we compile results related to the analytical properties of the fixed-point solution $v(-\lambda;\phi)$, as defined in \eqref{eq:basic-ridge-equivalence-v-fixed-point}.

The subsequent lemma establishes the existence and uniqueness of the solution $v(-\lambda;\phi)$. 
The properties of the derivatives outlined in Lemma \ref{lem:properties-sol} correspond with the properties of $\tv_v(-\lambda;\phi)$, as defined in Lemma \ref{lem:ridge-fixed-point-v-properties}. 

\begin{lemma}[Properties of the solution to the fixed-point equation, adapted from \citet{patil2022bagging}]\label{lem:properties-sol}
    Let $\lambda,\phi,a > 0$ and $b < \infty$ be real numbers.
    Let $P$ be a probability measure supported
    on $[a, b]$.
    Define the function $f$ such that
    \begin{align}
        f(x) = \frac{1}{x} - \phi \int\frac{r}{1 + r x}\rd P(r) - \lambda .\label{eq:lem:properties-sol:fx}
    \end{align}
    Then the following properties hold:
    \begin{enumerate}[label={(\arabic*)}]
        \item \label{lem:properties-sol-item-f-ridgeless}For $\lambda=0$ and $\phi\in(1,\infty)$, there is a unique $x_0\in(0,\infty)$ such that $f(x_0)=0$. The function $f$ is positive and strictly decreasing over $(0, x_0)$ and negative over $(x_0,\infty)$, with $\lim_{x \to 0^{+}} f(x) = \infty$ and $\lim_{x \to \infty} f(x) = 0$.
        
        \item \label{lem:properties-sol-item-f-ridge}For $\lambda>0$ and $\phi\in(0,\infty)$, there is a unique $x_0^{\lambda}\in(0,\infty)$ such that $f(x_0^{\lambda})=0$. The function $f$ is positive and strictly decreasing over $(0, x_0^{\lambda})$ and negative over $(x_0^{\lambda},\infty)$, with $\lim_{x \to 0^{+}} f(x) = \infty$ and $\lim_{x \to \infty} f(x) = -\lambda$.

        \item \label{lem:properties-sol-item-f'-ridgeless}For $\lambda=0$ and $\phi\in(1,\infty)$, $f$ is differentiable on $(0,\infty)$ and its derivative $f'$ is strictly increasing over $(0, x_0)$, with $\lim_{x \to 0^{+}} f'(x) = - \infty$ and $f'(x_0) < 0$.
        
        \item \label{lem:properties-sol-item-f'-ridge}For $\lambda>0$ and $\phi\in(0,\infty)$, $f$ is differentiable on $(0,\infty)$ and its derivative $f'$ is strictly increasing over $(0, \infty)$,
        with $\lim_{x \to 0^{+}} f'(x) = - \infty$
        and $f'(x_0^{\lambda}) < 0$.
        
    \end{enumerate}
\end{lemma}

The properties of the function $\phi \mapsto v(-\lambda;\phi)$, its continuity and limiting behavior, are provided for ridge regression (when $\lambda>0$), in Lemma \ref{lem:ridge-fixed-point-v-properties}, and ridgeless regression (when $\lambda=0$), in Lemma \ref{lem:fixed-point-v-properties}.

\begin{lemma}[Continuity properties in the aspect ratio for ridge regression, adapted from \citet{patil2022bagging}]
    \label{lem:ridge-fixed-point-v-properties}
    Let $\lambda,a > 0$ and $b < \infty$ be real numbers.
    Let $P$ be a probability measure supported on $[a, b]$.
    Consider the function $v(-\lambda; \cdot) : \phi \mapsto v(-\lambda; \phi)$, over $(0, \infty)$, where $v(-\lambda; \phi) > 0$ is the unique solution to the following fixed-point equation:
    \begin{equation}
        \label{eq:ridge-fixed-point-gen-phi}
        \frac{1}{v(-\lambda; \phi)}
         = \lambda + \phi \int \frac{r}{1+rv(-\lambda; \phi)} \rd P(r).
    \end{equation}
    Then the following properties hold:
    \begin{enumerate}[label={(\arabic*)}]
        \item 
        \label{lem:ridge-fixed-point-v-properties-item-v-bound}
        The range of the function $v(-\lambda; \cdot)$ is a subset of $(0,\lambda^{-1})$.
        
        \item 
        \label{lem:ridge-fixed-point-v-properties-item-v-properties}
        The function $v(-\lambda; \cdot)$ is continuous and strictly decreasing over $(0, \infty)$. Furthermore, $\lim_{\phi \to 0^{+}} v(-\lambda; \phi) = \lambda^{-1}$, and $\lim_{\phi \to \infty} v(-\lambda; \phi) = 0$.

        \item
        \label{lem:ridge-fixed-point-v-properties-item-vv-properties}
        The function 
        $\tv_v(-\lambda; \cdot) : \phi \mapsto \tv_v(-\lambda; \phi)$,
        where
        \[
           \tv_v(-\lambda; \phi) =
           \left(
                v(-\lambda; \phi)^{-2}
                - \int \phi  r^2(1 + r v(-\lambda; \phi))^{-2} 
                \, \mathrm{d}P(r)
           \right)^{-1},
        \]
        is positive and continuous over $(0, \infty)$.
        Furthermore,
        $\lim_{\phi \to 0^{+}} \tv_v(-\lambda; \phi) = \lambda^{-2}$,
        and $\lim_{\phi \to \infty} \tv_v(-\lambda; \phi) = 0$.
        
        \item
        \label{lem:ridge-fixed-point-v-properties-item-vb-properties}
        The function 
        $\tv_b(-\lambda; \cdot) : \phi \mapsto \tv_b(-\lambda; \phi)$,
        where
        \[
            \tv_b(-\lambda; \phi)
            = \tv_v(-\lambda; \phi)
            \int
            \phi r^2(1 + v(-\lambda; \phi) r)^{-2}
            \, \mathrm{d}P(r),
        \]
        is positive and continuous over $(0, \infty)$.
        Furthermore,
        $\lim_{\phi \to 0^{+}} \tv_b(-\lambda; \phi) =\lim_{\phi \to \infty} \tv_b(-\lambda; \phi) = 0$.
    \end{enumerate}
\end{lemma}

\begin{lemma}[Continuity properties in the aspect ratio for ridgeless regression, adapted from \citet{patil2022mitigating}]
    \label{lem:fixed-point-v-properties}
    Let $a > 0$ and $b < \infty$ be real numbers.
    Let $P$ be a probability measure supported on $[a, b]$.
    Consider the function $v(0; \cdot) : \phi \mapsto v(0; \phi)$, over $(1, \infty)$, where $v(0; \phi) > 0$ is the unique solution to the following fixed-point equation:
    \begin{equation}
       \label{eq:fixed-point-gen-phi}
        \frac{1}{\phi}
        = \int \frac{v(0; \phi) r}{1 + v(0; \phi) r} \, \mathrm{d}P(r).
    \end{equation}
    Then the following properties hold:
    \begin{enumerate}[label={(\arabic*)}]
        \item 
        \label{lem:fixed-point-v-properties-item-v-properties}
        The function $v(0; \cdot)$ is continuous and strictly decreasing over $(1, \infty)$.
        Furthermore, $\lim_{\phi \to 1^{+}} v(0; \phi) = \infty$,
        and $\lim_{\phi \to \infty} v(0; \phi) = 0$.
        
        \item \label{lem:fixed-point-v-properties-item-phivinverse-properties}
        The function 
        $\phi \mapsto (\phi v(0; \phi))^{-1}$ is strictly increasing over $(1, \infty)$.
        Furthermore,
        $\lim_{\phi \to 1^{+}} (\phi v(0; \phi))^{-1} = 0$
        and $\lim_{\phi \to \infty} (\phi v(0; \phi))^{-1} = 1$.
        
        \item
        
        \label{lem:fixed-point-v-properties-item-vv-properties}
        The function 
        $\tv_v(0; \cdot) : \phi \mapsto \tv_v(0; \phi)$,
        where
        \[
           \tv_v(0; \phi) =
           \left(v(0; \phi)^{-2}
                - \phi
                \int r^2(1 + r v(0; \phi))^{-2} 
                \, \mathrm{d}P(r)
           \right)^{-1},
        \]
        is positive and continuous over $(1, \infty)$.
        Furthermore,
        $\lim_{\phi \to 1^{+}} \tv_v(0; \phi) = \infty$,
        and $\lim_{\phi \to \infty} \tv_v(0; \phi) = 0$.
        \item
        \label{lem:fixed-point-v-properties-item-vb-properties}
        The function 
        $\tv_b(0; \cdot) : \phi \mapsto \tv_b(0; \phi)$,
        where
        \[
            \tv_b(0; \phi)
            = \tv_v(0; \phi
            )
            \int
            r^2(1 + v(0; \phi) r)^{-2}
            \, \mathrm{d}P(r),
        \]
        is positive and continuous over $(1, \infty)$.
        Furthermore,
        $\lim_{\phi \to 1^{+}} \tv_b(0; \phi) = \infty$, and $\lim_{\phi \to \infty} \tv_b(0; \phi) = 0$.
    \end{enumerate}
\end{lemma}

The continuity and differentiabilty properties of the function $\lambda\mapsto v(-\lambda;\phi)$ on a closed interval $[0,\lambda_{\max}]$ (for certain constant $\lambda_{\max}$) for $\phi\in(1,\infty)$ are detailed in \Cref{lem:fixed-point-v-lambda-properties}. 
The lemma is adapted from \citet{patil2022mitigating}.
This guarantees that $v(0;\phi)=\lim_{\lambda\rightarrow0^+}v(-\lambda;\phi)$ is well-defined for $\phi>1$, and additionally also implies that the related functions are bounded.

\begin{lemma}[Differentiability properties in the regularization parameter]
    \label{lem:fixed-point-v-lambda-properties}
    Let $0<a \leq b<\infty$ be real numbers.
    Let $P$ be a probability measure supported on $[a, b]$.
    Let $\phi>0$ be a real number.
    Let $\Lambda = [0, \lambda_{\max}]$ for some constant $\lambda_{\max}\in(0,\infty)$.
    For $\lambda \in \Lambda$, let $v(-\lambda; \phi) > 0$ denote the solution to the fixed-point equation
    \[
        \frac{1}{v(-\lambda; \phi)}
        =  \lambda
        + \phi \int \frac{r}{v(-\lambda; \phi) r + 1} \, \mathrm{d}P(r).
    \]
    When $\lambda=0$ and $\phi\in(0,1]$, $v(-\lambda;\phi):=+\infty$.
    Then the following properties hold:
    \begin{enumerate}[(1)]
        \item\label{lem:fixed-point-v-lambda-properties-item-monotonicity} (Monotonicity) For $\phi\in(0,\infty)$, the function $\lambda \mapsto v(-\lambda; \phi)$ is strictly decreasing in $\lambda\in[0,\infty)$.
        
        \item\label{lem:fixed-point-v-lambda-properties-item-differentiability} (Differentiability) 
        For $\phi\in(1,\infty)$, the function $\lambda \mapsto v(-\lambda; \phi)$ is twice differentiable over $\Lambda$.
        
        \item\label{lem:fixed-point-v-lambda-properties-item-boundedness} (Boundedness of the second derivative) 
        For $\phi\in(1,\infty)$, $v(-\lambda; \phi)$, $\partial / \partial \lambda [v(-\lambda; \phi)]$, and $\partial^2 / \partial \lambda^2 [v(-\lambda; \phi)]$ are bounded over $\Lambda$.
        
    \end{enumerate}
\end{lemma}
\begin{proof}[Proof of \Cref{lem:fixed-point-v-lambda-properties}]
    Start by re-writing the fixed-point equation as
    \[
        \lambda
        = \frac{1}{v(-\lambda; \phi)}
        - \phi  \int \frac{r}{v(-\lambda; \phi) r  + 1} \, \mathrm{d}P(r).
    \]
    Define a function $f$ by
    \[
        f(x) = \frac{1}{x} - \phi \int \frac{r}{x r + 1} \, \mathrm{d}P(r).
    \]
    Observe that $v(-\lambda; \phi) = f^{-1}(\lambda)$.
    We next study various properties of $f$ and prove the different parts in the statement.

    \paragraph{Part (1)}
    \underline{Properties of $f$ and $f^{-1}$:}

        Observe that
        \[
            f(x)
            =
            \frac{1}{x}
            - \phi \int \frac{r}{xr + 1}
            \, \mathrm{d}P(r)
            =
            \frac{1}{x}
            \left(
                1 - \phi \int \frac{xr}{xr + 1}
                \, \mathrm{d}P(r)
            \right).
        \]
        The function $g: x \mapsto 1/x$ is positive and strictly decreasing over $(0, \infty)$ with $\lim_{x \to 0^{+}} g(x) = \infty$ and $\lim_{x \to \infty} g(x) = 0$, while the function
        \[
            h: x \mapsto
            1 - \phi \int \frac{xr}{xr + 1}
            \, \mathrm{d}P(r)
        \]
        is strictly decreasing over $(0, \infty)$ with $h(0) = 1$ and $\lim_{x \to \infty} h(x) = 1 - \phi$.
        
        Thus, there is a unique $0 < x_0 < \infty$ when $\phi>1$ such that $h(x_0) = 0$, and consequently $f(x_0) = 0$; and $x_0=+\infty$ when $\phi\in(0,1]$ such that $g(x_0)=0$, and consequently $f(x_0)=0$.
        Because $h$ and $g$ are positive over $[0, x_0)$, $f$, a product of two positive strictly decreasing functions, is strictly decreasing over $(0, x_0)$, with $\lim_{x \to 0^{+}} f(x) = \infty$ and $f(x_0) = 0$.
        
        Because $f$ is strictly decreasing over $(0, x_0)$, $f^{-1}$ is strictly decreasing (see, e.g., Problem 2, Chapter 5 of \citet{rudin_1976}). 
       Since $f(x_0) = 0$, $f^{-1}(0) = x_0$, and since $\lim_{x \to 0^{+}} f(x) = \infty$, $\lim_{y \to \infty} f^{-1}(y) = 0$.
       Hence, $f^{-1}$ is strictly decreasing over $[0, \infty)$ for all $\phi>0$ and bounded above by $x_0 < \infty$ for all $\phi>1$.

    \paragraph{Parts (2) and (3)}
    We will prove the remaining two parts together.
    
   \underline{Properties of $f'$ and $(f^{-1})'$}:
   
    The derivative $f'$ at $x$ is given by
    \[
        f'(x)
        = - \frac{1}{x^2}
        + \phi \int \frac{r^2}{(x r + 1)^2}
        \, \mathrm{d}P(r)
        = 
        -
        \frac{1}{x^2}
        \left(
            1 - \phi \int \left(\frac{xr}{xr + 1}\right)^2
            \, \mathrm{d}P(r)
        \right).
    \]
    The function $g: x \mapsto 1/x^2$ is positive and strictly decreasing over $(0, \infty)$ with $\lim_{x \to 0^{+}} g(x) = \infty$ and $\lim_{x \to \infty} g(x) = 0$.
    On the other hand, the function
    \[
        h: 
        x \mapsto
        1 - \phi \int \left( \frac{xr}{xr + 1} \right)^2
        \, \mathrm{d}P(r)
    \]
    is strictly decreasing over $(0, \infty)$ with $h(0) = 1$ and $h(x_0) > 0$.
    This follows because for $x \in [0, x_0]$,
    \begin{equation}
        \begin{split}
            \label{eq:bound-deriv-v-in-lambda-part-2}
            \phi \int \left( \frac{xr}{xr + 1} \right)^2 \, \mathrm{d}P(r)
            &\le
            \left( \frac{x_0 b}{x_0 b + 1} \right)
            \phi 
            \int \left( \frac{xr}{xr + 1} \right) \, \mathrm{d}P(r) \\
            &<
            \phi \int \frac{xr}{xr + 1} \, \mathrm{d}P(r)
            \le
            \phi \int \frac{x_0 r}{x_0 r + 1} \, \mathrm{d}P(r)
            =
            1,
        \end{split}
    \end{equation}
    where the first inequality in the chain above follows as the support of $P$ is $[a, b]$, and the last inequality follows since $f(x_0) = 0$ and $x_0 > 0$, which implies that
    \[
        \frac{1}{x_0}
        = \phi \int \frac{r}{x_0 r + 1} \, \mathrm{d}P(r),
        \quad
        \text{ or equivalently that}
        \quad
        1
        = \phi \int \frac{x_0 r}{x_0 r + 1} \, \mathrm{d}P(r).
    \]
    Thus, $-f'$, a product of two positive strictly decreasing functions, is strictly decreasing, and in turn, $f'$ is strictly increasing.
    Moreover, $\lim_{x \to 0^{+}} f'(x) = -\infty$; when $\phi>1$, $f'(x_0) < 0$ and when $\phi\in(0,1]$, $f'(x)$ approaches zero from below as $x\rightarrow+\infty$.

    When $\phi>1$, because $f'(x)\neq$ over $(0, x_0)$, by the inverse function theorem, $(f^{-1})'$, we have
    \[
        \left| (f^{-1})'(f(x)) \right|
        = \left| \frac{1}{f'(x)} \right|
        <
        \left| \frac{1}{f'(x_0)} \right|
        =
        \ddfrac
        {1}
        {\frac{1}{x_0^2} \left( 1 - \phi \int \left( \frac{xr}{xr + 1} \right)^2 \, \mathrm{d}P(r) \right)}
        < \infty,
    \]
    where the first inequality uses the fact that $|f'(x_0)| < |f'(x)|$ for $x \in (0, x_0]$ from Part 1, and the last inequality uses the bound from \eqref{eq:bound-deriv-v-in-lambda-part-2}.

    \underline{Properties of $f''$ and $(f^{-1})''$}:
    
    The second derivative $f''$ at $x$
    is given by
    \[
        f''(x)
        = \frac{2}{x^3}
        - 2 \phi \int \frac{r^3}{(x r + 1)^3}
        \, \mathrm{d}P(r)
        =
        \frac{2}{x^3}
        \left(
            1 - \phi \int \left(\frac{xr}{xr + 1}\right)^3
            \, \mathrm{d}P(r)
        \right).
    \]
    The rest of the arguments are similar to those in Part 2.
    The function $g : x \mapsto 1/x^3$ is positive and strictly decreasing over $(0, \infty)$ with $\lim_{x \to 0^{+}} g(x) = \infty$ and $\lim_{x \to \infty} g(x) = 0$, while the function
    \[
        h:
        x \mapsto
        1 - \phi \int \left( \frac{xr}{xr + 1} \right)^3
        \, \mathrm{d}P(r)
    \]
    is strictly decreasing over $(0, \infty)$ with $h(0) = 1$ and $h(x_0) > 0$ as
    \begin{equation}
        \begin{split}
        \label{eq:bound-deriv-v-in-lambda-part-3}
            \phi \int \left( \frac{xr}{xr + 1} \right)^3 \, \mathrm{d}P(r)
            &\le
            \left( \frac{x_0 b}{x_0 b + 1} \right)^2
            \phi 
            \int \left( \frac{xr}{xr + 1} \right) \, \mathrm{d}P(r) \\
            &<
            \phi \int \frac{xr}{xr + 1} \, \mathrm{d}P(r)
            \le
            \phi \int \frac{x_0 r}{x_0 r + 1} \, \mathrm{d}P(r)
            =
            1.
        \end{split}
    \end{equation}
    It then follows that $f''$ is strictly decreasing, with $\lim_{x \to 0^{+}} f''(x) = \infty$; when $\phi>1$, $f''(x_0) > 0$ and when $\phi\in(0,1]$, $f''(x)$ approaches zero from above as $x\rightarrow+\infty$.

    When $\phi>1$, by inverse function theorem, we have
    \[
        \left|
            (f^{-1})''(f(x))
        \right|
        =
        \left|
            \frac{f''(x)}{f'(x)^3}
        \right|
        =
        \ddfrac
        {\frac{2}{x^3} \left( 1 - \phi \int \left( \frac{xr}{xr + 1} \right)^3 \, \mathrm{d}P(r)  \right)}
        {\frac{1}{x^6} \left( 1 - \phi \int \left( \frac{xr}{xr + 1} \right)^2 \, \mathrm{d}P(r) \right)^3}
        \le
        \ddfrac
        {2 x_0^3}
        {\left( 1 - \phi \int \left( \frac{xr}{xr + 1} \right)^2 \, \mathrm{d}P(r) \right)^3}
        < \infty,
    \]
    where the first inequality uses the bound from \eqref{eq:bound-deriv-v-in-lambda-part-3}, and the second inequality uses the bound from \eqref{eq:bound-deriv-v-in-lambda-part-2}.
    
    This finishes all the parts and concludes the proof.
\end{proof}

\section{Helper concentration results}\label{sec:appendix-concerntration}

\subsection{Size of the intersection of randomly sampled datasets}
\label{sec:size-intersection}

In this section, we collect various helper results concerned with concentrations and convergences.
Below we recall the definition of a hypergeometric random variable, along with its mean and variance.
See, e.g., \citet{greene2017exponential} for more related details.

    \begin{definition}[Hypergeometric random variable]\label{def:hypergeometric}
        A random variable $X$ follows the hypergeometric distribution $X\sim \operatorname {Hypergeometric} (n,K,N)$ if its probability mass function is given by
         $$\PP(X=k)=\frac{
         \binom{K}{k}\binom{N-K}{n-k}
         }{\binom{N}{n}
         },\quad \text{for} \quad \max\{0,n+K-N\}\leq k\leq \min\{n,K\}.$$
        The expectation and variance of $X$ are given by
        \begin{align*}
            \EE[X] &= \frac{nK}{N},
            \quad \text{and} \quad
            \Var(X) = \frac{nK(N-K)(N-n)}{N^2(N-1)}.
        \end{align*}
    \end{definition}

    The following lemma characterizes the limiting proportions of shared observations in two simple random samples under proportional asymptotics when both the subsample and full data sizes tend to infinity.
    The lemma is adapted from \citet{patil2022bagging}.
    \begin{lemma}[Asymptotic proportions of shared observations]\label{lem:i0_mean}
        For $n\in\NN$, define $\mathcal{I}_k := \{\{i_1, i_2, \ldots, i_k\}:\, 1\le i_1 < i_2 < \ldots < i_k \le n\}$.
        Let $I_1,I_2\overset{\textup{\texttt{SRSWR}}}{\sim}\cI_k$, define the random variable $i_{0}^{\textup{\texttt{SRSWR}}} :=|I_1\cap I_2|$ to be the number of shared samples, and define $i_{0}^{\textup{\texttt{SRSWOR}}}$ accordingly.
        Let $\{k_m\}_{m=1}^{\infty}$ and $\{n_m\}_{m=1}^{\infty}$ be two sequences of positive integers such that $n_m$ is strictly increasing in $m$, $n_m^{\nu}\leq k_m\leq n_m$ for some constant $\nu\in(0,1)$.
         Then, $i_0^{\textup{\texttt{SRSWR}}}/k_m-k_m/n_m\asto 0$, and $i_0^{\textup{\texttt{SRSWOR}}}/k_m-k_m/n_m\asto 0$.
    \end{lemma}

\subsection{Convergence of random linear and quadratic forms}
\label{sec:concen-linform-quadfrom}

In this section, we collect helper lemmas on the concentration of linear and quadratic forms of random vectors.

The following lemma provides the concentration of a linear form of a random vector with independent components.
It follows from a moment bound from Lemma 7.8 of \citet{erdos_yau_2017}, along with the Borel-Cantelli lemma.
It is adapted from Lemma S.8.5 of \citet{patil2022mitigating}.

\begin{lemma}
    [Concentration of linear form with independent components]
    \label{lem:concen-linform}
    Let $\bz_p \in \RR^{p}$ be a sequence of random vector with i.i.d.\ entries $z_{pi}$ for $i \in [p]$ such that for each i, $\EE[z_{pi}] = 0$, $\EE[z_{pi}^2] = 1$, $\EE[|z_{pi}|^{4+\alpha}] \le M_\alpha$ for some $\alpha > 0$ and constant $M_\alpha < \infty$.
    Let $\ba_p \in \RR^{p}$ be a sequence of random vectors independent of $\bz_p$ such that $\limsup_{p} \| \ba_p \|^2 / p \le M_0$ almost surely for a constant $M_0 < \infty$.
    Then, we have $\ba_p^\top \bz_p / p \to 0$ almost surely as $p \to \infty$.
\end{lemma}

The following lemma provides the concentration of a quadratic form of a random vector with independent components.
It follows from a moment bound from Lemma B.26 of \citet{bai2010spectral}, along with the Borel-Cantelli lemma.
It is adapted from Lemma S.8.6 of \citet{patil2022mitigating}.

\begin{lemma}
    [Concentration of quadratic form with independent components]
    \label{lem:concen-quadform}
    Let $\bz_p \in \RR^{p}$ be a sequence of random vector with i.i.d.\ entries $z_{pi}$ for $i \in [p]$ such that for each i, $\EE[z_{pi}] = 0$, $\EE[z_{pi}^2] = 1$, $\EE[|z_{pi}|^{4+\alpha}] \le M_\alpha$ for some $\alpha > 0$ and constant $M_\alpha < \infty$.
    Let $\bD_p \in \RR^{p \times p}$ be a sequence of random matrix such that $\limsup \| \bD_p \|_{\oper} \le M_0$ almost surely as $p \to \infty$ for some constant $M_0 < \infty$.
    Then, we have $\bz_p^\top \bD_p \bz_p / p - \tr[\bD_p] / p \to 0$
    almost surely as $p \to \infty$.
\end{lemma}

\subsection{Convergence of Ce\`saro-type mean and max for triangular array}
\label{sec:cesaro-mean-max}

In this section, we collect a helper lemma on deducing almost sure convergence of a Ce\`saro-type mean from almost sure convergence of the original sequence.
It is adapted from \citet{patil2022bagging}.

    \begin{lemma}[Convergence of conditional expectation]\label{lem:conv_cond_expectation}
        For $n\in\NN$, suppose $\{R_{n,\ell}\}_{\ell=1}^{N_n}$ is a set of $N_n$ %
        random variables defined over the probability space $(\Omega,\cF,\PP)$, with $1<N_n<\infty$ almost surely.
        If there exists a constant $c$ such that $R_{n,p_n}\asto c$ for all deterministic sequences $\{p_n\in[N_n]\}_{n=1}^{\infty}$, then the following statements hold:
        \begin{enumerate}[(1)]
            \item\label{lem:conv_cond_expectation-max}
            $\max_{\ell\in[N_n]}\left|R_{n,\ell}(\omega)-c\right|\asto 0$.
            \item\label{lem:conv_cond_expectation-avg} $N_n^{-1}\sum_{\ell=1}^{N_n}R_{n,\ell}\asto c$.
        \end{enumerate}
    \end{lemma}

    \begin{lemma}[Convergence of conditional expectation over simple random sampling]\label{lem:conv_cond_expectation_sampling}
        For $n\in\NN$ and $k=k_n\in\cK_n$, let $M_n=|\cI_k|$ and suppose 
        $\{R_{n,1}(I_{\ell})\}_{\ell\in[M_n]}$ and 
        $\{R_{n,2}(I_{m},I_{\ell})\}_{m,\ell\in[M_n],m\neq \ell}$ are sets of $M_n$ and $M_n(M_n-1)$ random variables, such that $R_{n,2}(I_{m},I_{\ell})\leq (R_{n,1}(I_{m})+R_{n,2}(I_{\ell}))/2$. Then the following statements hold:
        \begin{enumerate}[(1)]
            \item\label{lem:conv_cond_expectation_sampling-1}
            If there exists a constant $c_1$ such that $R_{n,1}(I_{\ell_n})\asto c_1$ for all deterministic sequences $\{\ell_n\in[M_n]\}_{n=1}^{\infty}$, then 
            $\max_{\ell\in[M_n]}\left|R_{n,\ell}(I_{\ell})-c\right|\asto 0$ and $\mathbb{E}_{I_{\ell}\overset{\SRS}{\sim}\cI_k}[\left|R_{n,\ell}(I_{\ell})-c\right|]\asto 0$.
            
            \item\label{lem:conv_cond_expectation_sampling-2} Further, if there exists a constant $c_2$ such that $R_{n,2}(I_{m_n}, I_{\ell_n})\asto c_2$ for all sequences of simple random samples $\{(I_{m_n},I_{\ell_n})\overset{\SRS}{\sim}\cI_{k_n}\}_{n=1}^{\infty}$, then 
            $\max_{(I_m,I_{\ell})\overset{\SRS}{\sim}\cI_k}\left|R_{n,2}(I_{m}, I_{\ell})-c_2\right|\asto 0$ and $\mathbb{E}_{(I_m,I_{\ell})\overset{\SRS}{\sim}\cI_k}[\left|R_{n,2}(I_{m_n}, I_{\ell_n})-c_2\right|]\asto 0$.
        \end{enumerate}
    \end{lemma}
    \begin{proof}[Proof of \Cref{lem:conv_cond_expectation_sampling}]
        We split the proof into two cases.

        \paragraph{Part (1)} The conclusion directly follows from \Cref{lem:conv_cond_expectation}.

        \paragraph{Part (2)}
        Observe that
        \begin{equation}\label{eq:Inequality-relating-M=2-and-M=1}
        R_{n,2}(I_{m},I_{\ell})\leq \frac{1}{2}(R_{n,1}(I_{m})+R_{n,2}(I_{\ell})).
        \end{equation}
        From (1), we have that $\EE_I[R_{n,1}(I)]\asto c_1$, where the expectation is taken with respect to the uniform distribution over $\cI_k$.
        From the condition, we have $R_{n,2}(I_{m},I_{\ell})\asto c_2$ for any $I_m,I_{\ell}\overset{\SRS}{\sim}\cI_k$.
        Then, by Pratt's lemma \citep[see, e.g.,][Theorem  5.5]{gut_2005}, the conclusion follows.    
    \end{proof}

\section[GCV correction for arbitrary M]{GCV correction for arbitrary $M$}\label{app:correct}
    Note that the asymptotic limit of the training error for arbitrary $M\in\NN$ is given by
    \begin{align*}
        \sT_{M}^{\lambda} &= 2 \sE_{k,2}^{\lambda} - \sE_{k,1}^{\lambda} + \frac{2}{M}(\sE_{k,1}^{\lambda} - \sE_{k,2}^{\lambda}),
    \end{align*}
    where $\sE_{k,j}^{\lambda} = c_{k,M,j} \sT_{k,j}^{\lambda} + (1-c_{k,M,j}) \sR_{k,j}^{\lambda}$.
    Here, $c_{k,M,j}$ is the limiting proportion of the distinct number of observations from $j$ simple random samples to the distinct number of observations from $M$ simple random samples of size $k$.
    Roughly speaking, the proportion of unseen observations from $M$ simple random samples of size $k$ is $(n-k)^M/n^M$ and thus
    \begin{align*}
        c_{k,M,j} & = \lim \frac{1 - (n-k)^j/n^j}{1 - (n-k)^M/n^M} = \frac{1 - (1 - \phi/\phi_s)^j}{1 - (1 - \phi/\phi_s)^M}.
    \end{align*}
    From the expression, one knows that the GCV asymptotics will not match the risk of the estimator in general. 
    In addition, the form of the expression also leads to an approach to correct the GCV estimator for general $M$ that we will discuss below. 
    We prove in \Cref{thm:uniform-consistency-k} that the difference between the two asymptotics vanishes as $M\rightarrow\infty$. We expect the difference to scale as $1/M$. 
    The explicit analysis of the finite-ensemble effect requires carefully analyzing the coefficients $c_{k,M,j}$, and even for the isotropic design, the expression for the GCV asymptotics for general appears to be very involved.
    It is, in principle, possible to perform this analysis, but we did not pursue it further in the paper, given our primary focus on the full-ensemble estimator. Numerically, we observe that the bias is small for a moderate $M$ (e.g., for $M=10$) and a reasonable data model with SNR (SNR = 0.6) from \Cref{fig:gcv-M}. Generally, we expect this to be the case for either moderate $k$ or $M$ and typical real-world SNR ranges. We will consider adding more numerical illustrations of the finite-ensemble effect in the revision under different settings.

    We aim to define the corrected GCV as
    \begin{align*}
        \overline{\gcv}_{k,M}^{\lambda} :=\frac{a_1 T_{k,M}^{\lambda}  + a_2 \bar{R}_{k,M}^{\lambda} }{D_{k,M}^{\lambda}},
    \end{align*}
    where $a_1$ and $a_2$ are two unknown parameters to be determined.
    We must match the limiting GCV with the true risk to determine the unknown parameters.    
    Since $$\RlamMtr[\lambda][1]= \Ddet[1] \RlamM[\lambda][1],\quad \text{and} \quad \RlamMtr[\lambda][2]= b_1 \RlamM[\lambda][1] + b_2 \RlamM[\lambda][2],$$
    for some known constants $b_1$ and $b_2$ which can be derived in the proof of \Cref{prop:inconsistency}, the adjustment is given by
    \begin{align*}
        &a_1\Bigg[-\left(1 - \frac{2}{M}\right)(c_{k,M,1} \Ddet[1] + 1 - c_{k,M,1}) \RlamM[\lambda][1] \\
        &\quad + 2 \left(1 + \frac{1}{M}\right)(c_{k,M,2} b_1\RlamM[\lambda][1] + (1 - c_{k,M,2} + b_2)\RlamM[\lambda][2]) \Bigg]\\
        &\qquad + \Ddet[M] a_2\left(-\left(1 - \frac{2}{M}\right) \RlamM[\lambda][1] + 2 \left(1 + \frac{1}{M}\right) \RlamM[\lambda][2]\right)\\
        &\quad =-\left(1 - \frac{2}{M}\right)[a_1(c_{k,M,1} \Ddet[1] + 1 - c_{k,M,1}) + a_2 \Ddet[M]] \RlamM[\lambda][1] \\
        &\qquad +2 \left(1 + \frac{1}{M}\right)[a_1(c_{k,M,2} b_1\RlamM[\lambda][1] + (1 - c_{k,M,2} + b_2)\RlamM[\lambda][2]) + a_2 \Ddet[M]]\RlamM[\lambda][2]), 
    \end{align*}
    which implies that
    \begin{align*}
        a_1(c_{k,M,1} \Ddet[1] + 1 - c_{k,M,1}) + a_2 \Ddet[M] &=\Ddet[M] \\
        a_1(c_{k,M,2} b_1\RlamM[\lambda][1] + (1 - c_{k,M,2} + b_2)\RlamM[\lambda][2]) + a_2 \Ddet[M]\RlamM[\lambda][2]&=\Ddet[M]\RlamM[\lambda][2].
    \end{align*}
    Solving the above linear system for $a_1>0$ and $a_2\in\RR$ gives the correct weights for defining a consistent GCV estimate.
    The solutions will depend on $\Ddet[M]$, $\RlamM[\lambda][1]$ and $\RlamM[\lambda][2]$.
    For the denominator, \eqref{eq:D-kM} is a consistent estimate for $\Ddet[M]$.
    For the prediction risks of $M=1,2$, we can use out-of-bag observations to estimate $\RlamM[\lambda][1]$ and $\RlamM[\lambda][2]$.

\section{Additional details for numerical experiments}\label{app:numerical-details}
    The covariance matrix of an auto-regressive process of order 1 (AR(1)) is given by $\bSigma_{\mathrm{ar1}}$, where $(\bSigma_{\mathrm{ar1}})_{ij} = \rhoar^{|i-j|}$ for some parameter $\rhoar\in(0,1)$, and the AR(1) data model is defined as:
        \begin{align}\tag{M-AR1}\label{eq:model-ar1}
            \begin{split}
                y_i &= \bx_i^{\top}\bbeta_0 + \epsilon_i,
        \quad \bx_i\sim\cN(0,\bSigma_{\mathrm{ar1}}),\\
             \bbeta_0&=\frac{1}{5}\sum_{j=1}^5 \bw_{(j)},
        \quad  \epsilon_i\sim \cN(0,\sigma^2), 
            \end{split}        
        \end{align}
        where $\bw_{(j)}$ is the eigenvector of $\bSigma_{\mathrm{ar1}}$ associated with the top $j$th eigenvalue $r_{(j)}$.
        From \citet[pp. 69-70]{grenander1958toeplitz}, the top $j$-th eigenvalue can be written as $r_{(j)}=(1-\rhoar^2)/(1-2\rhoar\cos \theta_{jp}+\rhoar^2)$ for some $\theta_{jp}\in((j-1)\pi/(p+1), j\pi/(p+1))$.
        Then, under the model \eqref{eq:model-ar1}, the signal strength $\rho^2$ defined in \Cref{asm:lin-mod} is $5^{-1} (1-\rhoar^2)/(1-\rhoar)^2$, 
        which is the limit of $25^{-1}\sum_{j=1}^5 r_{(j)} $.
        Thus, the model \eqref{eq:model-ar1} parameterized by two parameters $\rhoar$ and $\sigma^2$ satisfies Assumption \ref{asm:rmt-feat}-\ref{asm:lin-mod}.

\end{document}